\documentclass[a4paper,12pt,leqno]{amsart} % <-- original (changed for duke)
\usepackage[a4paper, top=2.4cm, bottom=2cm, left=3cm, right=3cm]{geometry}

\usepackage{amsmath,amssymb,latexsym,amsthm,tikz, enumerate,enumitem,amscd,hyperref}%,enumitem} 
\usepackage{graphicx,tikz}

\usetikzlibrary{positioning,arrows}
\usepackage[utf8]{inputenc} 
\usepackage[T1]{fontenc}
\usepackage{hyperref}
\usepackage{theoremref}
\usepackage[textsize=tiny]{todonotes}

\numberwithin{equation}{section}
\theoremstyle{plain}
\newtheorem{lem}[equation]{Lemma}

\newtheorem{prop}[equation]{Proposition}
\newtheorem{thm}[equation]{Theorem}
\newtheorem{assumption}[equation]{Assumption}
\newtheorem{cor}[equation]{Corollary}

\newtheorem{conj}[equation]{Conjecture}
\theoremstyle{definition}
\newtheorem{definition}[equation]{Definition}

\newtheorem{remark}[equation]{Remark}

\newtheorem{claim}{Claim}

\newtheorem*{mthm*}{Main Theorem}
\newtheorem*{claim*}{Claim}
\newtheorem*{que*}{Question}
\newtheorem*{remark*}{Remark}
\newtheorem*{definition*}{Definition}

\newcommand{\type}{\operatorname{Type}}

\newcommand{\bC}{\mathfrak C}

\newcommand{\Y}{\Upsilon}

\newcommand{\lk}{\operatorname{lk}}

\newcommand{\wtC}{\widetilde C}

\newcommand{\Sh}{\mathfrak{S}}

\newcommand{\N}{\mathbb{N}}
\newcommand{\HB}{B^{\mathrm h}}
\DeclareMathOperator{\Res}{Res}

\title[From Trees to Tripods]{From Trees to Tripods: Proof of $K(\pi,1)$ for Artin groups with $ABI$-type spherical parabolics}

\author{Nima Hoda}
\address{Instytut Matematyczny,
Uniwersytet Wroc\l awski,
pl.\ Grun\-wal\-dzki 2,
50--384 Wroc{\l}aw, Poland}
\email{nima@nimahoda.net}
\address{Department of Mathematics, Tufts University, Medford, MA 02155, USA}
\email{nima@nimahoda.net}

\author{Jingyin Huang}
\address{Department of Mathematics, The Ohio State University, 100 Math Tower, 231 W 18th Ave, Columbus, OH 43210, U.S.}
\email{huang.929@osu.edu}

\begin{document}

\begin{abstract}
    We describe contractibility criteria for typed simplicial complexes modeled on Coxeter diagrams.  Such complexes include buildings and also arise as coset complexes of groups with strongly admissible parabolic systems, including Shephard groups, which generalize both Artin and Coxeter groups.  Our result applies to a large class of complexes modeled on \emph{tree Coxeter diagrams} and implies deformation retraction onto a nonpositively curved subcomplex.
    
    As an application, we reduce the $K(\pi,1)$-conjecture for all Artin groups with tree Coxeter diagrams to properties of Artin groups with tripod-shaped Coxeter diagrams. Combining this reduction theorem and properties of braid groups in previous works of Charney, Crisp-McCammond, Haettel and the second named author, we deduce that the $K(\pi,1)$-conjecture holds for every Artin group whose spherical parabolic subgroups avoid type $D_n$ ($n \ge 4$) and the exceptional types.  We give an analogous tree to tripod reduction for Shephard groups.

    Our proof relies on two ingredients of independent interest: a notion of combinatorial convexity and a Bestvina-type inequality in certain injective orthoscheme complexes. These ingredients further rely on the use of structural properties of bi-Helly graphs developed in joint work of the first named author with Munro.
    
	% The reduction theorem relies on producing a ``tower'' of injective metric spaces from a single Artin group. The construction of such a tower relies on two ingredients of independent interests: a notion of combinatorial convexity and a Bestvina-type inequality, in certain injective orthoscheme complexes. These ingredients further rely on the use of structural properties of bi-Helly graphs (also known as absolute bipartite retracts) developed in joint work of the first named author with Munro.
    
	%The proof is based on constructing actions of Artin groups on injective orthoscheme complexes. We introduce a notion of combinatorial convexity for subcomplexes of such complexes and establish a Bestvina-type inequality. A further ingredient is the use of structural properties of bi-Helly graphs (also known as absolute bipartite retracts) developed in joint work of the first named author with Munro.
\end{abstract}

\maketitle

\section{Introduction}
We develop a local to global mechanism for constructing contractible simplicial complexes and group actions with features of non-positive curvature, with Artin groups and
Shephard groups as principal applications.

Artin groups form a large and natural class of groups arising in geometric group theory, low-dimensional topology, hyperplane arrangement and singularity theory. Shephard groups simultaneously generalize Artin
and Coxeter groups, and they arise in the study of  regular complex
polytopes,  orbifold braid groups, complex braid groups, moduli spaces of line arrangements in the
complex projective plane etc.

One of the fundamental questions about Artin groups is the $K(\pi,1)$-conjecture, originating in work of Arnol'd, Brieskorn, Pham, and Thom, which predicts that a canonical complex manifold associated to each Artin group is aspherical. An influential geometric approach to the $K(\pi,1)$-conjecture was proposed by Charney and Davis in the 1990s. The guiding principle is that Artin groups are governed by their irreducible spherical standard parabolic subgroups, which control the local geometry of the complexes on which Artin groups act. Such spherical subgroups are completely classified: they consist of four infinite families $A_n$, $B_n$, $D_n$, and $I_2(n)$, together with six exceptional groups. As an important instance, Charney and Davis proposed the following deep conjecture on braid groups and showed that it would imply the $K(\pi,1)$-conjecture for a substantial portion of Artin groups.

\begin{conj}[Charney--Davis]
	\label{conj:braid}
	The spherical Deligne complex of any braid group is CAT$(1)$.
\end{conj}

\begin{thm}[{\cite{CharneyDavis,charney2004deligne}}]
	\label{thm:connection}
	If Conjecture~\ref{conj:braid} holds, then the $K(\pi,1)$-conjecture holds for any Artin group whose every irreducible spherical parabolic subgroup is of type $A$, $B$, or $I_2$.
\end{thm}

Here type $A$, $B$ or $I_2$ means type $A_n$, $B_n$ or $I_2(n)$, for some $n$, respectively. 

At present, however, Conjecture~\ref{conj:braid} is known only for braid groups on at most four strands \cite{charney2004deligne}. Establishing the CAT$(1)$ property requires ruling out closed
local geodesics of length $<2\pi$. The interaction between short loops and the
simplicial structure becomes increasingly complicated in higher
dimensions. Even for loops contained in the $1$-skeleton, there
is no bound independent of the dimension on the number of edges in
a loop of length $<2\pi$ (due to the existence of shorter and shorter edges when the dimension increases).

One of the main observations of this paper is that the full
strength of Conjecture~\ref{conj:braid} is not required for the $K(\pi,1)$-conjecture. A substantially weaker condition,
involving only specific $4$-cycles and $6$-cycles in the
$1$-skeleton, already suffices. Crucially, the same bound of six
edges applies in each dimension. These conditions hold for braid
groups by previous work of Charney, Crisp--McCammond,
Haettel, and the second named author
\cite{charney2004deligne,haettel2021lattices,huang2024}. Our method also reveals the geometry by exhibiting a hierarchy of actions on
``non-positively curved'' spaces, with smaller Artin groups
appearing recursively as stabilizers.

\begin{thm}(Theorem~\ref{thm:ABI K(pi,1)})	\label{thm:ABI}
	The $K(\pi,1)$-conjecture holds for any Artin group whose every irreducible spherical parabolic subgroup is of type $A$, $B$, or $I_2$.
	
	Moreover, these Artin groups acts cocompactly by isometries on CAT$(0)$ spaces or injective metric spaces or Bestvina complexes with stabilizers being Artin groups with strictly less generators. 
\end{thm}

Equivalently, the $K(\pi,1)$-conjecture holds for all Artin groups whose spherical parabolic subgroups avoid type $D_n$ ($n\ge 4$) and the exceptional types.

A geodesic metric space is \emph{injective}, if any collection of pairwise intersecting metric balls have a nonempty common intersection. Injective metric spaces can be viewed as nonlinear version of $\ell^\infty$-spaces, while CAT$(0)$ spaces can be viewed as nonlinear version of $\ell^2$-spaces. Bestvina complexes can be viewed, roughly,
as ``central quotients'' of Garside groupoids
\cite{Bestvina1999,bessis2006garside}. They exhibit a combinatorial form
of non-positive curvature, with exceptionally nice properties that, in certain
aspects, go beyond those guaranteed by the CAT$(0)$ condition.

 More broadly, Theorem~\ref{thm:ABI} is an
application of a general geometric reduction theorem, which
constitutes the main contribution of this paper. Formulated in the general setting of typed simplicial
complexes, this theorem suggests a route to analogous $K(\pi,1)$ and geometric
results for the remaining Artin groups and extends the same strategy
to Shephard groups and beyond. 

In Section~\ref{subsec:main theorem}, we discuss the main theorem. It relies on two geometric ingredients: a notion of combinatorial convexity and a Bestvina-type inequality, in certain injective orthoscheme complexes, discussed in Section~\ref{subsec:proof sketch}.

\subsection{The main theorem}
\label{subsec:main theorem}
Informally, our main theorem asserts that a connected, simply connected
simplicial complex $X$ in a suitable class is contractible if sufficiently ``deep'' links of simplices satisfy simple combinatorial conditions
on $4$-cycles and $6$-cycles. Moreover, $X$ deformation retracts onto a ``core'' exhibiting features of (combinatorial) non-positive curvature. To motivate the precise statement, we first discuss the case of Artin groups and Shephard groups.

In a previous article \cite{huangbestvina}, the second named author reduced the $K(\pi,1)$-conjecture for all Artin groups to properties for Artin groups whose Coxeter diagrams are trees.
In this article, we reduce the verification $K(\pi,1)$-conjecture for Artin groups with tree Coxeter diagrams to understanding the following three families of Artin groups, whose Coxeter diagrams generalize the classical types $E$, $F$, and $H$.

\begin{definition}
	\label{def:special}
	See Figure~\ref{fig:special} for the following three families.
	A Coxeter diagram $\Lambda$ is of \emph{type $F_{r,s}$} if it is a linear graph with exactly one edge labeled $4$ and all other edges labeled $3$, where $r$ (resp.\ $s$) is the number of edges on one (resp.\ the other) side of the edge labeled $4$. For example, $F_{1,1}$ is the Coxeter diagram of type $F_4$, and $F_{1,2}=\widetilde F_4$.
	A Coxeter diagram $\Lambda$ is of \emph{type $H_{r,s}$} if it is a linear graph with exactly one edge labeled $5$ and all other edges labeled $3$, with $r$ and $s$ defined similarly. For example, $H_{1,0}=H_3$ and $H_{2,0}=H_4$.
	A Coxeter diagram $\Lambda$ is of \emph{type $E_{r,s,t}$} if it is a tripod with all edges labeled $3$, where $r,s,t$ denote the number of edges in the three arms. For example, $E_{2,2,1}=E_6$ and $E_{2,1,1}=D_5$.
\end{definition}

\begin{figure}[h]
	\centering
	\includegraphics[scale=0.78]{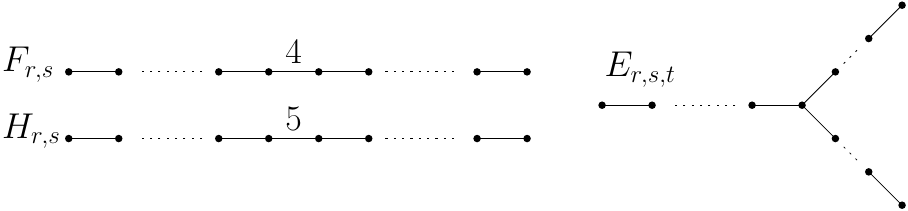}
	\caption{Three special families.}
	\label{fig:special}
\end{figure}

Let $A_\Lambda$ be an Artin group with standard generating set $S$ and Coxeter diagram $\Lambda$. The \emph{Artin complex} of $A_\Lambda$, denoted by $\Delta_S$ or $\Delta_\Lambda$, (\cite{CharneyDavis,paris2012k,cumplido2020parabolic}) is the simplicial complex whose vertices are in bijection with left cosets of the form $\{gA_{S\setminus\{s\}}\}_{g\in A_\Lambda,s\in S}$, where $A_{S\setminus\{s\}}$ is the subgroup of $A_\Lambda$ generated by $S\setminus\{s\}$. A collection of vertices spans a simplex if and only if the corresponding cosets have nonempty intersection. 	A vertex of $\Delta_S$ has \emph{type $\hat s=S\setminus\{s\}$} if it corresponds to a coset of the form $gA_{S\setminus\{s\}}$. The article \cite{paris2012k} relates $K(\pi,1)$-conjecture and the contractibility of Artin complexes, so the focus here is to establish the contractibility of $\Delta_\Lambda$, assuming $\Lambda$ is any tree Coxeter diagram which is not spherical.

The link of any simplex $\sigma$ in $\Delta_\Lambda$ is an Artin complex $\Delta_{\Lambda'}$ where $\Lambda'$ is an induced subdiagram of $\Lambda$. Bigger $\sigma$ corresponds to smaller $\Lambda'$. A \emph{deep link} of $\Delta_S$ refers to a link of a simplex $\sigma$ of codimension at least $3$ such that  $\Lambda'$ belongs to one of the diagrams in Figure~\ref{fig:special}. 

We use two short cycle conditions in deep links. A \emph{special $4$-cycle} in $\Delta_{\Lambda'}$ is an embedded
$4$-cycle with alternating
types $\hat s\hat t\hat s\hat t$; its \emph{center} is a vertex
adjacent to all four vertices. The relevant $6$-cycles in $\Delta_{\Lambda'}$ have types
$
\hat s\hat t_1\hat s\hat t_2\hat s\hat t_3,
$
where $s$ is a distinguished base of a $B_n$-like or $D_n$ 
subdiagram of $\Lambda'$ and the $t_i$ are leaves of that subdiagram. A $B_n$-like diagram
is a path with a terminal edge labeled at least $4$, with base at the opposite
end; for $n=2$, either endpoint may be the base. A $D_n$-diagram of $\Lambda'$ is a copy of the Coxeter diagram of type $D_n$ (with all edge label $3$), with base a leaf farthest from the branch vertex;
for $D_3$, the base is the middle of the three-vertex path. The $t_i$ may coincide. Such an embedded
$6$-cycle is \emph{special}, and a \emph{quasi-center} is a vertex adjacent to
its three vertices of type $\hat s$.

Special cases of these 4-cycles and 6-cycles and their close relatives appeared in previous works \cite{elder2002curvature,charney2004deligne,brady2010braids,haettel20166,weaklymodular,hirai2020uniform,haettel2021lattices,haettel2022link}, in connection to non-positive curvature aspects of polyhedral complexes. 

An \emph{extended Coxeter diagram} $\Lambda$ is a Coxeter diagram with each vertex labeled by an element in $\{2,3,4,\ldots\}\cup \{\infty\}$ such that any two adjacent vertices connected by an odd edge in the diagram share the same vertex label. Let $\mathfrak{S}_\Lambda$ be the associated \emph{Shephard group}, which is the quotient $
\mathfrak{S}_\Lambda=A_\Lambda/\langle\!\langle s_i^{p_i}\rangle\!\rangle_{i\in I}
$
where $\{s_i\}_{i\in I}$ denotes the collection of vertices of $\Lambda$ whose label $p_i$ is finite. We say $\Sh_\Lambda$ is an \emph{even order} Shephard group if all the finite $p_i$'s are even. We can define the complex $\Delta_\Lambda$, its deep links and special 4- and 6- cycles for Shephard groups in a similar (without regard to the vertex labels).
\begin{thm}(=Theorem~\ref{thm:s})
	\label{thm:intro shephard}
	Let $\Sh_\Lambda$ be an even order Shephard group such that $\Lambda$ is a tree not in Definition~\ref{def:special}.
	Suppose that in any deep link of $\Delta_\Lambda$, any special 4-cycle has a center and any special 6-cycle has a quasi-center.
	
	Then $\Delta_\Lambda$ is contractible, and there is a $\Sh_\Lambda$-invariant and $\Sh_\Lambda$-cocompact subcomplex of $\Delta_\Lambda$ admitting a $\Sh_\Lambda$-invariant metric which is either injective or CAT$(0)$.
\end{thm}
These injective metric spaces or CAT$(0)$ spaces arise from metrizing relative Artin complexes introduced in \cite{huang2023labeled} and their Shephard analogues.
We expect the ``even order'' assumption in the theorem can be dropped. 

This theorem immediately implies that the following ``tree to tripod'' reduction theorem for the $K(\pi,1)$-conjecture for Artin groups, which can then be combined with previous works \cite{huangbestvina,huang2024,huang2021helly} to conclude Theorem~\ref{thm:ABI}.

\begin{thm}(=Theorem~\ref{thmb:reduction single})
	\label{thm:reduction single}
	Let $\Lambda$ be a tree Coxeter diagram such that the following holds for any induced subdiagram $\Lambda'\subset \Lambda$ in the families $F_{r,s}$, $H_{r,s}$, and $E_{r,s,t}$:
	\begin{enumerate}
		\item the Artin group $A_{\Lambda'}$ satisfies the $K(\pi,1)$-conjecture;
		\item any special 4-cycles in the Artin complex $\Delta_{\Lambda'}$ has a center, and any special 6-cycle in $\Delta_{\Lambda'}$ has a quasi-center.
	\end{enumerate}
	The $A_\Lambda$ satisfies the $K(\pi,1)$-conjecture.
\end{thm}

The analogue for Shephard groups of this $K(\pi,1)$-conjecture reduction result will appear in a separate article, using Theorem~\ref{thm:intro shephard} as a main ingredient. 

\smallskip

Our main theorem applies to a class of complexes substantially broader
than those in Theorem~\ref{thm:intro shephard}. A simplicial complex $X$
of \emph{type $S$} is a flag complex whose vertices are labeled by
$\{\hat s:s\in S\}$, with each maximal simplex containing exactly one
vertex of each type.
Let $\Lambda$ be a Coxeter diagram with vertex set $S$. Informally,
$X$ is \emph{modeled on $\Lambda$} (Definition~\ref{def:model}) if the diagram determines the local structure of $X$ in two ways: (1) the link of a simplex $\sigma$ splits as a join according to the connected components of the induced subdiagram $\Lambda'$ obtained from $\Lambda$ by deleting the types of the vertices of $\sigma$, and (2) links of codimension two simplices satisfy girth bounds prescribed by the edge labels of $\Lambda$. This notion is inspired by the local structure of buildings and Coxeter complexes and also encompasses the Artin and Shephard complexes considered above. Then the link of each simplex is again a simplicial complex modeled on $\Lambda'$. So deep links and special cycles can be defined as before.

\begin{mthm*}(Theorem~\ref{thm:geometry})
	\label{thm:geometry intro}
	Suppose $\Lambda$ is tree Coxeter diagram which does not belong to Definition~\ref{def:special} with $|V\Lambda|\ge 3$.
Let $\Delta_S$ be a nondegenerate, regular, simplicial complex of type $S$ modeled on $\Lambda$ satisfying condition C$_5$.
Suppose in each deep link of $\Delta_S$, any special 4-cycle has a center; and any special 6-cycle  has a quasi-center. 
Then $\Delta_S$ is contractible, and it deformation retracts onto a subcomplex which admits either an injective metric or a CAT$(0)$ metric. 
\end{mthm*}

Sometimes the dimension this subcomplex can be much smaller than the dimension of $\Delta_S$, see Remark~\ref{rmk:dimension}.

Regularity (Definition~\ref{def:regular}) is a global property. It is the correct analogue of connected and simply-connected in the setting of simplicial complexes of type $S$.
Nondegeneracy (Definition~\ref{def:nondeg}) prevents links from being too small; in particular,
the link of every codimension one simplex contains $\ge 2$
vertices.  Condition C$_5$ (Definition~\ref{def:Ck}) is a local compatibility condition on generalized
$5$-cycles of simplices.

The main theorem applies to the general setting of groups with strongly admissible parabolic system (Definition~\ref{def:admissible parabolic system}) beyond Shephard groups: see Theorem~\ref{thm:group}.

\begin{remark}
	Before this work, the $K(\pi,1)$-conjecture had been established for
	several important classes of Artin groups, including those of spherical,
	FC, and affine type, as well as three-dimensional Artin groups
	\cite{deligne,CharneyDavis,charney2004deligne,callegaro2010k,
		paolini2021proof,juhasz2023class,goldman2022k,haettel2023new,
		huang2023labeled,huang2024,goldman2025deligne,huang2025353}.
	Nevertheless, broad classes of cases in dimensions at least four
	remained open. This article establishes the conjecture for a
	substantial family of such groups.
	
	The first version of this article was posted on arXiv in February
	2026. On October 6, 2026, OpenAI released a preprint \cite{OAI26}
	announcing a proof of the $K(\pi,1)$-conjecture for all Artin groups, hence superseding the $K(\pi,1)$ conclusion of
	Theorem~\ref{thm:ABI}. However, the geometric conclusions of Theorem~\ref{thm:ABI} go beyond
	asphericity. Our Main Theorem also
	extends far beyond Artin groups, notably to Shephard groups,
which lie outside the scope of \cite{OAI26}. 

The two approaches are strikingly different: geometric
representation theory in \cite{OAI26} and geometric group
theory in our work, which suggests that deeper connections
between these fields may remain to be discovered.
\end{remark}

\subsection{Discussion of proof}
\label{subsec:proof sketch}
We only discuss the proof of the Artin group case of Theorem~\ref{thm:intro shephard}, in order to avoid some technical details and keep the main idea clear.

\smallskip
\noindent
\textbf{Overview}\ \ \ 
Injective metric spaces are contractible and provide natural building
blocks for geometric models of classifying spaces. A central aim of
this article is to organize these geometric models into a systematic
inductive hierarchy.

Starting from a single Artin group, we construct a \emph{tower of injective metric spaces}. More precisely, let \(A_S\) be an Artin group whose Coxeter diagram is a tree. We construct a cocompact isometric action of \(A_S\) on an injective metric space \(X_S\) with the following recursive feature: every point stabilizer is itself an Artin group with strictly fewer generators, and these stabilizers admit cocompact isometric actions on injective metric spaces whose stabilizers are smaller still. This terminates at the class of Artin groups specified in Definition~\ref{def:special}, which form level \(0\) of the tower.

The principal new ingredient is a \emph{propagation theorem} asserting that injective actions ascend the tower: if every Artin group at level \(k\) admits a cocompact isometric action on an injective metric space, then so does every Artin group at level \(k+1\). This propagation mechanism is the key to construct such a tower. Once the tower is constructed, it is not hard to use the contractibility of injective metric spaces to deduce the desired results, so we will focus on how to construct this tower.

\medskip
\noindent
\textbf{Piecewise $\ell^\infty$-metric on relative Artin complexes}\ \ \ 
An $n$-dimensional \emph{unit orthoscheme}  is the convex hull of
$$
v_0=(0,0,\ldots,0),v_1=(1,0,\ldots,0),v_2=(1,1,\ldots,0),\cdots,v_n=(1,1,\ldots,1)
$$
in $\mathbb R^n$. We endow the unit orthoscheme with the $\ell^\infty$-metric. A \emph{piecewise $\ell^\infty$ orthoscheme complex} is a simplicial complex obtained by gluing a collection unit orthoscheme with $\ell^\infty$-metric (possibly of different dimensions) isometrically along their faces. Such complex has a natural gluing metric, and it is an \emph{injective orthoscheme complex} if this metric gives an injective metric space. 

Let $W_S$ be a Coxeter group with generating set $S$. The \emph{Coxeter complex}, denoted by $\bC_S$, is defined in a similar way to the Artin complex, and a vertex of $\bC_S$ is of type $\hat s$ if it corresponds to a coset of form $g W_{S/\{s\}}$. The fundamental domain of the action $W_S\curvearrowright \bC_S$ is a top-dimensional simplex. The same holds true for the action $A_S\curvearrowright \Delta_S$. These two fundamental domains can be canonically identified, by looking at types of vertices. Given 
$S'\subset S$, the \emph{$(S,S')$-relative Artin complex}
$\Delta_{S,S'}$, is the induced
subcomplex of $\Delta_S$ spanned by vertices of type $\hat s$ with $s\in S'$.
Equivalently, we write $\Delta_{S,S'}=\Delta_{\Lambda,\Lambda'}$, where
$\Lambda'$ is the induced subdiagram of $\Lambda$ on $S'$. As the action $A_S\curvearrowright \Delta_S$ preserves types of vertices, we obtain an action $A_S\curvearrowright \Delta_{S,S'}$, whose fundamental domain is again a simplex with its vertices labeled by $\{\hat s\}_{s\in S'}$. Hence the fundamental domain of the action $A_S\curvearrowright \Delta_{S,S'}$ can be canonically identified with the fundamental domain of $W_{S'}\curvearrowright \bC_{S'}$. 

We say a Coxeter diagram $\Lambda$ \emph{dominates} another Coxeter diagram $\Lambda'$, if there is an isomorphism $f:\Lambda\to\Lambda'$ of the underlying graphs such that for any edge $e$ of $\Lambda$, the label of $e$ is $\ge$ the label of $f(e)$. A tree Coxeter diagram is \emph{$\wtC$-free}, if it does not contain an induced subdiagram that dominates a Coxeter diagram of type $\widetilde C_n,\widetilde B_n$ or $\widetilde D_n$. These are exactly diagrams in Definition~\ref{def:special}. On the other hand, if a tree Coxeter diagram $\Lambda$ contains a subdiagram $\Lambda'$ dominating a Coxeter diagram $\Lambda''$ of type $\widetilde C_n,\widetilde B_n$ or $\widetilde D_n$, then we metrize $\Delta_{\Lambda,\Lambda'}$ or its appropriate subdivision as follows. It suffices to metrize the fundamental domain of $A_\Lambda\curvearrowright \Delta_{\Lambda,\Lambda'}$, which is identified with the fundamental domain of $W_{\Lambda''}\curvearrowright \bC_{\Lambda''}$ via the dominating map $\Lambda'\to \Lambda''$. 

%So it suffices to metrize the later fundamental domain.

When $\Lambda''$ is of type $\widetilde C_n$, $\bC_{\Lambda''}$ is isomorphic to the tessellation of $\mathbb R^n$ by unit orthoschemes, so the fundamental domain of $W_{\Lambda''}\curvearrowright \bC_{\Lambda''}$ can be metrized as a unit orthoscheme with $\ell^\infty$-metric. As up to scaling, the fundamental domain of the action of the Coxeter group of type $\widetilde B_n$ (resp. $\widetilde D_n$) on $\mathbb R^n$ can be naturally subdivided into 2 (resp. 4) unit orthoschemes, this induces a subdivision of $\Delta_{\Lambda,\Lambda'}$ as a piecewise $\ell^\infty$ orthoscheme complex.

\medskip
\noindent
\textbf{Propagation of link conditions}\ \ \ 
We wish to show $\Delta_{\Lambda,\Lambda'}$ (or its subdivision) is an injective orthoscheme complex. We use a theorem of Haettel \cite{haettel2021lattices,haettel2022link}, which roughly says that for certain type of piecewise $\ell^\infty$ orthoscheme complex $X$ (which applies to $\Delta_{\Lambda,\Lambda'}$), if the link $\lk(x,X)$ of each vertex $x\in X$ satisfies a purely combinatorial condition, and $X$ is simply-connected, then $X$ is an injective orthoscheme complex, in particular $X$ is contractible, see Theorem~\ref{thm:contractibleII} for a precise statement. This combinatorial condition on $\lk(x,X)$ can be formulated in graph theoretical terms - namely it is required that certain 4-cycles in the 1-skeleton of $\lk(x,X)$ have a center and certain 6-cycles in the 1-skeleton of $\lk(x,X)$ have a quasi-center.

The simple-connectedness of $\Delta_{\Lambda,\Lambda'}$ is already proved in \cite{huang2023labeled}. Given a vertex $x\in \Delta_{\Lambda,\Lambda'}$ of type $\hat s$, its link is again a relative Artin complex $\Delta_{\Lambda\setminus\{s\},\Lambda'\setminus\{s\}}$. So checking the link condition reduces to studying the existence of centers or quasi-centers for a collection of 4-cycles and 6-cycles in $\Delta_{\Lambda\setminus\{s\},\Lambda'\setminus\{s\}}\subset \Delta_{\Lambda\setminus\{s\}}$. This can be formulated as a property of $\Delta_{\Lambda\setminus\{s\}}$, called \emph{atomic BD robust} in Definition~\ref{def:good}. As we wish to focus on the overall strategy, we skip the technical discussion on the precise formulation of atomic BD robust, and simply mention that
when connected components of $\Lambda\setminus\{s\}$ belong to families in Definition~\ref{def:special}, then the assumption of Theorem~\ref{thm:intro shephard} exactly corresponds to these families being atomic BD robust. Thus if $\Lambda\setminus\{s\}$ belongs to the families in Definition~\ref{def:special}, we conclude immediately that $\Delta_{\Lambda,\Lambda'}$ is injective (hence contractible). This establishes the first level of the tower. In general, knowing the link condition (atomic BD robust) for Artin groups in level $k-1$ implies Artin groups in level $k$ admit desired actions on injective metric spaces. 

The main difficulty here is to establish the atomic BD robustness for Artin groups in higher levels (the assumption of Theorem~\ref{thm:intro shephard} only guarantees it in level 0). The key result in this article is the following propagation theorem on atomic BD robustness, from which the existence of the desired tower follows.

\begin{thm}[Theorem~\ref{thm:BD robust prop}]
	\label{thm:BD robust prop into}
	Let $\Lambda$ be a tree Coxeter diagram such that every proper induced subdiagram is atomic BD-robust. Suppose $\Lambda$ is not in level $0$ (i.e. does not belong to Definition~\ref{def:special}). Then $\Lambda$ is atomic BD-robust. 
\end{thm}

\noindent
\textbf{Bridging Theorem~\ref{thm:BD robust prop into} with convexity}\ \ \ 
Let $\Lambda$ be as in Theorem~\ref{thm:BD robust prop into}. There are different types of 4-cycles and 6-cycles to consider in the proof of atomic BD-robustness of $\Lambda$. To illustrate the idea, we only discuss a particular case and 
let $\omega=x_1x_2x_3x_4x_5x_6$ be the type of 6-cycle in the 1-skeleton of $\Delta_\Lambda$ which the atomic BD-robustness requires to find a quasi-center $x\in\Delta_\Lambda$ that is adjacent to $\{x_1,x_3,x_5\}$.

As $\Lambda$ is not at level $0$, we can find a subdiagram $\Lambda'\subset\Lambda$ dominating a Coxeter diagram $\Lambda''$ of type $\widetilde C_n,\widetilde B_n$ or $\widetilde D_n$. As all the diagrams at a lower level compared to $\Lambda$ satisfy the atomic BD-robust condition, by previous discussion, $\Delta_{\Lambda,\Lambda'}$ is an injective orthoscheme complex. In general, the subcomplex $\Delta_{\Lambda,\Lambda'}$ of $\Delta_\Lambda$ and the 6-cycle $\omega$ might not be related to each other. By a careful choice of $\Lambda'$, one might arrange that $x_2,x_4,x_6\in \Delta_{\Lambda,\Lambda'}$, however, in general $x_1,x_3,x_5$ are outside $\Delta_{\Lambda,\Lambda'}$.

For $i=1,3,5$, let $X_i$ be the full subcomplex of $\Delta_{\Lambda,\Lambda'}$ spanned by vertices of $\Delta_{\Lambda,\Lambda'}$ that are adjacent to $x_i$. As $x_2,x_4,x_6\in \Delta_{\Lambda,\Lambda'}$, we know $\{X_1,X_3,X_5\}$ pairwise intersects. Eventually we will show $\{X_1,X_3,X_5\}$ have a common intersection, which gives the desired quasi-center $x$. The idea is to argue $X_i$ is convex in $\Delta_{\Lambda,\Lambda'}$ in an appropriate sense, then we use convexity to find the common intersection.

\medskip
\noindent
\textbf{Normal form convexity via bi-Helly graphs}\ \ \ 
Since we only have local combinatorial information describing how the subcomplexes $X_i$ sit
inside $\Delta_{\Lambda,\Lambda'}$ in terms of vertex link conditions, the above argument requires a notion of convexity—usually a
global property—that admits a purely local combinatorial characterization in terms of vertex
links. This is achieved by relating our setting to properties of bi-Helly graphs studied in \cite{HodaMunro:BiHelly_Directed_Geodesics:2026}, as we now explain.

Let $Z$ be an injective orthoscheme complex. For each ordered pair of vertices $(x,y)$, we wish
to specify a preferred geodesic path in the $1$-skeleton $Z^1$ from $x$ to $y$, called the
\emph{normal form path} from $x$ to $y$. In general, the normal form path from $x$ to $y$ need
not coincide with the normal form path from $y$ to $x$. We first show that a certain subgraph
of $Z^1$ is a bi-Helly graph (see Definition~\ref{def:bihelly}). In \cite{HodaMunro:BiHelly_Directed_Geodesics:2026}, the first named author
and Munro showed that in any bi-Helly graph one can travel between any two vertices
via a canonical normal form sequence of near-cliques called a directed geodesic (Definition~\ref{def:directed geodesics}), and
that such sequences admit a purely local characterization (Theorem~\ref{thm:local to global}).

We translate this normal form sequence into a canonical geodesic path in $Z^1$. Under mild
additional assumptions on $Z$, the local characterization from \cite{HodaMunro:BiHelly_Directed_Geodesics:2026} admits a link condition
interpretation: given a sequence of consecutive vertices $z_1,z_2,\ldots,z_n$ in $Z$, one can
detect whether it forms a normal form path via a purely combinatorial condition describing how
$z_{i-1}$ and $z_{i+1}$ sit in the link $\lk(z_i,Z)$. The detectability of normal form
paths via vertex link condition is the key feature required for our approach.

We now define a subcomplex $Z'\subset Z$ to be \emph{convex} if, for any vertices $x,y\in Z'$,
the normal form path from $x$ to $y$ is entirely contained in $Z'$. Although this definition is
global in nature, the local characterization of normal form paths yields a vertex link
criterion for convexity. In particular, we show that the subcomplexes $X_i\subset
\Delta_{\Lambda,\Lambda'}$ introduced earlier are convex in this sense. %Interestingly, this notion of convexity already differs from other notions even in the simplest case where $Z$ is $\mathbb{R}^n$ equipped with its standard unit orthoscheme tessellation.

\medskip
\noindent
\textbf{Bestvina-type inequality and triangulation of disks}\ \ \  Unfortunately, for this notion of convexity it is not true in general that three pairwise
intersecting convex subcomplexes have nonempty common intersection. However, this property does
hold for certain special convex subcomplexes, including the subcomplexes
$X_1,X_3,X_5\subset\Delta_{\Lambda,\Lambda'}$ discussed above.

To prove this, let $\Theta$ denote the collection of all triples $(u_1,u_3,u_5)$ of vertices
such that $u_1\in X_1\cap X_3$, $u_3\in X_3\cap X_5$, and $u_5\in X_5\cap X_1$. Each such triple
determines a triangle whose sides are the normal form paths from $u_1$ to $u_3$, from $u_1$ to
$u_5$, and from $u_3$ to $u_5$. By convexity, each side of the triangle is contained in one of
the subcomplexes $X_1,X_3,X_5$. Among all triangles arising in this way, we show that any
triangle of minimal size must be degenerate, which implies that
$X_1\cap X_3\cap X_5\neq\emptyset$. Since $\Delta_{\Lambda,\Lambda'}$ is simply connected, each such triangle spans a triangulated
disk. It therefore suffices to show that if a triangle is of minimal size, then the
triangulation of the corresponding disk is necessarily degenerate.

A key ingredient in the proof is an inequality inspired by work of Bestvina. In the context
of Garside theory, Bestvina introduced an asymmetric metric on the central quotient of a
Garside group, and showed that given a vertex $x$ and a normal form sequence from $y$ to $z$,
an asymmetric distance from $x$ to points along this sequence first strictly decreases and
then strictly increases \cite[Proposition~3.12]{Bestvina1999}.

We show that a version of Bestvina inequality still holds for certain injective orthoscheme complexes. Namely, given $u_1$ and the normal form path
from $u_3$ to $u_5$, the distance in $Z^1$ from $u_1$ to vertices along this path
first decreases and then increases. Our inequality is slightly weaker than Bestvina’s original formulation: the decreasing segment need not be strictly monotone—explicit examples exhibit this behavior. Nevertheless, this weaker form suffices to control the geometry of the triangles described above and the triangulations of their spanning disks, which constitutes the most technically involved part of the article.

%Our inequality is slightly weaker in that the decreasing portion need not be strictly monotone—indeed, explicit examples show this phenomenon—but it is nevertheless sufficient to control the geometry of the triangles above and the triangulations of their spanning disks, which forms the most involved part of the article.

\subsection{Organization of the article}
In Section~\ref{sec:prelim} we collection preliminaries on simplicial complexes of type $S$, Artin groups, Shephard groups, Coxeter groups, posets and contractibility criterion on simplicial complexes.  In Section~\ref{sec:nf} we discuss normal forms in the 1-skeleton of injective orthoscheme complexes, notion of convex subcomplexes, and Bestvina type inequality.  Section~\ref{sec:subdivision} is mostly a review of material in \cite{huang2024} about subdivision of certain simplicial complexes. Section~\ref{sec:propagation} is devoted to study configuration of several convex subcomplexes in certain injective orthoscheme complexes, and we prove the key propagation theorem (Theorem~\ref{thm:BD robust prop}), which allows us to deduce the Main Theorem. In Section~\ref{sec:ABI} we prove Theorem~\ref{thm:ABI}.

\subsection{Acknowledgment} We kindly thank Piotr Przyticki and Katherine Goldman for helpful discussions in the course of this research.  We also thank AIM and the organizers of the AIM workshop ``Geometry and topology of Artin groups,'' in September 2023, where this research was initiated.

The first named author was partially funded by an NSERC Postdoctoral Fellowship.
The second named author is partially supported by a Sloan fellowship and NSF grant DMS-2305411. 

\section{Preliminaries}
\label{sec:prelim}

\subsection{Simplicial complexes of type $S$}

Let $S=\{s_1,s_2,\ldots,s_n\}$. A simplicial complex $\Delta_S$ is of \emph{type $S$} if it is flag, all the maximal simplices of $\Delta_S$ has dimension $n-1$ and there is a type function $\type$ from the vertex set of $\Delta_S$ to $\{\hat s_1,\hat s_2,\ldots,\hat s_n\}$ such that $\type(x)\neq\type(y)$ whenever $x$ and $y$ are adjacent vertices of $\Delta_S$. This labeling induces a bijection between $\{\hat s_1,\hat s_2,\ldots,\hat s_n\}$ and the vertex set of each maximal simplex of $\Delta_S$. For $S'\subset S$, we define the \emph{$(S,S')$ relative subcomplex} $\Delta_{S,S'}$ be the full subcomplex of $\Delta_S$ spanned by vertices of type $\hat s$ with $s\in S'$.

If the type of the vertices of a face $\sigma$ are $\{\hat s:s\in T\}$, then the \emph{type} of $\sigma$ is defined to be $\hat T$. For a simplex $\sigma\subset \Delta_{S}$ of type $\hat T$, we define $\Delta^\sigma_{S\setminus T}$ to be $\lk(\sigma,\Delta_{S})$, which is a simplicial complex of type $S\setminus T$. For any $R\subset S\setminus T$, let $\Delta^\sigma_{S\setminus T,R}$ be the full subcomplex of $\Delta^\sigma_{S\setminus T}$ spanned by vertices of type $\hat s$ with $s\in R$. Then for any $T\subset S'\subset S$, $\Delta^\sigma_{S\setminus T,S'\setminus T}=\lk(\sigma,\Delta_{S,S'})$. 

\begin{remark}
	\label{rmk:notation}
	For $R\subset S\setminus T$, we will also use $\Delta^\sigma_R$ to denote $\Delta^\sigma_{S\setminus T,R}$, as this complex is completed determined by $\sigma$ and $R$.
	If $R_1\subset R_2\subset S\setminus T$, then $\Delta^\sigma_{R_1}\subset \Delta^\sigma_{R_2}$. So $\Delta^\sigma_{R_1}$ can also be viewed as a relative subcomplex inside $\Delta^\sigma_{R_2}$ which is a simplicial complex of type $R_2$. Thus we will also use the notation $\Delta^\sigma_{R_2,R_1}$ to denote $\Delta^\sigma_{S\setminus T, R_1}=\Delta^\sigma_{R_1}$. These two notations give the same complex, but have a different emphasis on what is the ambient complex where this complex sits as a relative subcomplex.
\end{remark}

\begin{definition}
	\label{def:regular}
A complex $\Delta_S$ of type $S$ is \emph{regular} if
\begin{enumerate}
	\item for any simplex $\sigma$ of codimension $\ge 2$ in $\Delta_S$, $\lk(\sigma,\Delta_S)$ is connected;
	\item for any simplex $\sigma$ of codimension $\ge 3$ in $\Delta_S$, $\lk(\sigma,\Delta_S)$ is simply-connected.
\end{enumerate}
We do not exclude the possibility that $\sigma=\emptyset$, in which case the link is $\Delta_S$. 
\end{definition}

\begin{lem}(\cite[Lem 10.2]{huangbestvina})
	\label{lem:scgeneral}
	Suppose $\Delta_S$ is regular. Then for any $S'\subset S$, $\Delta_{S,S'}$ is regular as a simplicial complex of type $S'$. In particular, $\Delta_{S,S'}$ is connected if $|S'|\ge 2$ and $\Delta_{S,S'}$ is simply connected if $|S'|\ge 3$.
\end{lem}

Let $\Delta'_S$ be the barycentric subdivision of $\Delta_S$. The \emph{type} of a vertex of $\Delta'_S$ is defined to be the type of the face of $\Delta_S$ containing this vertex as barycenter. For points $x,y\in \Delta_S$, we write $x\sim y$ if they are contained in the same simplex of $\Delta_S$. A \emph{generalized $k$-cycle} in $\Delta'_S$ is a sequence of vertices $\{x_i\}_{i\in \mathbb Z/k\mathbb Z}$ such that $x_i\sim x_{i+1}$ for each $i\in \mathbb Z/k\mathbb Z$. A \emph{generalized path} $\{x_i\}_{i=1}^k$ is a sequence of vertices in $\Delta'_S$ with $x_i\sim x_{i+1}$ for $1\le i\le k-1$.

\begin{definition}
	\label{def:Ck}
The complex $\Delta_S$ satisfies \emph{property C$_k$} if for any generalized $k$-cycle $\{x_i\}_{i\in \mathbb Z/k\mathbb Z}$ in $\Delta'_S$, we can replace the generalized path $x_2x_3\cdots x_k$ by another generalized path $x'_2x'_3\cdots x'_k$ with the same endpoints and same vertex types, such that $x_1\sim x'_i$ for $2\le i\le k$. Condition C$_k$ implies condition C$_\ell$ for $\ell\le k$.
\end{definition}

\begin{lem}
	\label{lem:C_k link inherit}
If $\Delta_S$ satisfies property C$_k$, then for any simplex $\sigma$ of type $\hat T$, $\Delta^\sigma_{S\setminus T}$ also satisfies property C$_k$. 
\end{lem}

\begin{proof}
Each simplex in $\Delta^\sigma_{S\setminus T}$ (viewed as a simplicial complex of type $S\setminus T$) can be enlarged to a simplex in $\Delta_S$ by joining with $\sigma$. 
Hence a generalized cycle $\{x_i\}_{i\in \mathbb Z/k\mathbb Z}$ in $(\Delta^\sigma_{S\setminus T})'$ determines a generalized cycle $\{\bar x_i\}_{i\in \mathbb Z/k\mathbb Z}$ in $\Delta'_S$. Let $\bar x'_2\bar x'_3\cdots\bar x'_k$ be the alternative generalized path given by Definition~\ref{def:Ck} with $\bar x_1\sim \bar x'_i$ for $2\le i\le k$. The restriction on type of $\bar x'_i$ implies that each $\bar x'_i$ is the barycenter of a simplex $\sigma_i$ that contains $\sigma$, hence each $\bar x'_i$ also determines a vertex $x'_i$ in $(\Delta^\sigma_{S\setminus T})'$ by considering the barycenter of $\sigma_i\cap \Delta^\sigma_{S\setminus T}$. Then $x'_2x'_3\cdots x'_k$ satisfies all the desired properties.
\end{proof}

Let $\Lambda$ be a diagram (i.e. a simplicial graph) with vertex set $S$. We say $\Lambda$ is \emph{doubly marked}, if we specify two sets of edges of $\Lambda$, denoted $E_m$ and $E_h$ with $E_h\subset E_m$. Each induced subdiagram $\Lambda'$ of a doubly marked diagram $\Lambda$ also have an induced doubly marked structure.
We will also use $\Delta_{\Lambda,\Lambda'}$ to denote $\Delta_{S,S'}$, where $\Lambda'$ is the subdiagram of $\Lambda$ induced on $S'$.

\begin{definition}
	\label{def:model}
Let $\Lambda$ be a diagram with vertex set $S$. A simplicial complex $\Delta_S$ of type $S$ is \emph{modeled on $\Lambda$} if
 for any simplex $\sigma$ of type $\hat T$ in $\Delta_S$, we have a join decomposition 
	\begin{equation}
		\label{eq:join}
		\Delta^\sigma_{S\setminus T}=\Delta^\sigma_{S\setminus T,\Pi_1}*\Delta^\sigma_{S\setminus T,\Pi_2}*\cdots*\Delta^\sigma_{S\setminus T,\Pi_k}
	\end{equation}
	where $\Pi_1,\Pi_2,\ldots,\Pi_k$ are connected components of the subdiagram of $\Lambda$ induced on $S\setminus T$.
If $\Lambda$ is a doubly marked diagram, then we require in addition that
\begin{enumerate}
	\item for any codimension 2 simplex $\sigma$ of type $\hat T$ such that $S\setminus T$ spans an edge in $E_m$, $\Delta^\sigma_{S\setminus T}$ is a graph with girth $\ge 8$;
	\item for any codimension 2 simplex $\sigma$ of type $\hat T$ such that $S\setminus T$ spans an edge in $E_h$, $\Delta^\sigma_{S\setminus T}$ is a graph with girth $\ge 12$.
\end{enumerate}
\end{definition}

The following two lemmas are direct consequence of \eqref{eq:join}.

\begin{lem}
	\label{lem:transitive}
	Let $x_1,x_2,x_3$ be vertices of $\Delta'_\Lambda$ of types
	$\hat S_1,\hat S_2,\hat S_3$, respectively. Suppose that for any
	$s\in S_1\setminus S_2$ and $t\in S_3\setminus S_2$, the vertices $s$ and $t$
	lie in different connected components of $\Lambda\setminus S_2$.
	If $x_1\sim x_2$ and $x_2\sim x_3$, then $x_1\sim x_3$.
\end{lem}

\begin{lem}
		\label{lem:link}
	Let $\Delta_S$ be a simplicial complex of type $S$ modeled on a diagram $\Lambda$. Let $\Lambda'\subset \Lambda$ be an induced subdiagram.  Let $v\in \Delta_{\Lambda,\Lambda'}$ be a vertex of type $\hat s$ with $s\in \Lambda'$. Let $\Lambda_s$ and $\Lambda'_s$ be the induced subgraph of $\Lambda$ and $\Lambda'$ respectively spanned all the vertices which are not $s$. Let $I_s$ be the union of connected components of $\Lambda_s$ that contain at least one component of $\Lambda'_s$. Then $\Lambda'_s\subset I_s$ and  $\lk(v,\Delta_{\Lambda,\Lambda'})=\Delta^v_{I_s,\Lambda'_s}$ (see Remark~\ref{rmk:notation} for clarification of notations). More generally, links of faces are given by the same formula with the actual face as superscript.
\end{lem}

\begin{definition}
	\label{def:nondeg}
Let $\Delta_S$ be a simplicial complex of type $S$ modeled on a diagram $\Lambda$. We say $\Delta_S$ is \emph{nondegenerate}, if for any nonempty simplex $\sigma$ of type $\hat T$, and for
any $t,r$ in the same component of $\Lambda\setminus T$ and
every vertex $y\in \lk(\sigma,\Delta_S)$ of type $\hat t$, there is a vertex $z\in \lk(\sigma,\Delta_S)$ of type $\hat r$ which is not equal or adjacent to $y$.
\end{definition}
We allow $t=r$ in the definition. 
In particular, nondegenracy implies that each codimension 1 face is contained in at least two top dimensional simplices.

%The following is a definition chase.
%
%\begin{lem}
%		\label{lem:nonadj}
%		Let $\Delta_S$ be a simplicial complex of type $S$ modeled on a diagram $\Lambda$.
%	Let $\{s,t,r\}$ be three pairwise distinct vertices in a diagram $\Lambda$ such that $t$ and $r$ are in the same component of $\Lambda\setminus\{s\}$. Let $xy$ be an edge in $\Delta_{\Lambda}$ such that $x$ has type $\hat s$ and $y$ has type $\hat t$. Then there exists a vertex $z\in \Delta_\Lambda$ of type $\hat r$ such that $z\sim x$ and $z\nsim y$.
%\end{lem}

We record a result on homotopy types of $\Delta_S$.

\begin{lem}(\cite[Lem 10.5]{huangbestvina})
	\label{lem:dr}
	Let $\Delta_S$ be a simplicial complex of type $S$.
	Let $T$ be a subset of $S$. Suppose that for any simplex $\sigma$ in $\Delta_{S}$ of type $\hat R$ with $R\subset T$ and $R\neq\emptyset$, $\Delta^\sigma_{S\setminus R}$ is contractible. Then $\Delta_S$ deformation retracts onto $\Delta_{S,S\setminus T}$.
\end{lem}

\subsection{A group theoretical setting}
Let $\Lambda$ be a finite diagram with vertex set $S$. Let $G$ be a group generated by a family of subgroups $\{H_s\}_{s\in S}$. For $T\subset S$, let $G_T$ be the subgroup generated by $\{H_s\}_{s\in T}$ (with $G_\emptyset$ being the trivial subgroup). For each pair $s,t$ with $s\neq t$, choose relations \(\mathcal R_{st}\) presenting \(G_{\{s,t\}}\) as a quotient of \(H_s*H_t\).

For each $s\in S$, let $G_{\hat s}=G_{S\setminus\{s\}}$.
We define the \emph{coset complex} $\Delta_\Lambda$ (or $\Delta_S$) associated with the admissible parabolic system as follows. Its vertices is in one to one correspondence with left cosets of $\{G_{\hat s}\}_{s\in S}$. A collection of vertices span a simplex if the corresponding cosets have non-empty common intersection. In other words, $\Delta_S$ is the nerve of the covering of $G$ by left cosets of $\{G_{\hat s}\}_{s\in S}$. For each induced subdiagram $\Lambda'\subset \Lambda$, we define \emph{relative coset complex} $\Delta_{\Lambda,\Lambda'}$, and types of vertices and simplices in $\Delta_{\Lambda,\Lambda'}$ in the same way as before. 
\begin{definition}
	\label{def:admissible parabolic system}
	We call $(G,\{H_s\}_{s\in S})$ an
	\emph{strongly admissible parabolic system over $\Lambda$}, if the following properties hold simultaneously:
	\begin{enumerate}
		\item	 For any $T\subset S$, the natural homomorphism from the following group to $G_T$ is an isomorphism:
		\[
		\left(\underset{s\in T}{*}H_s\right)
		\Big/
		\left\langle\!\left\langle
		\mathcal R_{st}:\,s,t\in T,\ s\ne t
		\right\rangle\!\right\rangle;
		\]
		\item  $H_s$ and $H_t$ commute if $s,t$ are not adjacent in $\Lambda$;
		\item For any $U,T_1,T_2,T_3,T_4\subset S$, we have $$G_U\cap (G_{T_1}G_{T_2}G_{T_3}G_{T_4})=G_{U\cap T_1}G_{U\cap T_2}G_{U\cap T_3}G_{U\cap T_4}.$$
	\end{enumerate}
	If $\Lambda$ is a doubly marked diagram with marked edge sets $E_h\subset E_m$, then we require in addition that for any edge $e=st\in E_m$, the coset complex for $G_{\{s,t\}}$ is a graph with girth $\ge 8$, and for any edge $e=st\in E_h$,  $G_{\{s,t\}}$ has girth $\ge 12$.
	
The parabolic system  $(G,\{H_s\}_{s\in S})$ is \emph{nondegenerate}, if for any $T\subset S$ spanning a connected subdiagram of $\Lambda$, $G_T\neq G_{T\setminus\{t\}}G_{T\setminus\{r\}}$ for any pair $t,r\in T$.
\end{definition}

In Definition~\ref{def:admissible parabolic system} (3), we could assume instead $U,T_1,T_2,T_3,T_4\subsetneq S$, as if one of these sets are $S$, the property is trivial. Also if we set $T_2=T_3=T_4=\emptyset$, then Definition~\ref{def:admissible parabolic system} implies that $G_U\cap G_{T_1}=G_{U\cap T_1}$ for any $U,T_1\subset S$.

When the parabolic system is nondegenerate, $H_s$ is nontrivial subgroup of $G$ for each $s\in S$, by taking $T=\{s\}$ in the definition of nondegeneracy. 

\begin{lem}
	\label{lem:group}
	Suppose $(G,\{H_s\}_{s\in S})$ is a strongly admissible parabolic system over $\Lambda$ (doubly marked or unmarked). Then $\Delta_S$ is a regular simplicial complex of type $S$ modeled on $\Lambda$ and it satisfies condition C$_5$.
	
	If the parabolic system is nondegenerate, then $\Delta_S$ is nondegenerate.
\end{lem}

\begin{proof}
The first statement follows from the same argument in \cite[Lem 10.15]{huangbestvina}. For the second statement, let $\sigma\subset \Delta_S$ be a simplicial complex of type $\hat R$. Let $t,r$ be vertices in the same connected component $\Lambda_T$ of $\Lambda\setminus R$. Let $y\in \lk(\sigma,\Delta_S)$ be a vertex of type $\hat t$. Let $\Delta_T$ be the coset complex associated to the parabolic system $(G_T,\{H_s\}_{s\in S})$, with $T$ being the vertex set of $\Lambda_T$. Then $\Delta_T$ embeds naturally as a join factor of $\lk(\sigma,\Delta_S)$, with $y\in \Delta_T$. It suffices to find a vertex $z\in \Delta_T$ of type $\hat r$ such that $z\nsim y$. 

Suppose $t\neq r$. Up to left translation, we can assume $y$ corresponds to the identity cosets $G_{T\setminus\{t\}}$. Then any left cosets of $G_{T\setminus\{r\}}$ has a nontrivial intersection with $G_{T\setminus\{t\}}$ if and only if $G_T=G_{T\setminus\{t\}}G_{T\setminus\{r\}}$. So $G_T\neq G_{T\setminus\{t\}}G_{T\setminus\{r\}}$ implies that there is a vertex $z\in \Delta_T$ of type $\hat r$ with $z\nsim y$. It remains to consider $t=r$, in which the desired $z$ corresponds the coset $gG_{T\setminus\{t\}}$ with $g$ being a nontrivial element in $H_t$.
\end{proof}

\subsection{Artin groups, Coxeter groups and Shephard groups}
\label{subsec:rel Artin complex}
A \emph{Coxeter diagram} $\Lambda$ is a finite simple graph with vertex set
$S=\{s_i\}_i$, where each edge $s_is_j$ is labeled by an integer
$m_{ij}\in\{3,4,\dots,\infty\}$. If $s_is_j$ is not an edge, we set $m_{ij}=2$.
The associated \emph{Artin group} $A_\Lambda$ is generated by $S$ with defining
relations $
s_is_js_i\cdots = s_js_is_j\cdots,$
where both sides are alternating words of length $m_{ij}$ whenever
$m_{ij}<\infty$. The \emph{Coxeter group} $W_\Lambda$ is obtained from
$A_\Lambda$ by adding the relations $s_i^2=1$ for all $i$.

We say that $A_\Lambda$ is
\emph{spherical} if $W_\Lambda$ is finite. 

%For any subset $S'\subset S$, the subgroup generated by $S'$ is canonically isomorphic to the Artin group $A_{\Lambda'}$, where $\Lambda'$ is the induced subdiagram of $\Lambda$ on $S'$ \cite{lek}. Such a subgroup is called a \emph{standard parabolic subgroup}.

An \emph{extended Coxeter diagram} is a Coxeter diagram with each vertex labeled by an element in $\{2,3,4,\ldots\}\cup \{\infty\}$ such that any two adjacent vertices connected by an odd edge in the diagram share the same vertex label. We still use $A_\Lambda$ and $W_\Lambda$ to denote the corresponding Artin group and Coxeter group, ignoring the vertex labels. Let $\mathfrak{S}_\Lambda$ be the associated \emph{Shephard group}, which is the quotient
\[
\mathfrak{S}_\Lambda=A_\Lambda/\langle\!\langle s_i^{p_i}\rangle\!\rangle_{i\in I}
\]
where $\{s_i\}_{i\in I}$ denotes the collection of vertices of $\Lambda$ whose label $p_i$ is finite. Let $S$ be the vertex set of $\Lambda$. A \emph{standard parabolic subgroup} $\mathfrak{S}_T $ of $\mathfrak{S}_\Lambda$ is the subgroup of $\mathfrak{S}_\Lambda$ generated by $T\subset S$. 
A Shephard group $\mathfrak{S}_\Lambda$ is \emph{strongly admissible} if
\begin{enumerate}
	\item for any induced subdiagram $\Lambda'\subset\Lambda$, the natural homomorphism $\mathfrak{S}_{\Lambda'}\to \mathfrak{S}_\Lambda$ is injective;
	\item For any $U,T_1,T_2,T_3,T_4\subset S$, we have 
	\begin{center}
		$\mathfrak{S}_U\cap (\mathfrak{S}_{T_1}\mathfrak{S}_{T_2}\mathfrak{S}_{T_3}\mathfrak{S}_{T_4})=\mathfrak{S}_{U\cap T_1}\mathfrak{S}_{U\cap T_2}\mathfrak{S}_{U\cap T_3}\mathfrak{S}_{U\cap T_4}$.
	\end{center}
\end{enumerate}

Each extended Coxeter diagram $\Lambda$ determined a doubly marked diagram, by forgetting vertex labels and edge labels, and choosing $E_m$ to be the collection of edges with label $\ge 4$, $E_h$ to be the collection of edges with label $\ge 6$. 

\begin{lem}(\cite[Lem 10.17]{huangbestvina})
	\label{lem:shephard to general}
	Suppose $\mathfrak{S}_\Lambda$ is an admissible Shephard group over an extended Coxeter diagram $\Lambda$ with vertex set $S$. Then $(\mathfrak{S}_\Lambda,\{H_s\}_{s\in S})$ is an admissible parabolic system over the marked diagram determined by $\Lambda$, where $H_s$ is the cyclic subgroup of $\mathfrak{S}_\Lambda$ generated by $s$.
\end{lem}

The following is a combination of \cite{AppelSchupp1983}, \cite{lek} and \cite[Lem 2.17]{huangbestvina}.
\begin{thm}
	\label{thm:even Artin}
Artin groups are strongly admissible.
\end{thm}

A Shephard group has \emph{even order}, if its vertex labels are either even or infinity.

\begin{thm}(\cite{goldman2026})
	\label{lem:even shephard}
	Even order Shephard groups are strongly admissible.
\end{thm}

\begin{lem}
	\label{lem:nonadj}
Let $\mathfrak{S}_\Lambda$ be an even Shephard group, and let $\Delta_S$ be its coset complex. Then $\Delta_S$ is nondegenerate.
\end{lem}

%Let $\{s,t,r\}$ be three pairwise distinct vertices in a Coxeter diagram $\Lambda$ such that $t$ and $r$ are in the same component of $\Lambda\setminus\{s\}$. Let $xy$ be an edge in $\Delta_{\Lambda}$ such that $x$ has type $\hat s$ and $y$ has type $\hat t$. Then there exists a vertex $z\in \Delta_\Lambda$ of type $\hat r$ such that $z\sim x$ and $z\nsim y$.

\begin{proof}
Let $\sigma\subset \Delta_\Lambda$ be a simplex of type $\hat T$. Let $t,r$ be vertices in the same connected component $\Theta$ of $\Lambda\setminus T$, and let $y$ be a vertex in $\lk(\sigma,\Delta_S)$ of type $\hat t$. As in the proof of Lemma~\ref{lem:group}, we define $\Delta_{\Theta}$ which is a join factor of $\Delta_{\Lambda\setminus\{s\}}$ containing $y$.
It  suffices to find $z\in \Delta_\Theta$ of type $\hat r$ with $z\nsim y$. Assume without loss of generality that $y$ corresponds to the identity coset $\Sh_{\Theta\setminus\{t\}}$. We need to find $g\in \Sh_\Theta$ such that $g\Sh_{\Theta\setminus\{r\}}\cap \Sh_{\Theta\setminus\{t\}}=\emptyset$. As $\Sh_{\Theta}$ is even, by considering the homomorphism from $\Sh_\Theta$ to the associated Coxeter group $W_\Theta$, it suffices to find $g\in W_\Theta$ such that
	\begin{equation}
		\label{eq:empty intersection}
		gW_{\Theta\setminus\{r\}}\cap W_{\Theta\setminus\{t\}}=\emptyset.
	\end{equation}  Let $t_1t_2\cdots t_k$ be an embedded edge path in $\Theta$ from $t_1=t$ to $t_k=r$. Let $g=t_1t_2\cdots t_k$. Then \cite[Lem 5.9 (2)]{huang2023labeled} implies that $g$ is the point in $gW_{\Theta\setminus\{r\}}$ that are closest to identity with respect to the word metric on $W_\Theta$. By standard facts of Coxeter groups, each $h\in gW_{\Theta\setminus\{r\}}$ has a reduced representative with $g$ as its prefix. Since this representative has letter $t$, $h\notin W_{\Theta\setminus\{t\}}$. Thus \eqref{eq:empty intersection} follows.
\end{proof}

 When the Shephard group is an Artin group $A_\Lambda$, the coset complex $\Delta_S$ associated with the parabolic system $(A_S,\{A_s\}_{s\in S})$ is also called the  \emph{Artin complex}, which was introduced in
\cite{CharneyDavis} and further studied in
\cite{godelle2012k,cumplido2020parabolic}. When the Shephard group is a Coxeter group, the coset complex, denoted $\bC_S$, is called the \emph{Coxeter complex}.

\begin{thm}[\cite{godelle2012k}, Thm.~3.1]
	\label{thm:kpi1}
	Suppose the Artin group $A_S$ is not spherical. If the Artin complex $\Delta_S$ is contractible and each
	parabolic subgroup $A_{\hat s}$ satisfies the $K(\pi,1)$-conjecture, then
	$A_S$ satisfies the $K(\pi,1)$-conjecture.
\end{thm}

\subsection{Posets}
\label{subsec:posets}
A poset $P$ is called \emph{weakly graded} if there is a \emph{poset map} $r:P\to \mathbb Z$, i.e.
such that for every $x<y$ in $P$, we have $r(x)<r(y)$: the map $r$ is called
a \emph{rank map}. A poset $P$ is \emph{weakly boundedly graded} if there is a rank
map $r:P\to\mathbb Z$ with finite image. 
An \emph{upper bound} for a pair of elements $a,b\in P$ is an element $c\in P$ such
that $a\le c, b\le c$. A \emph{minimal upper bound} for $a,b$ is an upper bound $c$ such
that there does not exist upper bound $c'$ of $a,b$ such that $c'<c$. The \emph{join} of
two elements $a, b$ in $P$ is an upper bound $c$ of them such that for any other
upper bound $c'$ of $a, b$, we have $c\le c'$. We define \emph{lower bound}, \emph{maximal lower
	bound}, and \emph{meet} similarly. In general, the meet or join of two elements in $P$
might not exist. A poset $P$ is a \emph{lattice} if any pair of elements have a meet
and a join. A \emph{quasi-bowtie} $x_1y_1x_2y_2$ consists of elements of $P$ satisfying $x_i<y_j$ for all $i,j=1,2$, and it is a \emph{bowtie} if these elements are distinct, $y_1,y_2$ are minimal upper bounds for $\{x_1,x_2\}$, and $x_1,x_2$ are maximal lower bounds for $\{y_1,y_2\}$.

%The name comes from that if we draw $y_1,y_2$ above $x_1,x_2$ in the Hasse diagram, then we obtain a bowtie shaped configuration.
\begin{definition}
	A \emph{center} for a quasi-bowtie $x_1y_1x_2y_2$ in $P$ is an element $z\in P$ such that $x_i\le z\le y_j$ for all $i,j=1,2$.
	A poset $P$ is \emph{bowtie free} if each quasi-bowtie in $P$ has a center.
\end{definition}

\begin{lem}[{\cite[Prop 1.5]{brady2010braids} and \cite[Prop 2.4]{haettel2023new}}]
	\label{lem:posets}
	If $P$ is a bowtie free weakly graded poset, then any nonempty subset $Q\subset P$ with a lower bound has the meet, and any nonempty subset $Q\subset P$ with an upper bound has the join.
\end{lem}

\begin{definition}
	\label{def:flag}
	A poset $P$ is \emph{upward flag} if any three pairwise upper bounded elements have an upper bound. A poset is \emph{downward flag} if any three pairwise lower bounded elements have a lower bound. A poset is \emph{flag} if it is both upward flag and downward flag.
	
	A poset $P$ is \emph{weakly upward flag} if whenever each pair in $\{x,y,z\}$ have a upper bound in $P$ which is not maximal, then $\{x,y,z\}$ have a common upper bound.
	Similarly, we define weakly downward flag and weakly flag.
\end{definition}
A weakly graded poset $(P,\le)$ is  with a rank function $r:P\to \mathbb Z$ taking value between $1$ and $n$ is \emph{$r$-saturated}, if for any $p\in P$ and any integers $m_1,m_2$ with $n\ge m_1>r(p)>m_2\ge 1$, there is $p_1,p_2\in P$ such that $p_1\ge p\ge p_2$ and $r(p_i)=m_i$ for $i=1,2$. 
\begin{lem}
	\label{lem:bowtie free criterion}
	Let $P$ be an $r$-saturated weakly graded poset with its rank function $r$ taking value between $1$ and $n$. Then $P$ is bowtie free under the following two additional assumptions:
	\begin{enumerate}
		\item for each $p\in P$ with $r(p)=n$, $P_{<p}=\{q\in P\mid q<p\}$ is bowtie free;
		\item each quasi-bowtie $x_1y_1x_2y_2$ in $P$ with $r(x_1)=r(x_2)<r(y_1)=r(y_2)=n$ has a center.
	\end{enumerate}
\end{lem}

\begin{proof}
	First we show any quasi-bowtie $x_1y_1x_2y_2$ with $r(x_1)=r(x_2)$ has a center. Indeed, this follows from Assumption 2 if $r(y_1)=r(y_2)=n$. Now we assume exactly one of $\{y_1,y_2\}$, say $y_2$, has rank $n$. Let $y'_1$ be a rank $n$ element such that is $y'_1>y_1$. Then the quasi-bowtie $x_1y'_1x_2y_2$ has a center $z$ by Assumption 2. If $r(z)=n$, then $y'_1=y_2$ and $y_1\le y_2$, hence $y_1$ is a center for $x_1y_1x_2y_2$. If $r(z)<n$, then $x_1y_1x_2z$ is a quasi-bowtie in $P_{<y'_1}$. By Assumption 1, $x_1y_1x_2z$ has a center $z'$. Note that $z'\le z\le y_2$, hence $z'$ is a center for $x_1y_1x_2y_2$. If both $\{y_1,y_2\}$ have rank $<n$, then let $y'_1$ be a rank $n$ element such that is $y'_1>y_1$. By previous discussion the quasi-bowtie $x_1y'_1x_2y_2$ has a center, and we can repeat the previous argument to produce a center for $x_1y_1x_2y_2$. 
	
	Now we argue by contradiction and
	suppose there is a quasi-bowtie $x_1y_1x_2y_2$ without a center.
	Suppose $x_1$ and $x_2$ are chosen such that $r(x_1)+r(x_2)$ is maximized among all quasi-bowties in $P$ without center. We assume without loss of generality that $r(x_1)>r(x_2)$. Then we find $x'_1\in P$ with $r(x'_1)=r(x_2)$ and $x'_1<x_1$. Then $x'_1y_1x_2y_2$ is a quasi-bowtie which has a center $z$ by previous paragraph. If $r(z)=r(x_2)$, then $x'_1=x_2$ and $x_2< x_1$, hence $x_1$ is a center for $x_1y_1x_2y_2$, contradiction. If $r(z)>r(x_2)$, then the quasi-bowtie $x_1y_1zy_2$ has a center $z'$ as $r(x_1)+r(z)>r(x_1)+r(x_2)$. Then $z'\ge z\ge x_2$ and $z'$ is a center for $x_1y_1x_2y_2$, contradiction. Hence the lemma is proved.
\end{proof}

\begin{cor}
	\label{cor:bowtie free criterion1}
	Let $P$ be an $r$-saturated weakly graded poset with its rank function $r$ taking value between $1$ and $n$. Suppose each quasi-bowtie $x_1y_1x_2y_2$ in $P$ with $r(x_1)=r(x_2)<r(y_1)=r(y_2)$ has a center. Then $P$ is bowtie free.
\end{cor}

\begin{proof}
	Let $P_{[1,k]}$ be the collection of elements of $P$ with $r$-value between $1$ and $k$. Then $P_{[1,k]}$ is $r$-saturated. Note that $P_{[1,n-1]}$ is bowtie free implies that Assumption 1 of Lemma~\ref{lem:bowtie free criterion}. Thus Lemma~\ref{lem:bowtie free criterion} implies that $P$ is bowtie free given that $P_{[1,n-1]}$ is bowtie free and Assumption 2 of Lemma~\ref{lem:bowtie free criterion} holds. By repeatedly applying this observation, we deduce the corollary.
\end{proof}

\begin{lem}
	\label{lem:upward flag criterion}
	Let $P$ be an $r$-saturated weakly graded poset with its rank function $r$ taking value between $1$ and $n$.	 Suppose that
	\begin{enumerate}
		\item $P$ is bowtie free;
		\item for each $p\in P$ with $r(p)=1$, $P_{>p}$ is upward flag;
		\item each pairwise upper bounded rank 1 triple have a common upper bound.
	\end{enumerate}
	Then $P$ is upward flag.
\end{lem}

\begin{proof}
	We argue by contradiction and let $\{p_1,p_2,p_3\}$ be a pairwise upper bounded collection in $P$ which do not have a common upper bound. Suppose $q_i$ is a upper bound for $\{p_i,p_{i+1}\}$ for $i\in \mathbb Z/3\mathbb Z$.
	Then at least one of them, say $p_1$, has rank $>1$. Take $p'_1\in P$ such that $p'_1<p_1$ and $r(p'_1)=1$. Then $\{p'_1,p_2,p_3\}$ is pairwise upper bounded. As the rank sum is smaller, they have a common upper bound $z$. 
	
	Note that $p'_1q_1p_2z$ and $p'_1q_3p_3z$ are quasi-bowties. As $P$ is bowtie free, let $z_1$ and $z_2$ be centers for these quasi-bowties respectively. Thus $\{z_1,z_2,p_1\}$ are pairwise upper bounded, and they are contained in $P_{\ge p'_1}$. As $r(p'_1)=1$, by Assumption 2, $\{z_1,z_2,p_1\}$ have a common upper bound $z'$. Note that $p_2\le z_1\le z'$ and $p_3\le z_2\le z'$. Thus $z'$ is a common upper bound for $\{p_1,p_2,p_3\}$, contradiction. Thus the lemma is proved.
\end{proof}

\begin{cor}
	\label{cor:upward flag criterion}
	Let $P$ be an $r$-saturated weakly graded poset with its rank function $r$ taking value between $1$ and $n$.	  Suppose that
	\begin{enumerate}
		\item  $P$ is bowtie free;
		\item for any $1\le k\le n-1$, each pairwise upper bounded rank $k$ triple have a common upper bound.
	\end{enumerate}
	Then $P$ is upward flag.
\end{cor}

%The interest of the bowtie free condition lies in the following observation.

\subsection{Poset properties of simplicial complexes of type $S$}
\label{subsec:contractible}

Let $\Delta_S$ be a simplicial complex of type $S$ modeled on a diagram $\Lambda$. Let $\Lambda'\subset\Lambda$ be a linear induced subdiagram. Choose a linear order of the vertices $\{s_i\}_{i=1}^n$ of $\Lambda'$, and define a relation on the vertex set of $\Delta_{\Lambda,\Lambda'}$ as follows: for vertices $w,v\in \Delta_{\Lambda,\Lambda'}$, $v<w$ if $v$ and $w$ are adjacent in $\Delta_{\Lambda,\Lambda'}$ and the type $\hat s_i$ of $v$ and the type $\hat s_j$ of $w$ satisfy $i<j$. This relation depends on the choice of the linear order on $\Lambda'$.

We say an induced subdiagram $\Lambda'$ of $\Lambda$ is \emph{admissible}, if for any vertex $s\in \Lambda'$, if $s_1,s_2\in \Lambda'$ are vertices in different connected components of $\Lambda'\setminus\{s\}$, then they are in different components of $\Lambda\setminus\{s\}$. By Lemma~\ref{lem:transitive}, if $\Lambda'$ an admissible linear subdiagram of $\Lambda$, then for any choice of the linear order on $\Lambda'$, the associated relation on $\Delta^0_{\Lambda,\Lambda'}$ is a poset.

\begin{definition}[{\cite[Def 6.8]{huang2023labeled}}]
	\label{def:bowtie free}
	Let $\Delta_S$ be a simplicial complex of type $S$ modeled on a diagram $\Lambda$.
	Suppose $\Lambda'$ is an admissible linear subdiagram of $\Lambda$. We define $\Delta_{\Lambda,\Lambda'}$ is \emph{bowtie free} if the poset defined on its vertex set as above is bowtie free. The property of being bowtie free does not depend on the choice of one of the two linear orders on $\Lambda'$. Similarly, $\Delta_{\Lambda,\Lambda'}$ is \emph{2-sided flag} or \emph{weakly flag}, if the poset defined on its vertex set is flag or weakly flag.
\end{definition}

\begin{remark}
	\label{rmk:bowtie free}
Let $\Delta_\Lambda$ be modeled on a diagram $\Lambda$. Let $\sigma\in \Delta_\Lambda$ be a simplex of type $\hat T$ such that $\Lambda\setminus T$ has only one connected component which is a path $P$. Suppose $\lk(\sigma,\Delta_\Lambda)=\Delta^\sigma_{P}$ is bowtie free. Then for any connected subpath $Q\subset P$, $\Delta^\sigma_{P,Q}$ is bowtie free. Moreover, for any simplex $\tau\subset \Delta^\sigma_P$ of type $\hat R$ with $R\cap Q=\emptyset$, $\Delta^{\sigma\cup\tau}_{P\setminus R, Q}$ is bowtie free.
\end{remark}

\begin{definition}
	\label{def:order}
	Let $X$ be a simplicial complex of type $S$.
	We put a total order on $S$, and define a relation $<$ on the vertex set $V$ of $X$ induced by this total order as follows: $x<y$ if $x$ and $y$ are adjacent, and $\type(x)<\type(y)$.
\end{definition}

%As all maximal simplices of $X$ have the same dimension, we know that for each $x\in V$ of type $\hat s$ such that $s$ is not the smallest element in $S$, there exists $x'\in V$ with $x'<x$; and for each $x\in V$ of type $\hat s$ such that $s$ is not the biggest element in $S$, there exists $x'\in V$ with $x'>x$.

%An induced subgraph $\Lambda'$ of $\Lambda$ is \emph{admissible} if for any node $x\in \Lambda'$, if $x_1,x_2\in \Lambda'$ are in different connected components of $\Lambda'\setminus\{x\}$, then they are in different components of $\Lambda\setminus\{x\}$.

\begin{definition}
	\label{def:labeled 4-wheel}
	Let $\Lambda$ be a forest diagram with vertex set $S$. Let $Z$ be a simplicial complex of type $S$ modeled on $\Lambda$.
 We say $Z$ satisfies the \emph{labeled 4-cycle condition} if for any induced 4-cycle in $Z^1$ with consecutive vertices being $\{x_i\}_{i=1}^4$ and their types being $\{\hat s_i\}_{i=1}^4$, there exists a vertex $x\in Z$ adjacent to each of $x_i$ such that the type $\hat s$ of $x$ satisfies that  $s$ is in the smallest subtree of $\Lambda'$ containing all of $\{s_i\}_{i=1}^4$.
\end{definition}

\begin{remark}
	\label{rmk:inherit}
	As an immediate consequence of the definition, suppose $\Lambda'$ is a tree induced subdiagram of $\Lambda$ and $\Lambda''\subset \Lambda'$ is a subtree. If $\Delta_{\Lambda,\Lambda'}$ satisfies the labeled 4-cycle condition, then $\Delta_{\Lambda,\Lambda''}$ satisfies the labeled 4-cycle condition.
\end{remark}

\begin{lem}(\cite[Lem 10.11]{huangbestvina})
	\label{lem:subdiagramsgeneral}
	Let $\Delta_S$ be a simplicial complex of type $S$ modeled on a diagram $\Lambda$. Suppose $\Delta_S$ satisfies condition C$_4$. Let $\Lambda',\Lambda_1,\Lambda'$ be connected induced subdiagrams of $\Lambda$ with $\Lambda'_1\subset\Lambda_1$, $\Lambda'_1\subset \Lambda'$ and $\Lambda'$ being an admissible linear subdiagram of $\Lambda$. Suppose $\Delta_{\Lambda,\Lambda'}$ is bowtie free. Then $\Delta^\sigma_{\Lambda_1,\Lambda'_1}$ is bowtie free for any simplex $\sigma$ of type $\hat T$ with $T\cap \Lambda_1=\emptyset$.
\end{lem}
Note that $T\cap \Lambda_1=\emptyset$ guarantees that $\Delta^\sigma_{\Lambda_1,\Lambda'_1}$ is well-defined (Remark~\ref{rmk:notation}).

The following is proved for Artin complexes in \cite[Lem 6.14 and Prop 6.17]{huang2023labeled}. However, the same proof generalizes to simplicial complexes of type $S$.
\begin{lem}
	\label{lem:4wheel}
Suppose $\Delta_\Lambda$ is modeled on a diagram $\Lambda$ and $\Delta_\Lambda$ satisfies condition C$_4$.
	Suppose $\Lambda'$ is an admissible tree subdiagram of $\Lambda$. Then the relative complex $\Delta_{\Lambda,\Lambda'}$ satisfies the labeled 4-cycle condition if and only if for all maximal linear subdiagram $\Lambda''\subset \Lambda'$, $\Delta_{\Lambda,\Lambda''}$ is bowtie free.
\end{lem}

\begin{lem}
	\label{lem:tripod}
	Suppose $\Delta_\Lambda$ is modeled on $\Lambda$ and satisfies condition C$_4$.
	Suppose $\Lambda'$ is an admissible tree subdiagram such that $\Delta=\Delta_{\Lambda,\Lambda'}$ satisfies the labeled 4-cycle condition. Let $x_1x_2x_3x_4$ be an embedded 4-cycle in $\Delta$. Suppose $x_i$ has type $\hat a_i$ with $a_i\in \Lambda'$ such that $\{a_1,a_2,a_3\}$ are the valence one vertices of a tripod subdiagram $\Lambda''$ of $\Lambda'$, and $a_2=a_4$.
	
	Then either $x_1$ is adjacent to $x_3$, or there is a vertex $y\in\Delta$ adjacent to each $x_i$ such that $y$ has type $\hat a$ with $a$ contained in the linear subdiagram of $\Lambda''$ from the valence three vertex of $\Lambda''$ to $a_2$.
\end{lem}

This lemma is an immediate consequence of \cite[Lem 6.18]{huang2023labeled}. However, the original proof of \cite[Lem 6.18]{huang2023labeled} missed one case. Here we give the correct proof which is also much shorter.

\begin{proof}
	Let $\Lambda_{12}$ be the linear subdiagram of $\Lambda$ from $a_2$ to $a_1$. We put a partial order on vertex set $V$ of $\Delta_{\Lambda,\Lambda_{12}}$ as in Section~\ref{subsec:posets} such that vertices of type $\hat a_2$ are minimal. Then $x_1$ is a common upper bound for $\{x_2,x_4\}$ in $(V,<)$.
	By Lemma~\ref{lem:4wheel} and Lemma~\ref{lem:posets}, $x_2$ and $x_4$ have the join $y$ in $(V,<)$. In particular, $y\le x_1$ in $(V,<)$. By applying C$_4$ to $x_2x_3x_4y$, there is a vertex 
	$y'\in \Delta$ of the same type as $y$ such that $y'\sim\{x_2,x_3,x_4\}$. In particular, $y'\in V$, which implies $y=y'$. So $y\sim x_3$.

	Suppose $y$ has type $\hat a$ with $a\in \Lambda_{12}$. Let $a''$ be the valence 3 vertex in $\Lambda''$. If $a$ is in the interval $[a_1,a'']$, then $x_1\sim x_3$ by Lemma~\ref{lem:transitive}. The case $a\in [a'',a_2]$ exactly corresponds the second possibility in the lemma.
\end{proof}

Now we record several results on Artin complexes.

\begin{thm}[{\cite[Thm 8.1]{huang2023labeled}}]
	\label{thm:bowtie free}
	Suppose $A_\Lambda$ is an irreducible spherical Artin group. Then for any linear subdiagram $\Lambda'\subset\Lambda$, $\Delta_{\Lambda,\Lambda'}$ is bowtie free. 
\end{thm}

This theorem is a consequence of Theorem~\ref{thm:4 wheel} below and Lemma~\ref{lem:4wheel}.

\begin{thm}[{\cite[Prop 2.8]{huang2023labeled}}]
	\label{thm:4 wheel}
	Suppose $A_S$ is an irreducible spherical Artin group. Then $\Delta_S$ satisfies the labeled 4-cycle condition.
\end{thm}

\begin{thm}[{\cite[Prop 6.6]{haettel2021lattices}}]
	\label{thm:triple}
	Let $A_S$ be the Artin group of type $B_n$. Let $S=\{s_1,s_2,\ldots,s_n\}$ be consecutive vertices the Coxeter diagram with $m_{s_{n-1},s_n}=4$. We  order $S$ by $s_1<s_2<\ldots<s_n$. Then the induced poset on $\Delta^0_S$ is upward flag.
\end{thm}

%\begin{figure}[h]
%	\centering
%	\includegraphics[scale=1]{D}
%	\caption{Coxeter diagram of type $D_n$.}
%	\label{fig:ad}
%\end{figure}

\subsection{A contractibility criterion for simplicial complexes}
\label{subsec:subdiv}

Given a simplicial graph $\Gamma$ endowed with the path metric such that each edge has length $1$, a \emph{combinatorial ball} in $\Gamma$ is the collection of vertices in the metric ball $B(x,r)$ of radius $r$ centered at a vertex $x$. The graph $\Gamma$ is \emph{Helly} if whenever a collection of combinatorial balls in $\Gamma$ have non-empty pairwise intersection, then the common intersection of these balls is non-empty. %A geodesic metric space is \emph{injective} if whenever a collection of closed metric balls in the space have non-empty pairwise intersection, then the common intersection of these balls is non-empty.

Let $S=\{s_1,\ldots,s_n\}$ with a total order $s_1<s_2<\cdots<s_n$. This gives a relation on the vertex set of a simplicial complex of type $S$ as in Definition~\ref{def:order}. The following is a consequence of work of Haettel \cite[Thm 1.15]{haettel2021lattices}, \cite[\S 7]{haettel2021lattices} and \cite[Sec 4.3, Thm B]{haettel2022link}.

\begin{thm}	
	\label{thm:contractibleII}
	Let $X$ be a simply connected simplicial complex of type $S$. For vertex $x$, let $V_{\ge x}$ (resp. $V_{\le x}$) be the collection of vertices that is $\ge x$ (resp. $\le x$). Assume that
	\begin{enumerate}
		\item the relation $<$ on the vertex set $V$ of $X$ is a partial order;
		\item for each $x\in V$, $V_{\ge x}$ is bowtie free and upward flag;
		\item for each $x\in V$, $V_{\le x}$ is bowtie free and downward flag.
	\end{enumerate}
	Then $X$ is contractible. Moreover, let $\Y$ be a graph whose vertex set is the same as the vertex set of $X$, and two vertices $y_1,y_2\in \Y$ are adjacent if there exist vertices $z_1\in X$ of type $\hat s_1$ and $z_2\in X$ of type $\hat s_n$ such that $z_1\le y_i\le z_2$ for $i=1,2$. Then $\Y$ is a Helly graph.
\end{thm}

The graph $\Y$ in the above theorem is called the \emph{thickening} of $X$. 

\begin{remark}
	\label{rmk:injective}
	In \cite[Sec 4.3, Thm B]{haettel2022link}, it is required that $X$ is locally finite. However, for conclusion of the above theorem, the locally finite assumption can be dropped by repeating \cite[\S 7]{haettel2021lattices}. More precisely, following \cite[Def 1.10]{haettel2021lattices}, using the order on the vertex set of each top-dimensional simplex in $X$, we can identify it with a unit orthoscheme (as discussed in Section~\ref{subsec:proof sketch}) such that $v_0$ is identified with the smallest vertex and $v_n$ is identified with the biggest vertex. We put $\ell^\infty$-metric on each top-dimensional simplex in $X$, which induces a piecewise $\ell^\infty$-metric on $X$. Then by the argument in  \cite[\S 7]{haettel2021lattices}, one deduces that $X$ is uniformly locally injective from the assumptions of Theorem~\ref{thm:contractibleII} and \cite[Thm 6.3]{haettel2021lattices}. Thus $X$ is an injective metric space \cite[Thm 1.15]{haettel2021lattices}, hence is contractible.
\end{remark}

A simplicial complex $X$ of type $S$ is \emph{$\widetilde C$-like}, if it is simply-connected and satisfies all the assumptions of Theorem~\ref{thm:contractibleII}. Examples of $\widetilde C$-like complexes includes Coxeter complex or Artin complex of type $\widetilde C_n$, and Euclidean building of type $\widetilde C_n$ \cite{haettel2021lattices}.

\begin{lem}[{\cite[Lem 5.1]{huang2024}}]
	\label{lem:big lattice}
	Suppose $X$ and $(V,\le)$ satisfy the assumptions of Theorem~\ref{thm:contractibleII}.  
	Then $(V,\le)$ is bowtie free and flag.
\end{lem}

\section{Normal forms and combinatorial convexity on orthoscheme complexes}
\label{sec:nf}

%\todo{add a summarizing paragraph here for the whole section}
%In this section we introduce a local normal form for certain $\widetilde C_n$-like complexes, prove a Bestvina-type inequality for such normal forms, and use the normal form to define a notion of convex subcomplexes. In Section~\ref{subsec:biHelly} we review work of first named author and Munro on direct geodesics in bi-Helly graphs. In Section~\ref{}, we explain why such direct geodesics give normal forms in $\widetilde C_n$-like complexes with special properties. Applications of normal forms are in Section~\ref{subsec:bestvina} and Section~\ref{subsec:convex subcomplexes}.

\subsection{Bi-Helly graphs and normal forms}
\label{subsec:biHelly}
In this section we summarize recent work of the first named author and Munro on directed geodesics in bi-Helly graphs.  We will need these results in later sections in order to define normal forms in orthoscheme complexes.

We begin with some definitions.  Let $\Gamma$ be a bipartite simplicial graph.  For a subset $S \subset V(\Gamma)$ we use the notation $B(S,n)$ to indicate the $n$-neighborhood of $S$:  \[ B(S,n) = \{v \in V(\Gamma) : \text{$d(v,u) \le n$ for some $u \in S$} \}. \]
The \emph{half-ball} in $\Gamma$ of radius $n \in \N$ centered at a vertex $u \in V(\Gamma)$ is the set \[\HB(u,n) = \{v \in V(\Gamma) : \text{$d(v,u) \le n$ and $d(v,u) \equiv n \pmod 2$} \}.\]
Note that $\HB(u,n)$ is the intersection of the metric ball $B(u,n)$ with one of the two parts of the bipartition of $V(\Gamma)$.

\begin{definition}[\cite{Bandelt_Dahlmann_Schutte:Absolute_Retract_Bipartite:1987}]
	\label{def:bihelly}
	A bipartite simplicial graph $\Gamma$ is \emph{bi-Helly} if any pair-wise intersecting collection of half-balls $\bigl\{\HB(u_{\alpha},k_{\alpha})\bigr\}_{\alpha}$ has a common vertex.
\end{definition}
Bi-Helly graphs were first introduced by Bandelt, D\"ahlmann and Sch\"utte as \emph{absolute bipartite retracts} \cite{Bandelt_Dahlmann_Schutte:Absolute_Retract_Bipartite:1987}.  They are a bipartite analog of Helly graphs with which they share many similar properties \cite{Bandelt_Farber_Hell:AbsoluteReflexiveBipartiteRetract:1993,CCGHO:Helly:2025}.

The following definitions and results on \emph{directed geodesics} in bi-Helly graphs are the work, in preparation, of the first named author and Munro \cite{HodaMunro:BiHelly_Directed_Geodesics:2026}, inspired by analogous definitions for systolic complexes of Januszkiewicz and {\'S}wi{\c{a}}tkowski \cite{JanuszkiewiczSwiatkowski2006}.

A \emph{near-clique} $K$ in a bi-Helly graph $\Gamma$ is a non-empty collection of vertices that are pair-wise at distance $2$, i.e. $d(u,u') = 2$ for all $u,u' \in K$.  Two near-cliques $K$ and $K'$ are at \emph{uniform distance $n$} if $d(u,u') = n$ for all $u \in K$ and $u' \in K'$.  The \emph{residue} of a near clique $K$ is the set $\Res(K) = \{v \in V(\Gamma) : \text{$d(v,u) = 1$ for all $u \in K$}\}$, i.e., the residue of $K$ is the set of all vertices that are adjacent to every vertex of $K$.

\begin{definition}[\cite{HodaMunro:BiHelly_Directed_Geodesics:2026}]
	\label{def:directed geodesics}
	A sequence of near-cliques $K_0, K_1, \dots, K_n$ of length $n \ge 2$ in a bi-Helly graph $\Gamma$ is a \emph{directed geodesic} if
	\begin{enumerate}
		\item $K_0$ and $K_n$ are at uniform distance $n$; and
		\item $K_i = \Res(K_{i-1}) \cap B(K_n,n-i)$ for all $i = 1, 2, \dots, n-1$.
	\end{enumerate}
\end{definition}

% An inductive argument following from the bi-Helly property ensures that, for any ordered pair of near-cliques at uniform distance, there exists a unique directed geodesic.  Indeed, by induction, the near-clique $K_{i-1}$ is at uniform distance $n-i+1$ and so \[\Res(K_{i-1}) \cap B(K_n,n-i) = \bigcup_{u \in K_n} \Res(K_{i-1}) \cap B(u,n-i) = \bigcup_{u \in K_n}\bigcap_{v \in K_{i-1}}\HB(v,1) \cap \HB(u,n-i)\]
% It follows that any path $v_0, v_1, \dots, v_n$ in $\Gamma$ for which $v_i \in K_i$, for all $i$, is a geodesic.

\begin{lem}[\cite{HodaMunro:BiHelly_Directed_Geodesics:2026}]
	\label{lem:existence and uniqueness of direct geodesic}
	For every pair of near-cliques $K,K'$ at uniform distance $n$ in a bi-Helly graph such that $K'$ is a singleton, there exists a unique directed geodesic $K = K_0, K_1, \dots, K_n = K'$.
\end{lem}

\begin{lem}[\cite{HodaMunro:BiHelly_Directed_Geodesics:2026}]
	Let $K_0, K_1, \dots, K_n$ be a directed geodesic in a bi-Helly graph $\Gamma$ and let $v_0, v_1, \dots, v_n$ be a path in $\Gamma$ satisfying $v_i \in K_i$ for all $i$.  Then $v_0, v_1, \dots, v_n$ is a geodesic.
\end{lem}

\begin{thm}[Local characterization of directed geodesics \cite{HodaMunro:BiHelly_Directed_Geodesics:2026}]
	\label{thm:local to global}
	Let $K_0, K_1, \dots, K_n$ be a sequence of near-cliques of length $n \ge 2$ in a bi-Helly graph with $K_{n-1}$ and $K_n$ having uniform distance 1. Then this sequence is a directed geodesic if and only if, for every $i = 1,2,\dots, n-1$, the subsequence $K_{i-1}, K_i, K_{i+1}$ is a directed geodesic, i.e.,
	\[ \text{$K_{i-1}$ and $K_{i+1}$ are at uniform distance $2$} \]
	and
	\begin{equation}
		\label{eq:local}
		K_i = \Res(K_{i-1}) \cap B(K_{i+1},1).
	\end{equation}
\end{thm}

\subsection{Bi-Helly subgraphs of $\widetilde C$-like complexes}
Let $S=\{s_1,\ldots,s_n\}$ with $s_1<s_2<\cdots<s_n$.
Let $X$ be a $\widetilde C$-like simplicial complex of type $S$. A vertex of $X$ is \emph{extremal} if it is of type $\hat s_1$ or type $\hat s_n$.

The following proposition, whose proof relies on Theorem~\ref{thm:contractibleII} of Haettel ensures that the subgraph of extremal vertices of $X$ is bi-Helly.

\begin{prop}
	\label{prop:isometric}
	Let $X$ be $\widetilde C$-like, let $\Gamma$ be the induced subgraph of $X^1$ spanned on extremal vertices of $X$ and let $\Y$ be the Helly thickening of $X^1$ described in Theorem~\ref{thm:contractibleII}.
	
	Then $\Gamma$ is bi-Helly and the embedding $\Gamma \to \Y$ is isometric with respect to the path metric where edges have length $1$.  In particular, the embedding $\Gamma \to X^1$ is also isometric.
\end{prop}

\begin{proof}
	We begin by proving that the embedding $\Gamma \to \Y$ is isometric.  Recall that $\Y$ has the same vertex set as $X^1$ and two distinct vertices of $\Y$ are joined by an edge if and only if they have a common lower bound of type $\hat s_1$ and a common upper bound of type $\hat s_n$.  Note that these upper and lower bounds are vertices of $\Gamma$.
	
	Because $\Gamma$ embeds in $\Y$, it suffices to show that for any geodesic $x_0, x_1, \dots, x_k$ of $\Y$ between vertices of $x_0, x_k \in \Gamma$, there exists a path of length at most $k$ in $\Gamma$ between $x_0$ and $x_k$.  Without loss of generality, we may assume that $x_0$ is of type $\hat s_1$.  For each $i = 1, 2, \dots k$, there exist vertices $y_i,z_i \in \Gamma$ such that $y_i \le x_{i-1} \le z_i$ and $y_i \le x_i \le z_i$.  Then, for $i = 1, 2, \dots, k-1$, we have $y_i \le x_i \le z_{i+1}$ and $z_i \ge x_i \ge y_{i+1}$ so that $y_i$ and $z_{i+1}$ are adjacent and $z_i$ and $y_{i+1}$ are adjacent in $\Gamma$.  Thus there exists a path in $\Gamma$ of length $k-1$ from $y_1$ to either $y_k$ or $z_k$, depending on the parity of $k$.  But $x_0$ is of type $\hat s_1$ and $x_k$ is either of type $\hat s_1$ or $\hat s_n$ so $x_0 = y_1$ and either $x_k = y_k \le z_k$ or $x_k = z_k \ge y_k$.  Therefore there is a path of length at most $k$ from $x_0$ to $x_k$ in $\Gamma$.  This concludes the proof that $\Gamma \to \Y$ is an isometric embedding.
	
	We now use the isometric embeddedness of $\Gamma$ in the Helly graph $\Y$ to prove that $\Gamma$ is bi-Helly.  Let $\{\HB_{\Gamma}(u_{\alpha},k_{\alpha})\}_{\alpha}$ be a pairwise intersecting family of half-balls in $\Gamma$.  Then, since each half-ball is contained in a part of the bipartition of $\Gamma$, the union $\bigcup_{\alpha} \HB_{\Gamma}(u_{\alpha},k_{\alpha})$ must also be contained in one part of the bipartition.  Thus, without loss of generality, we may assume that every vertex of $\bigcup_{\alpha} \HB_{\Gamma}(u_{\alpha},k_{\alpha})$ is of type $\hat s_1$.  Since $\Gamma$ is isometrically embedded in $\Y$, we have $\HB_{\Gamma}(u_{\alpha},k_{\alpha}) \subset B_{\Gamma}(u_{\alpha},k_{\alpha}) \subset B_\Y(u_{\alpha},k_{\alpha})$, for each $\alpha$, where $B_{\Gamma}(\cdot,\cdot)$ and $B_\Y(\cdot,\cdot)$ denote metric balls in the vertex sets of $\Gamma$ and $\Y$, respectively, with the graph path metric.  Thus $\{B_\Y(u_{\alpha},k_{\alpha})\}_{\alpha}$ is a pairwise intersecting family of balls in $\Y$ so that, since $\Y$ is Helly, there is a common vertex $x \in \bigcap_{\alpha} B_\Y(u_{\alpha},k_{\alpha})$.  Let $v$ be any vertex of type $\hat s_1$ satisfying $v \le x$.  We will show that, for every $\alpha$, there exists a path in $\Gamma$ of length at most $k_{\alpha}$ from $v$ to $u_{\alpha}$.
	
	Let  $x = x_0,x_1,\dots,x_k=u_{\alpha}$ be a geodesic in $\Y$ from $x$ to $u_{\alpha}$.  Then $k \le k_{\alpha}$ and, for each $i = 1, 2, \dots k$, there exist vertices $y_i,z_i \in \Gamma$ such that $y_i \le x_{i-1} \le z_i$ and $y_i \le x_i \le z_i$.  Thus, by transitivity, we have a path $v,z_1,y_2,z_3,\ldots,w_k$ of length $k$ in $\Gamma$ where $w_k = y_k$ if $k$ is even and $w_k = z_k$ if $k$ is odd.  Since $u_{\alpha}$ is a vertex of $\Gamma$, either $u_{\alpha}$ has type $\hat s_1$ and $u_{\alpha} = y_k \le z_k$ or $u_{\alpha}$ has type $\hat s_n$ and $u_{\alpha} = z_k \ge y_k$.  Thus if either $k$ is even and $u_{\alpha}$ has type $\hat s_1$ or $k$ is odd and $u_{\alpha}$ has type $\hat s_n$ then there is a path of length at most $k \le k_{\alpha}$ from $v$ to $u_{\alpha}$.  So we may assume that either $k$ is even and $u_{\alpha}$ has type $\hat s_n$ or $k$ is odd and $u_{\alpha}$ has type $\hat s_1$.  Notice that $k_{\alpha}$ is even if and only if $u_{\alpha}$ and the vertices of $\HB_{\Gamma}(u_{\alpha},k_{\alpha})$ have the same type.  It follows that, since $\HB_{\Gamma}(u_{\alpha},k_{\alpha})$ consists of vertices of type $\hat s_1$, we see that $k$ and $k_{\alpha}$ do not have the same parity in either case.  So $k < k_{\alpha}$ and $v,z_1,y_2,z_3,\ldots,w_k,u_{\alpha}$ is a path of length $k+1 \le k_{\alpha}$ from $v$ to $u_{\alpha}$ in $\Gamma$.
	
	Thus $d(u_{\alpha},v) \le k_{\alpha}$.  Since $v$ and the vertices of $\HB_{\Gamma}(u_{\alpha},k_{\alpha})$ are all of type $\hat s_1$ (i.e. in the same part of the bipartition of $\Gamma$) we also have $d(u_{\alpha},v) \equiv k_{\alpha} \pmod 2$ and so $v \in \HB_{\Gamma}(u_{\alpha},k_{\alpha})$.  Then $v \in \bigcap_{\alpha} \HB_{\Gamma}(u_{\alpha},k_{\alpha})$, which completes the proof that $\Gamma$ is bi-Helly.
\end{proof}

We call the subgraph $\Gamma$ of $X^1$ from Proposition~\ref{prop:isometric} the \emph{bi-Helly subgraph} of $X$.

\subsection{Normal forms on orthoscheme complexes}
Let $X$ be a simplicial complex of type $S$ such that the relation on $X^0$ in Definition~\ref{def:order} is a partial order.
We define $X$ to be \emph{locally determined} if for any pair of vertices $x<y$ in $X$, there is an extremal $y'$ with $x<y'$ such that $y'$ and $y$ are not comparable; and there is an extremal vertex $x'$ with $x'<y$ such that $x'$ and $x$ are not comparable. The meaning of locally determined is justified by the following lemma, which is an immediate consequence of the definition and Lemma~\ref{lem:posets}. 
\begin{lem}
	\label{lem:locally determined}
	Suppose $(X^0,<)$ is bowtie free (this holds when $X$ is $\wtC$-like, see Lemma~\ref{lem:big lattice}). If $X$ is locally determined, then for each vertex $x\in X$, $x$ is the meet of all type $\hat s_n$ elements in $X$ that are $\ge x$, and $x$ is the join of all type $\hat s_1$ elements in $X$ that are $\le x$.
\end{lem}

For each vertex $x$, the \emph{lower link} (resp. \emph{upper link}) of $x$ in $X$, denoted by $\lk^-(x,X)$ (resp. $\lk^+(x,X)$), is the full subcomplex of $\lk(x,X)$ spanned by vertices of $X$ that are $<x$ (resp. $>x$). 

An \emph{up-down} path in $X$ is an edge path $x_1x_2\ldots x_n$ in $X$ such that $x_i$ and $x_{i+1}$ are adjacent vertices for $1\le i\le n-1$, and for each $2\le i\le n-1$, $x_{i-1}<x_i$ implies $x_i>x_{i+1}$ and $x_{i-1}>x_i$ implies $x_{i-1}<x_i$. This path is \emph{tight}, if $x_{i-1}<x_i>x_{i+1}$ implies that $x_i$ is the join of $x_{i-1}$ and $x_{i+1}$, and $x_{i-1}>x_i<x_{i+1}$ implies that $x_i$ is the meet of $x_{i-1}$ and $x_{i+1}$.

\begin{definition}
	\label{def:local normal}
	Let $\omega=x_1\ldots x_n$ be an up-down path in $X$. We say $\omega$ is a \emph{local normal form path from $x_1$ to $x_n$} if for each $2\le i\le n-1$ the following holds:
	\begin{enumerate}
		\item if $x_{i-1}<x_i>x_{i+1}$, then there do not exist a sequence of vertices $a_1=x_{i-1},a_2,a_3,a_4=x_{i+1}$ in $\lk^-(x_{i},X)$ with $a_1\le a_2\ge a_3\le a_4$;
		\item if $x_{i-1}>x_i<x_{i+1}$, then there do not exist a sequence of vertices $a_1=x_{i-1},a_2,a_3,a_4=x_{i+1}$ in $\lk^+(x_{i},X)$ with $a_1\ge a_2\le a_3\ge a_4$.
	\end{enumerate}
\end{definition} 

In the case $n=2$ and $\omega$ is a single edge, $\omega$ is automatically a local normal form from $x_1$ to $x_2$ as both two requirements above do not apply.

Given an up-down path $\omega=x_1\cdots x_n$ in $X$, let $K_i$ be the set of all extremal vertices of $X$ that is $\ge x_i$ (resp. $\le x_i$) if $x_i$ has a neighbour in $\omega$ that is $<x_i$ (resp. $>x_i$). 

\begin{prop}
	\label{prop:translate}
	Suppose $X$ is $\widetilde C$-like and locally determined, with its bi-Helly subgraph $\Gamma$.	
	Suppose $\omega=x_1\ldots x_n$ be an up-down path in $X$ which is in local normal form. Then $\{K_{i-1},K_i,K_{i+1}\}$ is a directed geodesic in $\Gamma$ for $1<i<n$.
	
	Conversely, suppose $n\ge 3$ and $\{K_1,\ldots,K_n\}$ is a sequence of near-cliques in $\Gamma$ such that for each $1\le i\le n-1$, $K_i$ and $K_{i+1}$ form a complete bipartite subgraph in $\Gamma$; and $\{K_{i-1},K_i,K_{i+1}\}$ is a directed geodesic in $\Gamma$ for $1<i<n$. Let $x_i\in X$ be the join (resp. meet) of all elements in $K_i$ if elements of $K_i$ are of type $\hat s_1$ (resp. $\hat s_n$). Assume that $K_1,K_n$ are \emph{full}: each
	consists of all extremal vertices of its type comparable with the corresponding vertex
	$x_1,x_n$.
	Then $x_1\ldots x_n$ is a locally normal form path from $x_1$ to $x_n$.
\end{prop}

Note that the $x_i$ in the second paragraph of the proposition exists by Lemma~\ref{lem:big lattice} and Lemma~\ref{lem:posets}.

\begin{proof}
	Suppose $x_i$ has a neighbour in $\omega$ that is $<x_i$. If $K_{i-1},K_i,K_{i+1}$ fails \eqref{eq:local}, then there is a vertex $z\notin K_i$ of type $\hat s_n$ such that $z$ is adjacent to each vertex in $K_{i-1}$ and a vertex $a_3\in K_{i+1}$. As $a_3\le x_{i+1}$ and $x_{i+1}< x_i$, we have $a_3< x_{i}$. As $X$ is locally determined, $z\ge x_{i-1}$. Then both $z$ and $x_i$ are common upper bounds for $\{x_{i-1},a_3\}$. Let $a_2$ be the join of $x_{i-1}$ and $a_3$, which exists by Lemma~\ref{lem:big lattice}. Then $a_2\le x_i$. Note that $a_2\neq x_i$, otherwise $x_i\le z$ which contradicts $z\notin K_i$. Thus $a_2<x_i$ and $a_2\in \lk^-(x_i,X)$. By construction, $x_{i-1}\le a_2\ge a_3\le x_{i+1}$, contradicting that $\omega$ is locally normal at $x_i$.  It remains to show $K_{i-1}$ and $K_{i+1}$ are at uniform distance. From the definition of each $K_i$, it suffices to show $K_{i-1}\cap K_{i+1}=\emptyset$. Suppose we can find $a\in K_{i-1}\cap K_{i+1}$, then $a\le x_{i-1}<x_i$ and $a\le x_{i+1}<x_i$, which gives $x_{i-1}\le x_{i-1}>a<x_{i+1}$ in $\lk^-(x_i,X)$, contradiction. Thus $K_{i-1}\cap K_{i+1}=\emptyset$, and this finishes the proof that $\{K_{i-1},K_i,K_{i+1}\}$ is a directed geodesic in $\Gamma$.
	The case $x_i$ has a neighbour in $\omega$ that is $>x_i$ is similar.
	
	For the converse, as $K_i$ and $K_{i+1}$ span a complete bipartite subgraph, $x_i$ and $x_{i+1}$ are comparable. We first show that each interior $K_i$ is also full. Indeed, if $K_i$ is made of minimal vertices, then $x_i$ is a common lower bound for $K_{i-1}\cup K_{i+1}$, hence any minimal vertex $\le x_i$ is contained in Res$(K_{i-1}\cap B(K_{i+1},1))=K_i$, as desired. The case of $K_i$ being maximal is similar. 
	Note that $x_i\neq x_{i+1}$, otherwise $x_{i-1}$ is comparable to $x_{i+1}$ and $K_{i-1}\cap K_{i+1}\neq\emptyset$ by the fullness of $\{K_i\}_{i=1}^n$, contradiction to that $K_{i-1}$ and $K_{i+1}$ having uniform distance 2. Thus $x_1\ldots x_n$ forms an up-down path in $X$. Now we show this path is locally normal from $x_1$ to $x_n$. We only consider the case of $x_{i-1}<x_i>x_{i+1}$ as the other case is similar. Suppose there is a sequence of vertices $a_1=x_{i-1},a_2,a_3,a_4=x_{i+1}$ in $\lk^-(x_i,X)$ with $a_1\le a_2\ge a_3\ge a_4$. As $a_2<x_i$ and $X$ is locally determined, there is $z$ of type $\hat s_n$ such that $z\notin K_i$ and $a_2\le z$. Then $x_{i-1}\le z\ge a_3$. Let $y\in X$ be a vertex of type $\hat s_1$ such that $y\le a_3$ (such $y$ exists as $X$ is a simplicial complex of type $S$). Then $x_{i-1}\le z\ge y\le x_{i+1}$. Thus $z$ is adjacent to each vertex of $K_{i-1}$, and $z$ is adjacent to $y\in K_{i+1}$. As $z\notin K_i$, this contradicts \eqref{eq:local}.
\end{proof}

\begin{definition}
	\label{def:normal}
	Let $\omega=x_1\ldots x_n$ be an up-down path in $X$. Let $d$ be the path metric on $X^1$ with unit edge length. We say $\omega$ is a \emph{normal form path} from $x_1$ to $x_n$, if $\omega$ is geodesic in $X^1$, and for each $1\le i\le n-1$ and any vertex $y$ of $X$ adjacent to $x_i$ with $d(y,x_n)=n-i-1$, we must have either $x_i<x_{i+1}\le y$ or $x_i>x_{i+1}\ge y$. 
\end{definition}

\begin{prop}
	\label{prop:local vs global}
	Suppose $X$ is $\widetilde C$-like and locally determined.
	If up-down path $\omega=x_1\ldots x_n$ is in locally normal form $x_1$ to $x_n$, then it is in normal form from $x_1$ to $x_n$.
\end{prop}

\begin{proof}	
	%	Suppose $\omega$ is normal. By contradiction we assume it fails to be locally normal at $x_i$. We consider the case $x_{i-1}<x_i>x_{i+1}$ (the other case is similar). Let $\{a_i\}_{i=1}^4$ be as in Definition~\ref{def:local normal} (1). As $x_{i+1}<x_{i+2}$ and $a_3\le x_{i+1}$, $a_3<x_{i+2}$. So $x_{i-1}a_2a_3x_{i+2}\ldots x_n$ is geodesic. However, $a_2<x_i$ contradicts Definition~\ref{def:normal} (with the role of $y$ played by $a_2$). Thus $\omega$ is locally normal.\todo{It seems this only works when $i\le n-2$.}
	We assume $n\ge 3$, otherwise the proposition is clear.
	Suppose $\omega$ is a local normal path from $x_1$ to $x_n$.
	We first show $\omega$ is geodesic.	
	We define $\{K_1,\ldots,K_n\}$ as before, which satisfies \eqref{eq:local} by Proposition~\ref{prop:translate}. By Theorem~\ref{thm:local to global}, $\{K_1,\ldots,K_n\}$ is a directed geodesic in the bi-Helly subgraph $\Gamma$ in the sense of Definition~\ref{def:directed geodesics}, and $d_\Gamma(k_1,k_n)=n-1$ for any $k_1\in K_1,k_n\in K_n$, where $d_\Gamma$ denotes the path metric measured in $\Gamma$. For each $i$, take $\beta_i\in K_i$.
	Proposition~\ref{prop:isometric} implies $d_X(\beta_1,\beta_n)=n-1$ with $d_X$ being the path metric measured in $X^1$, hence $d_X(x_1,\beta_n)\ge n-2$. 
	
	We claim $d_X(x_1,\beta_n)=n-1$ and $d_X(\beta_1,x_n)=n-1$. By symmetry, we only prove the first statement.
	If $d(x_1,\beta_n)=n-2$, then consider a shortest edge path $\omega_1$ from $x_1$ to $\beta_n$ in $X^1$.  Let $z$ be the vertex in $\omega_1$ adjacent to $x_1$. Note that $\omega_1$ is an up-down path. By replacing vertices of $\omega_1$ by appropriate extremal vertices, we can assume in addition that all vertices of $\omega_1$ are extremal except $x_1$.
	As $\beta_n$ is extremal and $\beta_2\ldots \beta_n$ is an up-down path as well, $x_1<z$ if and only if $\beta_2<\beta_3$ if and only if $x_2<x_3$. Thus $z\in K_1$, which contradicts that $d_X(z,\beta_n)=n-3$. Hence the claim follows. 
	
	Now we prove $d_X(x_1,x_n)=n-1$. We only treat the case $x_1<x_2$, as the other case is similar.
	By previous paragraph, $d_X(x_1,x_n)\ge n-2$. If $d_X(x_1,x_n)=n-2$, then let $\omega_2$ be shortest edge path connecting $x_1$ and $x_n$. As before, we can assume $\omega_2$ is an up-down path with all its interior vertices being extremal. Let $y_2$ be the vertex of $\omega_2$ adjacent to $x_1$. If $y_2<x_1$, then $y_2\in K_1$. As $d(y_2,x_n)=n-3$, this contradicts $d_X(\beta_1,x_n)=n-1$ as $\beta_1$ can be any element in $K_1$. If $y_2>x_1$, then by comparing the up-down path $\omega_2$ and $x_1x_2\cdots x_{n-1}$, we know $z<x_n$ if and only if $x_{n-2}<x_{n-1}$ if and only if $x_{n-1}>x_n$ where $z$ is the neighbour of $x_n$ in $\omega_2$. Thus $z\in K_n$ and $d_X(x_1,z)=n-3$. However, the previous paragraph implies that $d_X(x_1,z)=n-1$, contradiction. Thus $\omega$ is geodesic.
	
	It remains to prove $\omega$ is normal. We only verify the case $i=1$ in Definition~\ref{def:normal}, as the other cases are similar. As before, we will assume $x_1<x_2$.
	Take vertex $x\in X$ adjacent to $x_1$ such that $d_X(x,x_n)=n-2$. Then there is a geodesic $\omega'$ from $x_1$ to $x_n$ passing through $x$. As $\omega'$ is distance minimizing, it is an up-down path. If $x_1>x$, then $x<z$ where $z$ is the other neighbour of $x$ in $\omega'$. We take $\beta$ to be an extremal vertex $\le x$. Then $x_1>\beta<z$. Hence $\beta\in K_1$, and $d_X(\beta,x_n)=n-2$, which contradicts the above claim. Thus $x_1< x$. Let $K_x$ be the collection of extremal vertices that are $\ge x$. Then $x_1<\alpha>z$ for any $\alpha\in K_x$. Thus $\alpha$ is adjacent to each vertices of $K_1$ and $d_\Gamma(\alpha,\beta_n)=n-2$ for each $\alpha\in K_x$. By Definition~\ref{def:directed geodesics} (2), $K_x\subset K_2$. As $X$ is locally determined, by Lemma~\ref{lem:locally determined} we have $x\ge x_2$, as desired.
\end{proof}

\begin{prop}
	\label{prop:existence and uniqueness}
	Suppose $X$ is $\widetilde C$-like and locally determined. Then for any vertices $x,y\in X$ such that $y$ is extremal, there is a unique local normal form path from $x$ to $y$. Thus there is a unique normal form path from $x$ to $y$. 
\end{prop}

\begin{proof}
	Suppose $d_X(x,y)=n$. We assume $n\ge 2$, otherwise the proposition is clear.
	As any geodesic in $X^1$ from $x$ to $y$ is an up-down path and $y$ is extremal, we can assume the neighbour of $x$ in any such geodesic is $>x$ (up to possibly reversing the total order on $S$). Let $z$ be one such neighbour of $x$. Then $d_X(z,y)=n-1$.
	Let $K_x$ be the collection of all extremal vertices that are $\le x$. As $x<z$, $\alpha<z$ for any $\alpha\in K_x$. Thus $d_X(\alpha,y)\le n$ for any $\alpha\in K_x$ by considering the edge path $\omega$ which is the concatenation of $\alpha z$ with a geodesic from $z$ to $y$. Hence $d_\Gamma(\alpha,y)\le n$ by Proposition~\ref{prop:isometric}, where $\Gamma$ is the bi-Helly subgraph of $X$. Note that $\omega$ is an up-down path, so we can replace interior vertices of $\omega$ by extremal vertices of $X$ to obtain an edge path in $\Gamma$ of length $n$ from $\alpha$ to $y$.
	If $d_\Gamma(\alpha,y)<n$, then $d_\Gamma(\alpha,y)\le n-2$ as $\Gamma$ is bipartite. Thus $d_X(\alpha,y)\le n-2$ by Proposition~\ref{prop:isometric}. However, this contradicts that $d_X(x,y)=n$ and $d(\alpha,x)=1$. So $d_\Gamma(\alpha,y)=n$ for any $\alpha\in K_x$. By Lemma~\ref{lem:existence and uniqueness of direct geodesic},  there is a directed geodesic form $\{K_1,\ldots,K_n\}$ in $\Gamma$ from $K_1=K_x$ and $K_n=\{y\}$, which also satisfies \eqref{eq:local} by Theorem~\ref{thm:local to global}. As $K_n$ is a singleton, by Definition~\ref{def:directed geodesics}, $K_i$ and $K_{i+1}$ span a complete bipartite subgraph for $1\le i\le n-1$. This gives a local normal form path from $x$ to $y$ by Proposition~\ref{prop:translate}. The uniqueness of local normal form path follows from Proposition~\ref{prop:translate}, Theorem~\ref{thm:local to global}, the uniqueness in Lemma~\ref{lem:existence and uniqueness of direct geodesic} and $X$ being locally determined (Lemma~\ref{lem:locally determined}). The existence of normal form path follows from Proposition~\ref{prop:local vs global}, and the uniqueness follows from definition.
\end{proof}

\subsection{Convexity of combinatorial distance function}

\begin{lem}
	\label{lem:strip}
	Suppose $X$ is a simplicial complex of type $S$ which is $\widetilde C$-like and locally determined. Let $\omega=x_n\ldots x_1$ and $\omega'=y_m\ldots y_1$ be normal form paths from $x_n$ to $x_1$ and $y_m$ to $y_1$ respectively. Suppose $x_1=y_1$ is extremal, $n\le m$ and $x_n$ and $y_m$ are adjacent in $X$. Then  either $x_i\le y_i$ for $1\le i\le n$ or $x_i\ge y_i$ for $1\le i\le n$.
\end{lem}

\begin{proof}
	First suppose $n=m$. If $x_{n-1}<x_n$, then $y_{n-1}<y_n$ since both $\omega$ and $\omega'$ are up-down paths and $x_1=y_1$ is extremal. Suppose $x_n\le y_n$. Then $x_{n-1}<y_n$. As $y_m\ldots y_1$ is a normal form and $d(x_{n-1},y_1)=n-2$, by Definition~\ref{def:normal}, $x_{n-1}\le y_{n-1}$. As $y_{n-1}<y_{n-2}$, $x_{n-1}<y_{n-2}$. As $x_{n-1}\ldots x_1$ is a normal form and $d(x_1,y_{n-2})=n-3$, by Definition~\ref{def:normal} $x_{n-2}\le y_{n-2}$. Repeating this argument we obtain $x_i\le y_i$ for $1\le i\le n$. Similarly, if $x_n\ge y_n$, then $x_i\ge y_i$ for $1\le i\le n$. The case of $x_{n-1}>x_n$ is similar.
	
	Now suppose $m=n+1$. Suppose $y_n<y_{n+1}$. As $y_{n+1}\ldots y_1$ is a normal form and $d(x_n,y_1)=n-1$, Definition~\ref{def:normal} implies that $x_n<y_{n+1}$ and $x_n\le y_n$. Now we are reduced to the previous paragraph and deduce $x_i\le y_i$ for $1\le i\le n$. The case $y_n>y_{n+1}$ is similar.
\end{proof}

\begin{prop}
	\label{prop:bestvina}
	Suppose $X$ is a simplicial complex of type $S$ which is $\widetilde C$-like and locally determined.
	Let $\omega=x_1\ldots x_n$ be a local normal form path in $X$ from $x_1$ to $x_n$. Let $z$ be an extremal vertex in $X$. Then the function $f(i)=d(z,x_i)$ first decreases, then strictly increases. 
\end{prop}	
The decreasing part of $f$ is not necessarily strict. Also one of the decreasing part and the increasing part of $f$ could be trivial. %In particular, the proposition does not exclude the situation that $f$ is the constant function (this can indeed happen by considering appropriate configuration in Coxeter complex of type $\widetilde C_2$).
\begin{proof}
	%	By Proposition~\ref{prop:local vs global}, $\omega$ is a local normal form path from $x_1$ to $x_n$.
	We assume $z\notin \omega$, otherwise the proposition is clear. We claim there does not exist $i$ such that $d(x_{i-1},z)<d(x_i,z)>d(x_{i+1},z)$. First we consider the case $x_{i-1}<x_i>x_{i+1}$. Let $\omega$ be a normal form path from $x_i$ to $z$ which exists by Proposition~\ref{prop:existence and uniqueness}, and let $y_i$ be the neighbour of $x_i$ in $\omega$.  Then Definition~\ref{def:normal} implies $x_{i-1}\le y_i\ge x_{i+1}$, contradicting Definition~\ref{def:local normal} (1) (note that $x_{i-1}\le y_i\ge x_{i+1}\le x_{i+1}$ play the role of $a_1,a_2,a_3,a_4$). The case of $x_{i-1}>x_i<x_{i+1}$ is similar.
	
	We claim there does not exist $i$ with $d(x_{i-1},z)<d(x_i,z)=d(x_{i+1},z)$. Suppose $x_{i-1}<x_i>x_{i+1}$. Let $a_2$ (resp. $a_3$) be neighbour of $x_{i}$ (resp. $x_{i+1}$) in a normal form path from $x_{i}$ (resp. $x_{i+1}$) to $z$. By Definition~\ref{def:normal}, $a_2<x_i$ and $a_2\ge x_{i-1}$. As $d(x_{i+1},z)=d(x_i,z)$ and $z$ is extremal, $a_2<x_i$ implies $a_3<x_{i+1}$. Thus $a_3<x_i$. By Lemma~\ref{lem:strip}, $x_{i}>x_{i+1}$ implies $a_2\ge a_3$. So in $\lk^-(x_i,X)$ we have $x_{i-1}\le a_2\ge a_3< x_{i+1}$, contradicting Definition~\ref{def:local normal} (1). The case of $x_{i-1}>x_i<x_{i+1}$ is similar. 
	
	These two claim imply that if $f(i)<f(i+1)$ for some $i$, then $f(j)<f(j+1)$ for any $j\ge i$. Thus the proposition follows.
\end{proof}

\subsection{A notion of convex subcomplexes}

\begin{definition}
	\label{def:local convex}
	Let $X$ be a simplicial complex of type $S=\{s_1,\ldots,s_n\}$ which is $\widetilde C$-like and locally determined. We say a full subcomplex $Y\subset X$ is \emph{locally convex}, if 
	\begin{enumerate}
		\item $Y$ is also a simplicial complex of type $S$;
		\item for any vertex $y\in Y$, any tight up-down edge path of length $2$ in $\lk^+(y,X)$ with endpoints in $\lk^+(y,Y)$ is contained in $\lk^+(y,Y)$, and any tight up-down edge path of length $2$ in $\lk^-(y,X)$ with endpoints in $\lk^-(y,Y)$ is contained in $\lk^-(y,Y)$;
		\item for any vertex $y\in Y$, any tight up-down edge path $a_1<a_2>a_3<a_4$ in $\lk^-(y,X)$ with $a_1,a_4\in\lk^-(y,Y)$ is contained in $\lk^-(y,Y)$, and any tight up-down edge path $a_1>a_2<a_3>a_4$ in $\lk^+(y,X)$ with $a_1,a_4\in \lk^+(y,Y)$ is contained in $\lk^+(y,Y)$.
	\end{enumerate}
\end{definition}

\begin{prop}
	\label{prop:convex subcomplex}
	Let $X$ be a simplicial complex of type $S=\{s_1,\ldots,s_n\}$ which is $\widetilde C$-like and locally determined. Let $Y\subset X$ be a connected locally determined locally convex subcomplex. Then $Y$ itself is $\widetilde C$-like. Moreover, if $y_1,y_2$ are vertices of $Y$ with a local normal form path $\omega\subset X$ from $y_1$ to $y_2$ with $y_2$ extremal, then $\omega\subset Y$.
\end{prop}

\begin{proof}
	For each vertex $y\in Y$, by Definition~\ref{def:local convex}, the vertex set of $\lk^+(y,Y)$ (resp. $\lk^-(y,Y)$) is closed under taking join and meet in $\lk^+(y,X)$ (resp. $\lk^-(y,X)$). Thus $Y$ satisfies all the requirements of $\widetilde C$-like complexes except simply connectedness. Let $\widetilde Y$ be the universal cover of $Y$ with the induced structure of simplicial complex of type $S$. Then $\widetilde Y$ is $\widetilde C$-like and locally determined. 
	
	Let $\pi:\widetilde Y\to Y$ be the covering map. Given $\omega=y_1\ldots y_n\subset \widetilde Y$ be a local normal path  from $y_1$ to $y_n$, we claim $\bar\omega=\bar y_1\ldots\bar y_n$ is a local normal path in $X$ from $\bar y_1$ to $\bar y_n$, where $\bar y_i=\pi(y_i)$. Suppose $\bar\omega$ is not locally normal at $\bar y_i$ with $\bar y_{i-1}<\bar y_i>\bar y_{i+1}$. Then there are vertices $\bar a_2,\bar a_3\in \lk^-(\bar y_i,X)$ with $\bar y_{i-1}\le \bar a_2\ge \bar a_3\le \bar y_{i+1}$. Up to replacing $\bar a_2$ and $\bar a_3$ by appropriate vertices, we can assume either $\bar y_{i-1}$ and $\bar y_{i+1}$ are connected by a tight up-down path $P$ of length $\le 2$ in $\lk^-(\bar y_i,X)$, or $\bar y_{i-1}<\bar a_2>\bar a_3<\bar y_{i+1}$ gives a tight up-down path $P$ in $\lk^-(\bar y_i,X)$. As $Y$ is locally convex, in either cases we have $P\subset \lk^-(\bar y_i,Y)$. By lifting $P$ to $\lk^-(y_i,\widetilde Y)$, we deduce that $\omega$ is not locally normal at $y_i$, contradiction. The case $\bar y_{i-1}>\bar y_i<\bar y_{i+1}$ is similar. Thus the claim follows. 
	
	We conclude that $\pi$ is the identity map, otherwise there exist two extremal vertices $y\neq y'\in \widetilde Y$ that are mapped to the same vertex $\bar y\in Y$ and the local normal path from $y$ to $y'$ (which exists by Proposition~\ref{prop:existence and uniqueness}) is mapped to a local normal path from $\bar y$ to itself, contradicting Proposition~\ref{prop:local vs global}. Thus $Y$ is simply-connected, hence it is a $\widetilde C$-like complex. The moreover part of the proposition follows from the above claim, and Proposition~\ref{prop:existence and uniqueness}.
\end{proof}

\section{Subdivision of some relative Artin complexes}
\label{sec:subdivision}
We review subdivision of certain types of relative Artin complexes from \cite{huang2024}, adapted to simplicial complexes of type $S$. Let $\Lambda$ be doubly marked diagram with edge sets $E_m\subset E_h$. 

For $n\ge 2$, a \emph{$D_{n+1}$-like} subdiagram $\Lambda'$ of a  diagram $\Lambda$ is an induced subdiagram whose underlying graph is isomorphic to the Coxeter diagram of type $D_{n+1}$ (see Figure~\ref{fig:BD} left). For $n\ge 3$, a $D_{n+1}$-like subdiagram is \emph{$\widetilde B_{n}$-like}, if the  edge $b_nb_{n+1}\in E_m$.
For $n\ge 1$, a \emph{$\widetilde D_{n+3}$-like} subdiagram $\Lambda'$ of $\Lambda$ is an induced subdiagram which is isomorphic as a graph to the Coxeter diagram of type $\widetilde D_{n+3}$ (see Figure~\ref{fig:BD} right).
\begin{definition}
	\label{def:subdivision}
	Let $\Lambda$ be a doubly marked diagram with an induced $D_{n+1}$-like subdiagram $\Lambda'$. Let $\{b_i\}_{i=1}^{n+1}$ be vertices of $\Lambda'$ as in Figure~\ref{fig:BD} left. Let $\Delta_\Lambda$ be a regular simplicial complex of type $S=V\Lambda$ modeled on $\Lambda$.
	
	Let $\Delta=\Delta_{\Lambda,\Lambda'}$ be the associated relative subcomplex. We subdivide each edge of $\Delta$ connecting a vertex of type $\hat b_1$ and a vertex of type $\hat b_2$. We say the middle point of such edge is of type $m$. Cut each top dimensional simplex in $\Delta$ into two simplices along the codimensional 1 simplex spanned by vertices of type $m$ and $\{\hat b_i\}_{i=3}^{n+1}$. This gives a new simplicial complex, which we denoted by $\Delta'$. Define a map $\tau$ from the vertex set $V\Delta'$ of $\Delta'$ to $\{1,2,\ldots,n,n+1\}$ by sending vertices of type $\hat b_1,\hat b_2$ to $1$, vertices of type $m$ to $2$, vertices of type $\hat b_i$ to $i$ for $i\ge 3$. We then view $\Delta'$ as a simplicial complex of type $S=\{1,2,\ldots,n+1\}$ and define a relation $<$ on $V\Delta'$ as follows. For $x,y\in V\Delta'$, $x<y$ if $x$ and $y$ are adjacent and $\tau(x)<\tau(y)$. The simplicial complex $\Delta'$, together with the relation $<$ on its vertex set, is called the \emph{$(b_1,b_2)$-subdivision of $\Delta_{\Lambda,\Lambda'}$}. 
\end{definition}

\begin{figure}[h]
	\centering
	\includegraphics[scale=1]{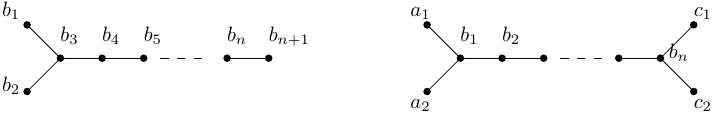}
	\caption{$\widetilde B$-like and $\widetilde D$-like subdiagrams.}
	\label{fig:BD}
\end{figure}

When  $\Lambda'$ is an admissible subgraph of $\Lambda$, the relation $<$ is a partial order (by the same argument as in \cite[Lem 8.2]{huang2024}).
The $(b_1,b_2)$-subdivision $\Delta'$ of $\Delta$ is \emph{upward or downward flag or bowtie free}, if $(V\Delta',<)$ is a poset which is upward or downward flat or bowtie free.

\begin{lem}
	\label{lem:subdivision bowtie free}
	Let $\Delta_\Lambda$ be a regular simplicial complex of type $V\Lambda$ modeled on $\Lambda$ ($V\Lambda$ is the vertex set of a diagram $\Lambda$). Suppose $\Delta_\Lambda$ satisfies C$_4$.
	Let $\Lambda'$ be an induced admissible $D_{n+1}$-like subdiagram of $\Lambda$. If $\Delta_{\Lambda,\Lambda'}$ satisfies the labeled 4-cycle condition, then the $(b_1,b_2)$-subdivision of $\Delta_{\Lambda,\Lambda'}$ is bowtie free.
\end{lem}

\begin{proof}
	Let $\Delta'$ be the $(b_1,b_2)$-subdivision of $\Delta_{\Lambda,\Lambda'}$, with the poset structure on its vertex set as in Definition~\ref{def:subdivision}. Let $\mathcal P$ be the poset $((\Delta')^0,<)$, with function $\tau$ defined as above. For any $p\in \mathcal P$ with $\tau(p)=1$, then $\mathcal P_{>p}$ can be identified with the vertex set of $\Delta^p_{\Lambda\setminus\{b_i\},\Lambda'\setminus\{b_i\}}$ for $i=1$ or $2$. Lemma~\ref{lem:4wheel} implies $\Delta_{\Lambda,\Lambda'\setminus\{b_i\}}$ is bowtie free. Lemma~\ref{lem:subdiagramsgeneral} implies that $\Delta^p_{\Lambda\setminus\{b_i\},\Lambda'\setminus\{b_i\}}$ is bowtie free for $i=1,2$. Given a quasi-bowtie $x_1y_1x_2y_2$ in $\mathcal P$ with $1=\tau(x_1)=\tau(x_2)<\tau(y_1)=\tau(y_2)$. If $\tau(y_1)=\tau(y_2)=2$, then either $x_1=x_2$ or $y_1=y_2$. If $\tau(y_1)=\tau(y_2)\ge 3$, then Lemma~\ref{lem:tripod} gives the center for this quasi-bowtie. Now the lemma follows from Lemma~\ref{lem:bowtie free criterion}.
\end{proof}

\begin{lem}
	\label{lem:weakly flag equivalent}
	Suppose $\Delta_\Lambda$ be a regular simplicial complex modeled on a diagram $\Lambda$.
Let $\Lambda'$ be an admissible $D_3$-like subdiagram with its three consecutive vertices being $\{b_1,b_3,b_2\}$.	
	Suppose $\Delta_{\Lambda,\Lambda'}$ is bowtie free. Then $\Delta_{\Lambda,\Lambda'}$ is weakly flag (Definition~\ref{def:flag}) if and only if the $(b_1,b_2)$-subdivision of $\Delta=\Delta_{\Lambda,\Lambda'}$ is downward flag.
\end{lem}

\begin{proof}
	The only if direction follows the same argument as \cite[Lem 8.7]{huang2024}. It remains to prove the if direction. Let $\Delta'$ be the $(b_1,b_2)$-subdivision of $\Delta$ and suppose $\Delta'$ is downward flag. We assume the partial order on $\Delta^0$ is induced by $b_1<b_3<b_2$.
	Given $\{p_1,p_2,p_3\}\in \Delta^0$ of type $\hat b_1$ such that $p_i$ and $p_{i+1}$ have a common upper bound $q_i$ in $(\Delta^0,<)$ of type $\hat b_3$ for $i\in \mathbb Z/3\mathbb Z$. We can assume $\{q_1,q_2,q_3\}$ are pairwise distinct, otherwise we are already done. Then $\{q_1,q_2,q_3\}$ are pairwise lower bounded in $((\Delta')^0,<)$, hence they have a common lower bound $q$ in $((\Delta')^0,<)$, which can be chosen to have type $\hat b_1$ or type $\hat b_2$. If $q$ has type $\hat b_1$, then $q$ and $p_1$ are both common lower bounds for $\{q_1,q_3\}$ in $(\Delta^0,<)$. As $q_1\neq q_3$ and $(\Delta^0,<)$ is bowtie free, $q=p_1$. Similarly, $q=p_1=p_2=p_3$. If $q$ has type $\hat b_2$, then we deduce $p_1\sim q$ from $p_1\sim q_1$, $q_1\sim q$ and Lemma~\ref{lem:transitive}. Similarly, $p_i\sim q$ for $1\le i\le 3$, so $q$ is common upper bound for $\{p_1,p_2,p_3\}$ in $(\Delta^0,<)$, as desired. This shows $(\Delta^0,<)$ is weakly upward flag. Weakly downward flagness can be proved in a similar way.
\end{proof}

\begin{lem}
	\label{lem:d3 upward flag}
	Let $\Lambda'=b_1b_3b_2$ be a $D_3$-like subdiagram of a tree diagram $\Lambda$. Suppose $(\Delta_{\Lambda,\Lambda'})^0$ with the partial order induced by $\hat b_1<\hat b_3<\hat b_2$ is bowtie free and upward flag.
	Then the $(b_1,b_2)$-subdivision of $\Delta_{\Lambda,\Lambda'}$ is bowtie free and downward flag.
\end{lem}

\begin{proof}
Bowtie freeness follows from Lemma~\ref{lem:bowtie free criterion}, and downward flagness reduces to proving $\Delta_{\Lambda,\Lambda'}$ is weakly flag by Lemma~\ref{lem:weakly flag equivalent}. Upward flagness of $\Delta_{\Lambda,\Lambda'}$ implies the weakly flagness of $\Delta_{\Lambda,\Lambda'}$, whose proof is identical to \cite[Lem 12.2]{huang2024}.
\end{proof}

We recall several results from \cite{huang2024} proved for relative Artin complexes. However, their proofs generalize directly to simplicial complexes of type $S$ (the verification of relevant complexes satisfying assumptions of Theorem~\ref{thm:contractibleII} involving checking simply connectedness, which follows from Lemma~\ref{lem:scgeneral}). We need condition C$_4$ in Proposition~\ref{prop:propagation} (3) and Proposition~\ref{prop:propagation2} (3), as the original proof of both in \cite{huang2024} relies on the Artin version of Lemma~\ref{lem:4wheel}.
\begin{prop}(\cite[Prop 2.24]{huang2024})
	\label{prop:ori link0}
	Suppose $\Delta_\Lambda$ is a regular simplicial complex of type $V\Lambda$ modeled on a doubly marked diagram $\Lambda$.
Let $\Lambda'$ be an admissible $\widetilde B_n$-like subdiagram, with $\{b_i\}_{i=1}^{n+1}$ being vertices of $\Lambda'$ as in Figure~\ref{fig:BD} left. Suppose $n\ge 3$.
	For $i=1,2,n+1$, let $\Lambda_{i}$ (resp. $\Lambda'_{i}$) be the connected component of $\Lambda\setminus\{b_i\}$ (resp. $\Lambda'\setminus\{b_i\}$) that contains $b_3$. 
	Suppose that the following holds:
	\begin{enumerate}
		\item for each vertex $v$ of type $\hat b_{n+1}$, the $(b_1,b_2)$-subdivision of the relative complex $\Delta^v_{\Lambda_{n+1},\Lambda'_{n+1}}$ is bowtie free and downward flag;
		\item for $i=1,2$ and each vertex $v$ of type $\hat b_i$, the vertex set of the relative complex $\Delta^v_{\Lambda_{i},\Lambda'_{i}}$, endowed with the partial order induced from $b_i<b_3<b_4<\cdots<b_{n+1}$ is bowtie free and upward flag.
	\end{enumerate}
	Then the $(b_1,b_2)$-subdivision of $\Delta_{\Lambda,\Lambda'}$, viewed as a simplicial complex of type $S$ with $S=\{1,\ldots,n+1\}$, satisfies all the assumptions of Theorem~\ref{thm:contractibleII}. In particular $\Delta_{\Lambda,\Lambda'}$ is contractible.
\end{prop}

\begin{prop}[{\cite[Prop 9.1]{huang2024}}]
	\label{prop:propagation}
	Let $\Lambda,\Lambda',\Lambda_i,\Lambda'_i,\{b_i\}_{i=1}^{n+1}$ be as in Proposition~\ref{prop:ori link0}. Suppose all the assumptions in Proposition~\ref{prop:ori link0} holds true, and suppose $\Delta_\Lambda$ satisfies C$_4$. Then the following holds true:
	\begin{enumerate}
		\item For $i=1,2$, the vertex set of the relative  complex $\Delta_{\Lambda,\Lambda'_i}$, endowed with the order induced from $b_{3-i}<b_3<\cdots<b_{n+1}$, is a bowtie free, upward flag poset.
		\item The $(b_1,b_2)$-subdivision of $\Delta_{\Lambda,\Lambda'_{n+1}}$ is bowtie free and downward flag.
		\item Assume in addition that $\Delta^v_{\Lambda_{n+1},\Lambda'_{n+1}}$ satisfies the labeled 4-cycle condition for each vertex $v$ of type $\hat b_{n+1}$, then $\Delta_{\Lambda,\Lambda'}$ satisfies the labeled 4-cycle condition.
	\end{enumerate}
\end{prop}

\begin{definition}
	\label{def:subdivisionD}
	Let $\Lambda$ be a doubly marked diagram with a $\widetilde D_m$-like subdiagram $\Lambda'$ such that $m\ge 4$.
	Let $a_1,a_2,\{b_i\}_{i=1}^n,c_1,c_2$ be vertices of $\Lambda'$ as in Figure~\ref{fig:BD} right. Let $\Delta_\Lambda$ be a regular simplicial complex of type $V\Lambda$ modeled on $\Lambda$.
	Let $\Delta=\Delta_{\Lambda,\Lambda'}$ be the associated relative subcomplex. We add midpoint to each edge of $\Delta$ of type $\hat a_1\hat a_2$ or type $\hat c_1\hat c_2$, and define the midpoint to be of type $\hat a$ or $\hat c$ respectively.
	Cut each top dimensional simplex in $\Delta$ into four simplices whose vertex set is of type $\{\hat a_i,\hat a,\hat b_1,\ldots,\hat b_n,\hat c,\hat c_j\}_{1\le i,j\le 2}$. This gives a new simplicial complex $\Delta'$. Define a map $\tau$ from the vertex set $V\Delta'$ of $\Delta'$ to $\{1,\ldots,n+4\}$ by sending vertices of type $\hat a_i,\hat a,\hat b_1,\ldots,\hat b_n,\hat c,\hat c_j$ to $1,2,3,\ldots,n+2,n+3,n+4$ respectively.  We will then view $\Delta'$ as a simplicial complex of type $S=\{1,\ldots,n+4\}$. We define a relation $<$ on $V\Delta'$ as follows. For $x,y\in V\Delta'$, $x<y$ if $x$ and $y$ are adjacent and $\tau(x)<\tau(y)$. $(V\Delta',\le)$ is a poset, under the additional assumption that $\Lambda'$ is admissible in $\Lambda$.
\end{definition}

\begin{prop}[{\cite[Prop 2.26]{huang2024}}]
	\label{prop:ori link2}
	Suppose $\Delta_\Lambda$ is a regular simplicial complex of type $V\Lambda$ modeled on a doubly marked diagram $\Lambda$.
	Let $\Lambda'$ be a $\widetilde D_m$-like admissible subdiagram such that $m\ge 4$. Let $a_1,a_2,\{b_i\}_{i=1}^n,c_1,c_2$ be vertices of $\Lambda'$ as in Figure~\ref{fig:BD} right.
	Let $\Lambda_{a_i}$ (resp. $\Lambda'_{a_i}$) be the connected component of $\Lambda\setminus\{a_i\}$ (resp. $\Lambda'\setminus\{a_i\}$) that contains $\{b_i\}_{i=1}^n$. Similarly we define $\Lambda_{c_i}$ and $\Lambda'_{c_i}$. 
	Suppose that the following holds for $i=1,2$:
	\begin{enumerate}
		\item for each vertex $v$ of type $\hat c_i$, the $(a_1,a_2)$-subdivision of $\Delta^v_{\Lambda_{c_i},\Lambda'_{c_i}}$ is bowtie free and downward flag;
		\item for each vertex $v$ of type $\hat a_i$, the $(c_1,c_2)$-subdivision of $\Delta^v_{\Lambda_{a_i},\Lambda'_{a_i}}$ is bowtie free and downward flag.
	\end{enumerate}
	Then the simplicial complex $\Delta'$ in Definition~\ref{def:subdivisionD}, viewed as a simplicial complex of type $S$ with $S=\{1,\ldots,n+4\}$, satisfies all the assumptions of Theorem~\ref{thm:contractibleII}. In particular $\Delta_{\Lambda,\Lambda'}$ is contractible.
\end{prop}
\begin{prop}[{\cite[Prop 9.2]{huang2024}}]
	\label{prop:propagation2}
	Let $\Lambda,\Lambda',\{\Lambda_{a_i}\}_{i=1}^2,\{\Lambda_{c_i}\}_{i=1}^2$, $\{\Lambda'_{a_i}\}_{i=1}^2$, $\{\Lambda'_{c_i}\}_{i=1}^2$ be as in Proposition~\ref{prop:ori link2}. Suppose all the assumptions in Proposition~\ref{prop:ori link2} holds true, and suppose $\Delta_\Lambda$ satisfies C$_4$.
	Then the following holds true.
	\begin{enumerate}
		\item The $(a_1,a_2)$-subdivision of $\Delta_{\Lambda,\Lambda'_{c_i}}$ is bowtie free and downward flag ($i=1,2$).
		\item The $(c_1,c_2)$-subdivision of $\Delta_{\Lambda,\Lambda'_{a_i}}$ is bowtie free and downward flag ($i=1,2$).
		\item Assume in addition that for $i=1,2$, $\Delta^v_{\Lambda_{c_i},\Lambda'_{c_i}}$ and $\Delta^w_{\Lambda_{a_i},\Lambda'_{a_i}}$ satisfy the labeled 4-cycle condition for each vertex $v$ of type $\hat c_i$ and each vertex $w$ of type $\hat a_i$, then $\Delta_{\Lambda,\Lambda'}$ satisfies the labeled 4-cycle condition.
	\end{enumerate}
\end{prop}

Note that Assertion 3 is not explicitly stated in \cite[Prop 9.2]{huang2024}, however, by Assertions 1 and 2, and Lemma~\ref{lem:4wheel}, to prove Assertion 3, it remains to prove $\Delta_{\Lambda,a_1b_1a_2}$ and $\Delta_{\Lambda,c_1b_nc_2}$ are bowtie free. However, this can be proved in the same way as \cite[Prop 9.1 (3)]{huang2024}.

\section{Reduction from tree to special tripods}
\label{sec:propagation}

Throughout this section, $\Delta_S$ is nondegenerate, regular, modeled on the a doubly marked diagram, and satisfies C$_5$. We will be (implicitly) using the following.

\begin{lem}
	\label{lem:link inherit}
Under such assumption, for each simplex $\sigma$ of type $\hat T$, $\Delta^{\sigma}_{\Lambda\setminus T}$ is nondegenerate, regular, modeled on $\Lambda\setminus T$, and satisfies C$_5$. The same holds for $\Delta^\sigma_\Pi$ where $\Pi$ is any connected component of $\Lambda\setminus T$.
\end{lem}

\begin{proof}
The inheritance of regularity, Definition~\ref{def:model} and nondegeneracy follows from
\[
\lk\bigl(\tau,\lk(\sigma,\Delta_S)\bigr)
=\lk(\sigma\ast\tau,\Delta_S).
\]
The inheritance of C$_k$ follows from Lemma~\ref{lem:C_k link inherit}.
\end{proof}

\subsection{Atomic BD-robustness}
Let $\Lambda$ be a doubly marked diagram with edge sets $E_m\subset E_h$. An edge of $\Lambda$ is \emph{plain} if it is not inside $E_m$.
An induced subdiagram $\Lambda'$ of $\Lambda$ is $B_n$-like, if it is a linear subdiagram with consecutive vertices $s_1\cdots s_n$ such that $s_{n-1}s_n\in E_m$ (no restriction on other edges). For $n\ge 3$, a $B_n$-like subdiagram is \emph{$\widetilde C_{n-1}$-like} if in addition $s_1s_2\in E_m$.
Suppose $n\ge 2$. A $B_n$-like subdiagram $\Lambda'=s_1\cdots s_n\subset \Lambda$ with $s_{n-1}s_n\in E_m$ is \emph{$\Lambda$-atomic}, if all other edges are plain and the embedding $\Lambda'\to \Lambda$ preserve valence of interior vertices of $\Lambda'$ except at $s_2$ and $s_{n-1}$. A $D_{n+1}$-like subdiagram $\Lambda'$ is \emph{$\Lambda$-atomic} if all edges of $\Lambda'$ are plain and the embedding $\Lambda'\to \Lambda$ preserves valence of interior vertices of $\Lambda'$ except at vertices $b_3$ and $b_n$ in Figure~\ref{fig:BD} left.
\begin{definition}
	\label{def:good}
	Let $\Delta_\Lambda$ be a simplicial complex of type $V\Lambda$ modeled on a doubly marked tree (or forest) diagram $\Lambda$. We say that $\Delta_\Lambda$ is \emph{BD-robust} if the following two conditions hold simultaneously for all $n\ge 2$:
	\begin{enumerate}
		\item (\emph{B}-robust) for any $B_n$-like subdiagram $\Lambda'$ with consecutive vertices $s_1\ldots s_n$ and $s_{n-1}s_n\in E_m$, $\Delta_{\Lambda,\Lambda'}$ is bowtie free and upward flag, where the partial order on vertices of $\Delta_{\Lambda,\Lambda'}$ is induced by $s_1<s_2<\cdots<s_n$;
		\item (\emph{D}-robust) for any $D_{n+1}$-like subdiagram $\Lambda'$ with $b_1,b_2\in \Lambda'$ be as in Figure~\ref{fig:BD} left, the $(b_1,b_2)$-subdivision of $\Delta_{\Lambda,\Lambda'}$ is bowtie free and downward flag with respect to the partial order in Definition~\ref{def:subdivision}, moreover, $\Delta_{\Lambda,\Lambda'}$ satisfies the labeled 4-cycle condition.
	\end{enumerate} 
	We say $\Delta_\Lambda$ is \emph{atomic BD-robust} if the above two condition are only required to hold for $\Lambda$-atomic $B_n$-like subdiagrams and $D_n$-like subdiagrams.
\end{definition}

In later application of Definition~\ref{def:good}, sometimes we might reverse the partial order on the vertex set of $\Delta_{\Lambda,\Lambda'}$ (or its appropriate subdivision), in which case downward flag and upward flag are exchanged.

An induced subdiagram $\Lambda'$ of a doubly marked diagram $\Lambda$ is a \emph{$\widetilde C$-core}, if $\Lambda'$ is either $\widetilde B_n$-like, or $\widetilde C_n$-like or $\widetilde D_n$-like. 
%We say $\Lambda$ is $\widetilde C$-elementary, if it does not contain a $\widetilde C$-core.
Given a $\widetilde C$-core $\Lambda'$ of $\Lambda$, we define
$\bar\Delta_{\Lambda,\Lambda'}$ to be the subdivision of $\Delta_{\Lambda,\Lambda'}$ as in Definition~\ref{def:subdivision} if $\Lambda'$ is $\widetilde B_n$-like; $\bar\Delta_{\Lambda,\Lambda'}=\Delta_{\Lambda,\Lambda'}$ if $\Lambda'$ is $\widetilde C_n$-like; and $\bar\Delta_{\Lambda,\Lambda'}$ is the subdivision of $\Delta_{\Lambda,\Lambda'}$ as in Definition~\ref{def:subdivisionD} if $\Lambda'$ is $\widetilde D_n$-like. In each case the vertex set of $\bar\Delta_{\Lambda,\Lambda'}$ is endowed with a partial order and let $\tau(x)$ be defined in Definition~\ref{def:subdivision} and Definition~\ref{def:subdivisionD} (if $\Lambda'=s_1\ldots s_{n+1}$ is $\widetilde C_n$-like, then $\tau(x)=i$ if $x$ is of type $\hat s_i$). Sometimes we will reverse the partial order (hence the value of $\tau$) on $(\bar\Delta_{\Lambda,\Lambda'})^0$ to obtain another partial order. However, if $((\bar\Delta_{\Lambda,\Lambda'})^0,<)$ satisfies the assumptions of Theorem~\ref{thm:contractibleII}, then the same holds after reversing the partial order. Given a vertex $s\in \Lambda'$, note that all vertices in $\Delta_{\Lambda,\Lambda'}$ of type $\hat s$ have the same $\tau$-value, which is defined to $\tau(s)$.
A vertex of $\bar\Delta_{\Lambda,\Lambda'}$ is \emph{fake}, if it is not a vertex of $\Delta_{\Lambda,\Lambda'}$, otherwise this vertex is \emph{real}. Fake vertices of $\bar\Delta_{\Lambda,\Lambda'}$ are barycenters of edges in $\Delta_{\Lambda}$, hence we can define their types as in Section~\ref{subsec:rel Artin complex}.

\begin{lem}
	\label{lem:ld for Ccore}
	Suppose $\Lambda$ is a doubly marked tree diagram and $\Lambda'\subset\Lambda$ is a $\wtC$-core. Then $\bar\Delta_{\Lambda,\Lambda'}$ is locally determined. Moreover, for any simplex $\sigma$ of $\Delta_\Lambda$ of type $\hat T$ with $T\cap \Lambda'=\emptyset$, $\Delta^\sigma_{\Lambda\setminus\{T\},\Lambda'}$ is locally determined.
\end{lem}

\begin{proof}
	This follows from nondegeneracy of $\Delta_\Lambda$ and Lemma~\ref{lem:link inherit}.
\end{proof}

A $\widetilde C$-core $\Lambda'$ of $\Lambda$ is \emph{robust} with respect to $\Delta_\Lambda$, if 
$\bar\Delta_{\Lambda,\Lambda'}$ is $\widetilde C$-like. 

\subsection{Convexity of certain subcomplexes}
\label{subsec:convex}
%\begin{lem}
%	\label{lem:subcomplex}
%	Suppose $\Lambda$ is a tree Coxeter diagram such that any proper induced subdiagram of $\Lambda$ is BD-robust. Then for any $\widetilde C$-core $\Lambda'\subset \Lambda$, $\bar\Delta_{\Lambda,\Lambda'}$ is $\widetilde C_n$-like and wide.
%\end{lem}

\begin{lem}
	\label{lem:5cycle}
	Let $\Lambda$ be a doubly marked diagram with a $D_{n+1}$-like subdiagram $\Lambda_0$ such that its vertices are labeled as Figure~\ref{fig:core} left. 
	
	Let $\bar\Delta_{\Lambda,\Lambda_0}$ be the $(b_1,b_2)$-subdivision of $\Delta_{\Lambda,\Lambda_0}$ and assume $(\bar\Delta_{\Lambda,\Lambda_0})^0$ is endowed with the partial order in Definition~\ref{def:subdivision} such that vertices of type $\hat b_1$ and $\hat b_2$ are minimal. 
		We assume $((\bar\Delta_{\Lambda,\Lambda_0})^0,<)$ is bowtie free and downward flag.
		
		 Let $s$ be a vertex in $\Lambda$ such that $b_1$ is the vertex in $\Lambda_0$ that is closest to $s$. Let vertex $x\in \Delta_\Lambda$ be of type $s$ and let vertices $x_1\le x_2\ge x_3\le x_4$ be in $\bar\Delta_{\Lambda,\Lambda_0}$. We assume $x\sim\{x_1,x_4\}$. Then either $x_1$ and $x_4$ are equal or adjacent in $\bar\Delta_{\Lambda,\Lambda_0}$, or $x\sim x_2$. 
\end{lem}

Recall that $x\sim y$ means $x$ and $y$ are contained in the same simplex of $\Delta_\Lambda$. For $x,y\in \bar\Delta_{\Lambda,\Lambda_0}$, $x\sim y$ does not necessarily mean they are adjacent in $\bar\Delta_{\Lambda,\Lambda_0}$. Also note that $s=b_1$ is allowed in the lemma.

\begin{proof}
	Suppose $x_i$ has type $\hat S_i$ for $S_i\subset \Lambda_0$. If $x_1$ is not type $\hat b_2$, then any vertex in $S_2\setminus S_1$ and $s$ are different connected components of $\Lambda\setminus S_1$. Lemma~\ref{lem:transitive} implies that $x\sim x_2$. Now assume $x_1$ has type $\hat b_2$. Let $\tau_0$ be the function on $(\bar\Delta_{\Lambda,\Lambda_0})^0$ as in Definition~\ref{def:subdivision}.
	Suppose the lemma fails. Then we can find a counterexample to the lemma with smallest possible value of $\tau_0(x_2)+\tau_0(x_4)$. By condition C$_5$ applying to $xx_1x_2x_3x_4$, there are $z_1,z_2\in \bar\Delta_{\Lambda,\Lambda_0}$ with $x_1\le z_1\ge z_2\le x_4$ such that $z_2,x_3$ have the same type and $z_1,x_2$ have the same type. Let $z'_2$ be the meet of $x_4$ and $z_1$ in $((\bar\Delta_{\Lambda,\Lambda_0})^0,<)$, which exists as this poset is bowtie free (Lemma~\ref{lem:posets}). By condition C$_4$ applying to $x_4z'_2z_1x$, there is vertex $z''_2\in \bar\Delta_{\Lambda,\Lambda_0}$ such that $z''_2$ and $z'_2$ have the same type and $z''_2\sim\{z_1,x,x_4\}$. Thus $z''_2$ is a lower bound for $\{z_1,x_4\}$ in $((\bar\Delta_{\Lambda,\Lambda_0})^0,<)$. By definition of $z'_2$, we obtain $z'_2=z''_2$.  Thus $z'_2\sim x$. 
	
	%By Lemma~\ref{lem:link}, $\lk^-(y,\bar X)$ can be realized as a (possibly subdivided) subcomplex of $\Delta_{\Lambda_0,\Lambda'_0}$ for a proper induced subdiagram $\Lambda_0$ of $\Lambda$ and $\Lambda'_0\subset \Lambda'$ being $D_m$-like or $B_m$-like. 
	
	As $\{x_4,z_1,x_2\}$ are pairwise lower bounded in $((\bar\Delta_{\Lambda,\Lambda_0})^0,<)$, by assumption of the lemma, 
	there is a vertex $x'_3\in \bar\Delta_{\Lambda,\Lambda_0}$ which is a common lower bound for $\{x_4,z_1,x_2\}$. As $z'_2$ is the meet of $x_4$ and $z_1$, $z'_2\ge x'_3$.
	Let $x'_2$ be the join of $x'_3$ and $x_1$ in $((\bar\Delta_{\Lambda,\Lambda_0})^0,<)$. Then $x'_2\le z_1$ and $x'_2\le x_2$. Set $x'_4=z'_2$. Then $x_1\le x'_2\ge x'_3\le x'_4$. By construction, $x'_4\le x_4$, $x'_2\le x_2$, and $x_1,x'_4\sim x$. If $x'_2\neq x_2$, then $\tau_0(x'_2)+\tau_0(x'_4)<\tau_0(x_2)+\tau_0(x_4)$. Thus either $x_1$ is equal or adjacent to $x'_4$ in $\bar\Delta_{\Lambda,\Lambda_0}$ or $x'_2\sim x$. In the former case, as $\tau_0(x_1)=1$, $x_1\le x'_4$. As $x'_4\le x_4$, we obtain $x_1\le x_4$, contradiction to that $\{x_1,x_2,x_3,x_4\}$ is a counterexample to the lemma. In the latter case, 
	as $\tau_0(x_1)=1$, we must have $x'_2\neq x_1$, otherwise we have $x'_3=x_1$, which implies that $x_1$ and $x_4$ are adjacent in $\bar\Delta_{\Lambda,\Lambda_0}$, contradiction. Thus $\tau_0(x'_2)>1$. Then $x'_2\le x_2$ implies that any vertex in $S_2\setminus S'_2$ and $s$ are in different connected components of $\Lambda\setminus S'_2$, where $x'_2$ has type $\hat S'_2$. Lemma~\ref{lem:transitive} implies that $x_2\sim x$, contradiction again. If $x_2=x'_2$, then we deduce $x_2=z_1$ from $\tau_0(x_2)=\tau_0(z_1)$. Hence $x_2\sim x$, contradiction again. So there are no counterexamples to the lemma and the proof is finished.
\end{proof}

Given a $\widetilde C$-core $\Lambda'\subset\Lambda$ and a vertex $s\in \Lambda$, we say $s$ is \emph{$\Lambda'$-extremal} if the vertex $t\in \Lambda'$ closest to $s$ is a leaf vertex of $\Lambda'$. Such $s$ is  \emph{$\Lambda'$-maximal} (resp. \emph{$\Lambda'$-minimal}) if $\tau(t)$ achieves maximal (resp. minimal) value. 

\begin{lem}
	\label{lem:convex}
	Suppose $\Lambda$ is a doubly marked tree diagram with a $\widetilde C$-core $\Lambda'$ that is robust with respect to $\Delta_\Lambda$. Let $s\in \Lambda$ be a vertex that is $\Lambda'$-extremal. Let $x\in \Delta_{\Lambda}$ be a vertex of type $\hat s$ and let $X$ be the full subcomplex of $\Delta_{\Lambda,\Lambda'}$ spanned by vertices that are adjacent or equal to $x$. Let $\bar X$ be the subcomplex of $\bar\Delta_{\Lambda,\Lambda'}$ corresponding to $X$. 
	
	Then for any vertex $y\in \bar X$ with $x_1\le x_2\ge x_3\le x_4$ in $\lk^-(y,\bar\Delta_{\Lambda,\Lambda'})$ (or $\lk^+(y,\bar\Delta_{\Lambda,\Lambda'})$) such that $x_1,x_4$ are distinct non-adjacent vertices in $\bar X$, we have $x_3\in \bar X$ if $s$ is $\Lambda'$-maximal and $x_2\in \bar X$ if $s$ is $\Lambda'$-minimal.
\end{lem}

\begin{proof}
	We assume $y\neq x$, otherwise the lemma is clear. We only treat the case when $s$ is $\Lambda'$-minimal. The $\Lambda'$-maximal case follows by reversing the partial order on  $(\bar\Delta_{\Lambda,\Lambda'})^0$. Let $b_1\in \Lambda'$ be the vertex closest to $s$. Then $\tau(b_1)=1$. Note that there is at most one another vertex in $\Lambda'$ with $\tau$-value being $1$. This vertex, if exists, is denoted by $b_2$, and $b_2$ must be a leaf vertex of $\Lambda'$. The lemma follows immediately from Lemma~\ref{lem:transitive} if $x_1$ is not of type $\hat b_2$. So we will assume $b_2$ exists and $x_1$ is of type $\hat b_2$. Then $\tau(x_1)=1$. Hence $\{x_i\}_{i=1}^4$ are in $\lk^-(y,\bar\Delta_{\Lambda,\Lambda'})$. Suppose $\tau(y)\ge 3$ (otherwise $x_1=x_2=x_3=x_4$ and the lemma is trivial), and suppose $y$ of type $\hat T$. Let $\Gamma$ be the connected component of $\Lambda\setminus T$ that contains $s,b_1,b_2$, and let $\Gamma_0=\Gamma\cap \Lambda'$. Let $\sigma$ be the simplex of $\Delta_{\Lambda}$ containing $y$ as an interior point. Then $\lk^-(y,\bar \Delta_{\Lambda,\Lambda'})$ can be identified with $\bar\Delta^\sigma_{\Gamma,\Gamma_0}$ (see Remark~\ref{rmk:notation}), which is the $(b_1,b_2)$-subdivision of $\Delta^\sigma_{\Gamma,\Gamma_0}$. By our assumption, $\bar \Delta_{\Lambda,\Lambda'}$ is $\widetilde C$-like, so $((\bar\Delta^\sigma_{\Gamma,\Gamma_0})^0,<)$ is bowtie free and downward flag. Note that $x$ can be viewed as a vertex in $\Delta^\sigma_\Gamma$, and $\{x_i\}_{i=1}^4$ are vertices in $\bar\Delta^\sigma_{\Gamma,\Gamma_0}$. Now we are done by Lemma~\ref{lem:5cycle} and Lemma~\ref{lem:link inherit}.
\end{proof}

\begin{cor}
	\label{cor:convex}
	The following hold under the same assumption as Lemma~\ref{lem:convex}:
	\begin{enumerate}
		\item  $\bar X$ is locally convex in $\bar \Delta_{\Lambda,\Lambda'}$ in the sense of Definition~\ref{def:local convex};
		\item if $y_1,y_2$ are vertices of $\bar X$ with a local normal form path $\omega$ in $\bar\Delta_{\Lambda,\Lambda'}$ from $y_1$ to $y_2$ with $y_2$ extremal, then $\omega\subset \bar X$;
		\item if $y_1,y_2$ are vertices of $\bar X$ such that $y_1$ and $y_2$ have a common upper bound (resp. lower bound) in $\bar\Delta_{\Lambda,\Lambda'}$, then the join (resp. meet) of $y_1$ and $y_2$ in $\bar\Delta_{\Lambda,\Lambda'}$ is contained in $\bar X$. 
	\end{enumerate}
\end{cor}

%$\bar X$ is connected, locally determined, and locally convex in $\bar \Delta_{\Lambda,\Lambda'}$ in the sense of Definition~\ref{def:local convex}. 

\begin{proof}
	Each vertex $y\in \bar X$ is contained in a maximal simplex of $\bar\Delta_{\Lambda,\Lambda'}$, by considering the maximal simplex of $\Delta_\Lambda$ that contains $y$ and $x$ simultaneously. Thus Definition~\ref{def:local convex} (1) follows. For Definition~\ref{def:local convex} (2), we assume $y\neq x$, and consider a tight up-down edge path $x_1x_2x_3$ in $\lk^+(y,\bar \Delta_{\Lambda,\Lambda'})$ with $x_1,x_3\in \lk^+(y,\bar X)$. As $y\sim\{x,x_1,x_2,x_3\}$, by Lemma~\ref{lem:link} and Lemma~\ref{lem:link inherit}, we can apply C$_4$ to $xx_1x_2x_3$ in $\lk(y,\Delta_\Lambda)$ to deduce that there exists vertex $x'_2\in \bar\Delta_{\Lambda,\Lambda'}$ such that $x'_2\sim\{x_1,x,y,x_3\}$ and $x'_2$ has the same type as $x_2$. In particular, $x'_2\in \lk(y,\bar X)$. The tightness of $x_1x_2x_3$ implies $x'_2=x_2$, as desired. Other cases of Definition~\ref{def:local convex} (2) can be checked similarly.
	Definition~\ref{def:local convex} (3) reduces to Definition~\ref{def:local convex} (2) by Lemma~\ref{lem:convex}.
	
	Now we prove Assertion 2. First we consider the case $s\notin \Lambda'$. Here $X$ is the relative complex $\Delta^x_{\Lambda\setminus\{s\},\Lambda'}$, which is connected (actually simply-connected) by Lemma~\ref{lem:scgeneral} and Lemma~\ref{lem:link inherit}. By Lemma~\ref{lem:ld for Ccore},  $\bar X\cong \bar \Delta^x_{\Lambda\setminus\{s\},\Lambda'}$ is locally determined. Now Assertion 2 follows by Lemma~\ref{lem:ld for Ccore} (again) and Proposition~\ref{prop:convex subcomplex}.

	Suppose $s\in \Lambda'$. Then $x\in \Delta_{\Lambda,\Lambda'}$. If there does not exists vertex $s'\neq s$ of $\Lambda'$ with $\tau(s')=\tau(s)$, then $\bar X$ is the closed star of $x$ in $\bar\Delta_{\Lambda,\Lambda'}$, and Assertion 2 follows from Proposition~\ref{prop:bestvina} and Lemma~\ref{lem:ld for Ccore}. Suppose there is a vertex $s'\neq s$ of $\Lambda'$ with $\tau(s')=\tau(s)$. Assume without loss of generality that $\tau(s)=\tau(s')=1$. Let $\omega=z_1\cdots z_k$ be a local normal form in $\bar\Delta_{\Lambda,\Lambda'}$ from $z_1=y_1$ to $z_k=y_2$. Let $d$ be the path metric on $(\bar\Delta_{\Lambda,\Lambda'})^1$ with unit edge length.
	As $\omega$ is geodesic in $\bar\Delta_{\Lambda,\Lambda'}$ by Lemma~\ref{lem:ld for Ccore} and Proposition~\ref{prop:local vs global}, $k\le 5$. Note that $d(z_i,x)\le 2$ for $i=1,k$, hence this holds for $1\le i\le k$ by Proposition~\ref{prop:bestvina}. We claim that if $d(z_1,x)=2$, then $d(z_2,x)=1$; and if $d(z_k,x)=2$, then $d(z_{k-1},x)=1$. This claim together with Proposition~\ref{prop:bestvina} implies that $d(z_i,x)\le 1$ for $1<i<k$, which gives $\omega\subset \bar X$.

	%If both $d(z_i,x)\le 1$ for $i=1,k$, then $\omega\subset \bar X$ by Proposition~\ref{prop:bestvina}. Now we assume at least one of $\{z_1,z_k\}$, say $z_1$, is not adjacent to $x$ in $\bar\Delta_{\Lambda,\Lambda'}$, hence is of type $\hat s'$. Note that $d(z_1,x)=2$ and $d(z_k,x)\le 2$. By Proposition~\ref{prop:bestvina}, $d(z_i,x)\le 2$ for all $i$.
	
	It remains to prove the claim. We will only prove $d(z_1,x)=2$ implies $d(z_2,x)=1$, as the other part is similar.
	Assume by contradiction that $d(z_2,x)=2$. Let $z_2z'_2x$ be the local normal form path from $z_2$ to $x$ in $\bar\Delta_{\Lambda,\Lambda'}$, which exists by Lemma~\ref{lem:ld for Ccore} and Proposition~\ref{prop:existence and uniqueness}. Then $z_2z'_2x$ is also a normal form path by Proposition~\ref{prop:local vs global}. As $\tau(x)=1$, we have $x<z'_2>z_2$. If $d(z_3,x)=1$, then by Definition~\ref{def:normal} applied to $z_2z'_2x$, we have $z_3>z_2$, which contradicts $z_1<z_2>z_3$ (as $\omega$ is an up-down path). Thus $d(z_3,x)=2$. For $i=1,3$, let $z_iz'_ix$ be a local normal form path from $z_i$ to $x$, which is also in normal form from $z_i$ to $x$. Then $x<z'_i>z_i$ for $i=1,3$. Note that $\{z'_1,z_2,z'_3\}$ is pairwise lower bounded, hence they have a common lower bound $z$ by Lemma~\ref{lem:big lattice}. We can assume $\tau(z)=1$. As $\tau(z'_1)=2$, we obtain $z=z_1$ or $x$. If $z=x$, then $d(z_2,x)=d(z_2,z)=1$. If $z=z_1$, then $z_1\le z'_3$. As $x<z'_3$ and $z'_1$ is the join of $\{z_1,x\}$, $z'_1\le z'_3$. As $z_2$ is the join of $\{z_1,z_3\}$ and $z'_3$ is a common upper bound for $\{z_1,z_3\}$, $z_2\le z'_3$. If $k=3$, then $z_3$ is of type $\hat s'$ and $\tau(z'_3)=2$. As $z'_1\le z'_3$ and $\tau(z'_1)=2$, we have $z'_1=z'_3$ and $z_3<z'_1$. Hence $z_3=z_1$ or $z_3=x$. The former is impossible as $\omega$ is geodesic, and the latter implies $d(z_2,x)=1$. Now assume $k>3$. If $d(z_4,x)=1$, then $ z'_3\le z_4$ as $z_3z'_3x$ is in normal form from $z_3$ to $x$. Then $z_2\le z'_3\le z_4$, contradicting that $\omega$ is geodesic. Thus $d(z_4,x)=2$. It follows that $k>4$, otherwise $z_4\in \bar X$ and $\tau(z_4)=1$, contradicting $z_3<z_4$. Let $z_4z'_4x$ be the local normal form (hence normal form) from $z_4$ to $x$. Then $z_4<z'_4>x$. If $d(z_5,x)=1$, then $x<z_5>z_4$, contradicting $z_4>z_5$ (as $\omega$ is up-down). So $d(z_5,x)=2$. Let $z_5z'_5x$ be the local normal form $z_5$ to $x$. As $z_5\in \bar X$, we know $z_5$ is of type $\hat s'$ and $\tau(z'_5)=2$. Hence $\{z'_3,z'_5,z_4\}$ is pairwise lower bounded, and we argue as before to deduce that $z_4\le z'_3$. By considering $z_2\le z'_3\ge z_4$ in $\lk^+(z_3,\bar \Delta_{\Lambda,\Lambda'})$, this contradicts that $z_2z_3z_4$ is locally normal. Thus the claim is proved.
	
	Now we prove Assertion 3. Let $z$ be the join of $y_1$ and $y_2$ in $\bar\Delta_{\Lambda,\Lambda'}$. By considering $xy_1zy_2$ and using C$_4$, there exists $z'$ of the same type as $z$ such that $z'\sim\{x,y_1,y_2\}$. In particular, $z'\in \bar X$. The type of $z'$ implies $z'$ is a vertex of $\bar\Delta_{\Lambda,\Lambda'}$, so $z'$ is a common upper bound for $\{y_1,y_2\}$. Thus $z'=z$ and Assertion 3 follows. The case of $z$ being the meet of $y_1$ and $y_2$ is similar.
\end{proof}

\subsection{Intersection of convex subcomplexes}
\label{subsec:convex intersection}
Throughout Section~\ref{subsec:convex intersection}, we will be working under Assumption~\ref{assum:1}, and the assumptions in the beginning of Section~\ref{sec:propagation}.

\begin{assumption}
	\label{assum:1}
	Let $\Lambda$ be a doubly marked tree diagram with a robust $\widetilde C$-core $\Lambda'$. Let $\Lambda_0$ be a $B_m$-like $(m\ge 2)$ or $D_m$-like $(m\ge 4)$ subdiagram of $\Lambda$.
	
	If $\Lambda_0$ is $B_m$-like, we label $\Lambda_0=s_1\ldots s_m$ with $s_{m-1}s_m\in E_m$ and assume that $\Lambda_0\cap \Lambda'=s_{m-1}s_m$, as in Figure~\ref{fig:core} (1) and (2), or $\Lambda_0\cap \Lambda'$ is $s_{m-1}s_m$ together with another edge as in Figure~\ref{fig:core} (5).
	
	If $\Lambda_0$ is $D_m$-like, we label the vertices of $\Lambda_0$ as in Figure~\ref{fig:core} (0) and assume that $\Lambda_0\cap \Lambda'$ is $D_4$-like, as in Figure~\ref{fig:core} (3) and (4).
	
	Let $\bar\Delta_{\Lambda,\Lambda_0}$ denote $\Delta_{\Lambda,\Lambda_0}$ if $\Lambda_0$ is $B_m$-like, and the $(b_m,b_{m-1})$-subdivision of $\Delta_{\Lambda,\Lambda_0}$ if $\Lambda_0$ is $D_m$-like. Endow $\bar\Delta_{\Lambda,\Lambda_0}$ with a poset structure on its vertex set such that vertices of type $\hat s_m$ are maximal when $\Lambda_0$ is $B_m$-like, and vertices of type $\hat b_m$ or $\hat b_{m-1}$ are maximal when $\Lambda_0$ is $D_m$-like (i.e. we reverse the partial order from Definition~\ref{def:subdivision}).
	
	Suppose that for each vertex $s\in \Lambda'\setminus \Lambda_0$ and each $v\in \Delta_S$ of type $\hat s$, the complex $\bar\Delta^v_{\Omega,\Lambda_0}$ is bowtie-free and upward flag, where $\Omega$ is the connected component of $\Lambda\setminus\{s\}$ containing $\Lambda_0$. The poset on  $(\bar\Delta^v_{\Omega,\Lambda_0})^0$ is defined similarly to the previous paragraph.

	If $\Lambda'$ is $\widetilde D_4$-like, we further assume that $(\bar\Delta^v_{\Gamma,\Gamma'})^0$ is upward flag for each vertex $v$ of type $\hat b_{m-1}$, where  $\Gamma'=(\Lambda'\cup\Lambda_0)\setminus\{b_{m-1}\}$, $\Gamma$ is the connected component of $\Lambda\setminus\{b_{m-1}\}$ containing $b_{m-2}$, and $\bar\Delta^v_{\Gamma,\Gamma'}$ is the subdivision of $\Delta^v_{\Gamma,\Gamma'}$ as in Definition~\ref{def:subdivision}, endowed with the poset structure in which vertices of type $\hat b_1$ are minimal.
	
	Let $t_1$ be the vertex of $\Lambda'$ that is closest to the vertices in $\Lambda_0\setminus \Lambda'$ (when $\Lambda_0\setminus \Lambda'=\emptyset$, set $t_1=s_1$ or $b_1$).
	Let $\theta_1,\theta_2,\theta_3$ be vertices in $(\Lambda_0\setminus \Lambda')\cup\{t_1\}$. For each $i=1,2,3$, let $x_i\in \Delta_\Lambda$ be a vertex of type $\hat \theta_i$, and let $X_i$ be the full subcomplex of $\Delta_{\Lambda,\Lambda'}$ spanned by vertices that are adjacent to or equal to $x_i$.
	
	Define $\bar\Delta_{\Lambda,\Lambda'}$, $\bar X_i$, and $\tau\colon (\bar\Delta_{\Lambda,\Lambda'})^0\to \mathbb{Z}$ as before, and assume $\tau$ takes values in $\{1,\ldots,n\}$. Up to reversing the partial order on $(\bar\Delta_{\Lambda,\Lambda'})^0$, we assume that $\theta_1,\theta_2,\theta_3$ are $\Lambda'$-minimal. There is at most one additional vertex of $\Lambda'$ with $\tau$-value $1$; if it exists, we denote it by $t'_1$ (in Figure~\ref{fig:core} (3) and (4), $t'_1=b_{m-1}$).
\end{assumption}

\begin{figure}[h]
	\centering
	\includegraphics[scale=0.9]{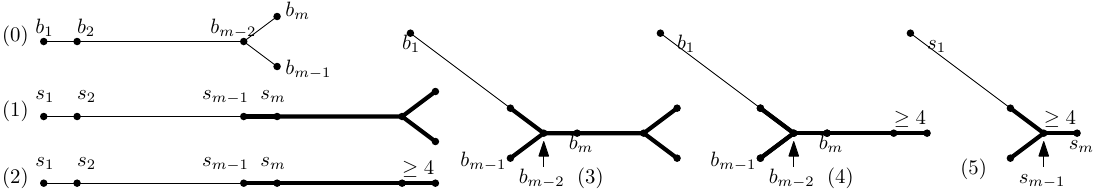}
	\caption{The thickened subdiagram indicates the robust $\wtC$-core $\Lambda'$. In (3), $\Lambda'$ can be $\widetilde D_4$-like, in which case $b_m$ is a leaf vertex of $\Lambda'$. Similarly, in (4), $\Lambda'$ can be $\widetilde B_3$-like, in which case $b_m$ is a leaf vertex of $\Lambda'$. }
	\label{fig:core}
\end{figure}

%Let $\Gamma'$ be the subtree of $\Lambda$ spanned by $s_1$ or $b_1$ and all vertices of $\Lambda'$ with $\tau$-value $n$. Then $\Gamma'$ is either $B_m$-like or $D_m$-like and $\Gamma'$ contains all vertices of $\Lambda'$ except $t'_1$. Let $\Gamma$ be the connected component of $\Lambda\setminus\{t'_1\}$ that contains $\Gamma'$ (if $t'_1$ does not exist, then $\Gamma=\Lambda$). Let $\bar\Delta_{\Gamma,\Gamma'}$ be either $\Delta_{\Gamma,\Gamma'}$ or its appropriate subdivision (when $\Gamma'$ is $D_m$-like). We order vertices in $\bar\Delta_{\Gamma,\Gamma'}$ such that elements of type $\hat s_1$ or $\hat b_1$ are minimal.

Note that $\bar X_i\subset \bar\Delta_{\Lambda,\Lambda'}$ satisfies the conclusion of Corollary~\ref{cor:convex}. As $\bar\Delta_{\Lambda,\Lambda'}$ is $\widetilde C$-like and locally determined by Lemma~\ref{lem:ld for Ccore}, there is a local normal form path between any two extremal vertices of $\bar\Delta_{\Lambda,\Lambda'}$ by Proposition~\ref{prop:existence and uniqueness}, which is also a normal form path by Proposition~\ref{prop:local vs global}.

\begin{lem}
	\label{lem:corner1}
	Under Assumption~\ref{assum:1}, suppose that $\bar X_2\cap \bar X_3$ is nonempty and contains a vertex of type $\hat s$ for some $s\in \Lambda_0$. Let $(u_1,u_2,u_3)$ be three extremal vertices in $\bar\Delta_{\Lambda,\Lambda'}$ such that $u_1,u_2\in \bar X_2$ and $u_2,u_3\in \bar X_3$. Assume that $u_2$ is chosen so as to minimize the quantity
	\begin{equation}
		\label{eq:ds0}
		d(u_1,u_2)+d(u_2,u_3),
	\end{equation}
	among all extremal vertices $u_2\in \bar X_2\cap \bar X_3$, where $d$ denotes the path metric on the $1$-skeleton of $\bar\Delta_{\Lambda,\Lambda'}$ with unit edge lengths. Suppose further that $d(u_1,u_2)\ge 2$.
	
	Let $\omega=y_1y_2\cdots y_k$ be a path from $u_2$ to $u_3$ that is in local normal form either from $u_2$ to $u_3$ or from $u_3$ to $u_2$. Then the following hold:
	\begin{enumerate}
		\item If $\tau(u_2)=1$, then $d(u_1,y_2)>d(u_1,u_2)$.
		\item If $\tau(u_2)=n$ and we are not in the special case where $\Lambda'$ is $\widetilde D_4$-like, $u_2$ is of type $\hat b_m$, and $y_2$ is of type $\hat b_{m-1}$, then
		\[
		d(u_1,y_i)\ge d(u_1,u_2)\quad \text{for all } 1\le i\le k.
		\]
		\item If $\Lambda'$ is $\widetilde D_4$-like, $u_2$ is of type $\hat b_m$, and $y_2$ is of type $\hat b_{m-1}$, then there exists at most one index $i$ such that $d(u_1,y_i)<d(u_1,u_2)$; moreover, if such an index exists, then $i=2$.
	\end{enumerate}
\end{lem}

\begin{proof}
	By Corollary~\ref{cor:convex}, $\omega\subset \bar X_3$.
	For each $i$, let $y^{a_i}_{i}y^{a_i-1}_{i}\cdots y^1_{i}$ be the local normal form path from $y_i=y^{a_i}_{i}$ to $y^1_{i}=u_1$ ($a_i=d(y_i,u_1)+1$), whose existence is guaranteed by Proposition~\ref{prop:existence and uniqueness}. This path is also a normal form path by Proposition~\ref{prop:local vs global}. 
	
	\begin{figure}[h]
		\centering
		\includegraphics[scale=0.9]{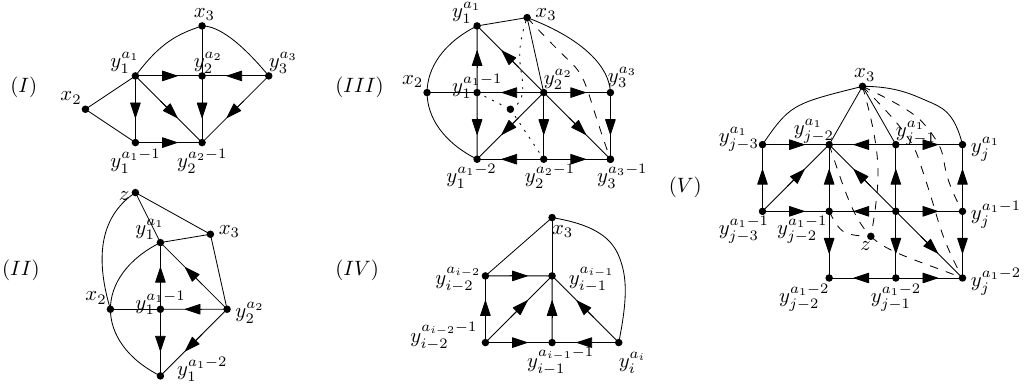}
		\caption{Some diagrams in the proof of Lemma~\ref{lem:corner1}.}
		\label{fig:corner}
	\end{figure}
	
	First we consider the case when $\tau(u_2)=1$. See Figure~\ref{fig:corner} (I). Then $y^{a_2}_2>y^{a_1}_1$ and $y^{a_1-1}_1>y^{a_1}_1$. We first consider the subcase $d(u_1,y_2)=d(u_1,u_2)$.  By Lemma~\ref{lem:strip}, $y^{i}_2\ge y^{i}_1$ for any $i$, in particular, $y^{a_2-1}_2\ge y^{a_1-1}_1$. As $y^{a_1-2}_1<y^{a_1-1}_1$, we deduce $y^{a_1-2}_1<y^{a_2-1}_2$. Then $y^{a_2-1}_2>y^{a_2}_2$, otherwise $y^{a_2}_2>y^{a_1-2}_1$ by transitivity, which implies that $d(y^{a_2}_2,u_1)\le d(y^{a_1-2}_1,u_1)+1=d(u_1,u_2)-1$, contradiction. 	As $\omega$ is an up-down path, $y^{a_3}_3<y^{a_2}_2$. Thus $y^{a_3}_3<y^{a_2-1}_2$ and $y^{a_1}_1<y^{a_1-1}_2$. Note that $x_2\sim y^{a_1-1}_1$ and $y^{a_1-1}_1\sim y^{a_2-1}_2$. As $\tau(y^{a_1-1}_1)>1$, Lemma~\ref{lem:transitive} implies $x_2\sim y^{a_2-1}_2$. Hence $y^{a_2-1}_2\in \bar X_2$. Similarly, $y^{a_2-1}_2\in \bar X_3$. Let $u'_2\in \bar\Delta_{\Lambda,\Lambda'}$ be an extremal vertex with $u'_2\ge y^{a_2-1}_2$. As $\tau( y^{a_2-1}_2)\ge 2$, Lemma~\ref{lem:transitive} implies $u'_2\sim\{x_2,x_3\}$. 
	Note that $u'_2\ge y^{a_2-2}_2$ and $u'_2\ge y^{a_3}_3$. Hence $d(u'_2,u_1)\le d(u_2,u_1)-1$ and $d(u'_2,u_3)\le d(u_2,u_3)-1$. Thus replacing $u_2$ by $u'_2$ decreases \eqref{eq:ds0}, contradiction. The subcase $d(u_1,y_2)<d(u_1,u_2)$ is similar.
	
	Suppose $\tau(u_2)=n$. Suppose there is $i$ such that $d(u_1,y_i)<d(u_1,u_2)$, and let $i$ be the smallest possible such value. By Proposition~\ref{prop:bestvina} and the choice of $\omega$, $d(u_1,y_j)=d(u_1,u_2)$ for $j<i$. First assume $i=2$. See Figure~\ref{fig:corner} (II). Then $y^{a_1-1}_1<y^{a_1}_1$ and $y^{a_2}_2<y^{a_1}_1$. Definition~\ref{def:normal} implies that $y^{a_2}_2\le y^{a_1-1}_1$. As $y^{a_1-1}_1<y^{a_1-2}_1$ (the normal form path from $y_1$ to $u_1$ is an up-down path), $y^{a_2}_2<y^{a_1-2}_1$. If $y_2$ is not of type $\hat t'_1$, then Lemma~\ref{lem:transitive} implies that $x_3\sim y^{a_1-2}_1$ and we can decrease \eqref{eq:ds0} by replacing $u_2$ by any $u'_2$ such that $\tau(u'_2)=n$ and $u'_2\ge y^{a_1-2}_1$. 
	Now assume $y_2$ is of type $\hat t'_1$. %Then $\Lambda_0$ is $D_m$-like. 
	If $\tau(y^{a_1-1}_1)\le 3$, then Lemma~\ref{lem:transitive} implies that $x_2\sim y_2$, hence we can decrease \eqref{eq:ds0} by replacing $u_2$ by $y_2$. So we assume $\tau(y^{a_1-1}_1)>3$, in which case $n\ge 5$.

	Suppose $y^{a_1}_1=u_2$ has type $\hat t$ with $t\notin \Lambda_0$.
	Let $\Omega$ be the connected component of $\Lambda\setminus t$ that contains $b_1$. 
	Our assumption gives vertex $z\in \bar\Delta_{\Lambda,\Lambda'}$ of type $\hat s$ with $s\in \Lambda_0\cap\Lambda'$ such that $z\sim \{x_2,x_3\}$. By Lemma~\ref{lem:transitive}, we can assume $s=b_m$ or $b_{m-1}$. By applying C$_4$ to the 4-cycle $x_2zx_3y_1$ in $\Delta_\Lambda$, up to replacing $z$, we can assume $z$ is adjacent to $y_1$ in $\Delta_\Lambda$. We first assume that we are not in the special situation that $\Lambda'$ is $\widetilde D_{n-1}$-like and $\tau(y^{a_1-1}_1)=n-1$ hold simultaneously.
	As $\tau(y^{a_1-1}_1)>3$, $zx_2y^{a_1-1}_1y_2x_3$ forms a 5-cycle in $\lk(y^{a_1}_1,\Delta_\Lambda)$.  Lemma~\ref{lem:link} implies we can view $zx_2y^{a_1-1}_1y_2x_3$ as a 5-cycle in $\Delta^{u_2}_{\Omega}$ with  $z,x_2,y_2,x_3\in \Delta^{u_2}_{\Omega,\Lambda_0}$. As $(\bar\Delta^{u_2}_{\Omega,\Lambda_0})^0$ is upward flag and bowtie free (Assumption~\ref{assum:1}, paragraph 5), Lemma~\ref{lem:5cycle} implies that either $x_2$ and $y_2$ are adjacent in $\Delta_\Lambda$, in which case we conclude as the previous paragraph; or $x_3$ is adjacent to $y^{a_1-1}_1$. As $y^{a_1-1}_1<y^{a_1-2}_1$ and $\tau(y^{a_1-1}_1)\ge 2$, Lemma~\ref{lem:transitive} implies $y^{a_1-2}_1\in \bar X_3$. Hence replacing $u_2$ by a maximal vertex in $\bar\Delta_{\Lambda,\Lambda'}$ that is $\ge y^{a_1-2}_1$ decreases \eqref{eq:ds0}. In the special case that $\Lambda'$ is $\widetilde D_{n-1}$-like and $\tau(y^{a_1-1}_1)=n-1$, $y^{a_1-1}_1$ is the midpoint of the edge $y^{a_1}_1y^{a_1-2}_1$ of $\Delta_\Lambda$, so  $zx_2y^{a_1-2}_1y_2x_3$ forms a 5-cycle in $\lk(y^{a_1}_1,\Delta_\Lambda)$ and we can finish as before.
	
	Now we consider the case $t\in \Lambda_0$. This can only happen when $\Lambda'$ is $\widetilde D_4$-like and $t=b_m$. Then $n=5$ and $\tau(y^{a_1-1}_1)=4$.
	See Figure~\ref{fig:corner} (III) for the following discussion.
	If $d(y_3,u_1)=d(y_2,u_1)$, as $y^{a_1-1}_1<y^{a_1-2}_1$, we know $y^{a_2}_2<y^{a_2-1}_2$ and $y^{a_3}_3<y^{a_3-1}_3$.
	As $y_2<y_3$, by Lemma~\ref{lem:strip}, $y^{a_2-1}_2\le y^{a_3-1}_3$. As $\tau(y^{a_3}_3)\ge 2$, $x_3\sim y^{a_3}_3$ and $y^{a_3}_3\sim y^{a_3-1}_3$, Lemma~\ref{lem:transitive} implies that $x_3\sim y^{a_3-1}_3$. Let $\mathcal P=((\bar\Delta^{y_2}_{\Gamma,\Gamma'})^0,<)$ be in Assumption~\ref{assum:1}.
	Note that $y^{a_2}_2\sim\{y^{a_1-1}_1,y^{a_2-1}_2,x_3\}$. By Lemma~\ref{lem:link}, we view $\{y^{a_1-1}_1,y^{a_2-1}_2,x_3\}$ as in $\mathcal P$ that are pairwise upper bounded (if $\tau(y^{a_2-1}_2)=2$, then $y^{a_2-2}_2$ exists and we replace $y^{a_2-1}_2$ by $y^{a_2-2}_2$). Thus they have a common upper bound $z$ in $\mathcal P$. As $y^{a_1}_1$ and $y^{a_1-2}_1$ are the only two elements in $\mathcal P$ that is $>y^{a_1-1}_1$, we can assume $z=y^{a_1}_1$ or $y^{a_1-2}_1$. The former is impossible as $d(y^{a_1}_1,y^{a_2-1}_2)=2$. The latter implies $y^{a_1-2}_1\sim x_3$ and we conclude the proof as before. The case of $d(y_3,u_1)<d(y_2,u_1)$ can be ruled out similarly. Thus $d(y_3,u_1)>d(y_2,u_1)$ and Proposition~\ref{prop:bestvina} implies $d(y_j,u_1)\ge d(u_1,u_2)$ for any $j\ge 3$.

	Now we consider the case $i>2$. See Figure~\ref{fig:corner} (IV). As $y^{a_1-1}_{1}<y^{a_1}_1$ and $d(u_1,y_j)=d(u_1,y_1)$ for $1\le j<i$, we know $y^{a_j-1}_j<y^{a_j}_j$ for $1\le j<i$. In particular $y^{a_{i-1}-1}_{i-1}<y^{a_{i-1}}_{i-1}$. Thus $y^{a_i}_i<y^{a_{i-1}}_{i-1}$ as $y^{a_i}_i$ is contained in a geodesic from $u_1$ to $y^{a_{i-1}}_{i-1}$. Thus $i$ is even. Lemma~\ref{lem:strip} implies that $y^{a_{i-2}-1}_{i-2}\le y^{a_{i-1}-1}_{i-1}$ and $y^{a_i}_i\le y^{a_{i-1}-1}_{i-1}$. Thus by considering $y^{a_{i-2}}_{i-2}\ge y^{a_{i-2}-1}_{i-2}\le y^{a_{i-1}-1}_{i-1}\ge y^{a_i}_i$ in $\lk^-(y^{a_{i-1}}_{i-1},\bar\Delta_{\Lambda,\Lambda'})$ and Lemma~\ref{lem:convex}, we know $y^{a_{i-1}-1}_{i-1}\in \bar X_3$.

	Next we show if $y^{a_1-1}_j\in \bar X_3$ for $j$ odd with $i>j>3$, then $y^{a_{1}-1}_{j-2}\in \bar X_3$ (note that $a_j=a_1$ for any $j<i$). See Figure~\ref{fig:corner} (V). 
	Indeed, as $j$ is odd, by Lemma~\ref{lem:strip}, $y^p_{j-3}\le y^p_{j-2}\ge y^p_{j-1}\le y^p_j$ for $1\le p\le a_1$. If $y^{a_1-1}_{j-1}=y^{a_1-1}_j$, then by considering 
$y^{a_1}_{j-3}y^{a_1-1}_{j-3}y^{a_1-1}_{j-2}y^{a_1-1}_{j-1}$ in $\lk^-(y^{a_1}_{j-2},\bar\Delta_{\Lambda,\Lambda'})$
and Lemma~\ref{lem:convex}, we have $y^{a_{1}-1}_{j-2}\in \bar X_3$. Now assume $y^{a_1-1}_{j-1}<y^{a_1-1}_j$. Then $\tau(y^{a_1-1}_j)\ge 2$. As $x_3\sim y^{a_1-1}_j$ and $y^{a_1-1}_j<y^{a_1-2}_j$, Lemma~\ref{lem:transitive} implies $x_3\sim y^{a_1-2}_j$. Thus $y^{a_1-2}_j\in \bar X_3$. By Corollary~\ref{cor:convex}, the meet $z$ of $y^{a_1-2}_j$ and $y^{a_1}_{j-2}$ is contained in $\bar X_3$ (note that $y^{a_1-2}_j$ and $y^{a_1}_{j-2}$ have a common lower bound $y^{a_1-1}_{j-1}$). As $d(u_1,z)=a_1-2=d(u_1,y_{j-2})-1$, $z$ is in a geodesic from $y_{j-2}$ to $u_1$. Definition~\ref{def:normal} implies that $z\le y^{a_1-1}_{j-2}$. By considering $y^{a_1}_{j-3}\ge y^{a_1-1}_{j-3}\le y^{a_1-1}_{j-2}\ge z$ in $\lk^-(y^{a_1}_{j-2},\bar\Delta_{\Lambda,\Lambda'})$ and Lemma~\ref{lem:convex}, then $y^{a_{1}-1}_{j-2}\in \bar X_3$.
	
	The previous two paragraphs imply that $y^{a_1-1}_3\in \bar X_3$. As $y^{a_1-1}_2\le y^{a_1-1}_3\ge y^{a_1-1}_4$ by Lemma~\ref{lem:strip}, if $\tau(y^{a_1-1}_3)=1$, then $y^{a_1-1}_2=y^{a_1-1}_4$, contradicting that $\omega$ is a geodesic. Thus $\tau(y^{a_1-1}_3)\ge 2$. As $x_3\sim y^{a_1-1}_3$ and $y^{a_1-1}_3<y^{a_1-2}_3$, Lemma~\ref{lem:transitive} implies that $y^{a_1-2}_3\in \bar X_3$. Let $y$ be the meet of $y^{a_1}_1$ and $y^{a_1-2}_3$. Then $y\in \bar X_3$ by Corollary~\ref{cor:convex}. Similar to the previous paragraph, $y\le y^{a_1-1}_1$. If $\tau(y)\ge 2$, then $x_3\sim y$ and Lemma~\ref{lem:transitive} implies that $y^{a_1-1}_1\sim x_3$. By Lemma~\ref{lem:transitive} again $y^{a_1-2}_1\sim x_3$. Then we can replace $u_2$ by any $u'_2$ with $\tau(u'_2)=n$ and $u'_2\ge y^{a_1-2}_1$ to decrease \eqref{eq:ds0}. If $\tau(y)=1$, as $y^{a_1-1}_2$ is a common lower bound for $y^{a_1}_1$ and $y^{a_1-2}_3$, $y=y^{a_1-1}_2$. Hence $y^{a_1-1}_2\in \bar X_3$, and we can argue in the same way as the $i=2$ case to finish the proof.
\end{proof}
%Suppose $(\bar\Delta_{\Gamma,\Gamma'})^0$ is upward flag. Let $t_1,t'_1,X_1,X_2,X_3$ and $\Lambda'\subset \Lambda$ be as in the beginning of Section~\ref{subsec:convex intersection}.
\begin{prop}
	\label{prop:intersection}
	Under Assumption~\ref{assum:1}, suppose for $1\le i\neq j\le 3$, $X_i\cap X_j$ contains a vertex of type $\hat s$ with $s\in \Lambda_0$. Then $X_1\cap X_2\cap X_3\neq\emptyset$.
\end{prop}

\begin{proof}
	By Lemma~\ref{lem:transitive}, if $\bar X_i\cap \bar X_j$ contains a vertex which is not of type $\hat t'_1$, then $\bar X_i\cap \bar X_j$ contains a vertex with $\tau$-value $n$. Thus $\bar X_i\cap \bar X_j$ always contain an extremal vertex. Let $\Xi$ be the set of triples of extremal vertices $(z_1,z_2,z_3)$ in $\bar\Delta_{\Lambda,\Lambda'}$ such that $z_i\in \bar X_i\cap \bar X_{i+1}$ for $i\in \mathbb Z/3\mathbb Z$. Let $(u_1,u_2,u_3)$ be an element in $\Xi$ that minimizes 
	\begin{equation}
		\label{eq:ds}
		d(u_1,u_2)+d(u_2,u_3)+d(u_3,u_1)
	\end{equation}
	among all elements in $\Xi$, where $d$ denotes the path distance in the 1-skeleton of $\bar\Delta_{\Lambda,\Lambda'}$ with edge length $1$. We will prove two of $\{u_1,u_2,u_3\}$ are the same, which implies the proposition. For now we assume $\{u_1,u_2,u_3\}$ are pairwise distinct, and aim at deducing contradictions. Let $\omega=y_1y_2\cdots y_k$ be the local normal form path from $u_2$ to $u_3$. Then $\omega\subset \bar X_3$ by Corollary~\ref{cor:convex}. For each $i$, let $y^{a_i}_{i}y^{a_i-1}_{i}\cdots y^1_{i}$ be the local normal form path from $y_i=y^{a_i}_{i}$ to $y^1_{i}=u_1$ (Proposition~\ref{prop:existence and uniqueness}), which is also a normal form path by Proposition~\ref{prop:local vs global}. 
	
	We first treat the case when $\Lambda'$ is not $\widetilde D_4$-like.
	As each of $\{u_1,u_2,u_3\}$ is extremal, they can not be pairwise adjacent in $\bar\Delta_{\Lambda,\Lambda'}$. So there must be two of them, say $u_1,u_2$, satisfies $d(u_1,u_2)\ge 2$. Then Lemma~\ref{lem:corner1} implies that $d(u_1,u_3)\ge d(u_1,u_2)\ge 2$. Now applying Lemma~\ref{lem:corner1} with roles of $u_2$ and $u_3$ exchanged implies that $d(u_1,u_2)\ge d(u_1,u_3)$. Thus $d(u_1,u_2)=d(u_1,u_3)$. By Proposition~\ref{prop:bestvina}, each vertex in $\omega$ has the same distance to $u_1$. Similarly, each vertex in a local normal form path from $u_i$ to $u_j$ (or from $u_j$ to $u_i$) has the same distance to $u_k$ for any pairwise distinct $\{i,j,k\}=\{1,2,3\}$. Lemma~\ref{lem:corner1} also implies $\tau(u_i)=n$ for $1\le i\le 3$.

	\begin{figure}
		\centering
		\includegraphics[scale=0.83]{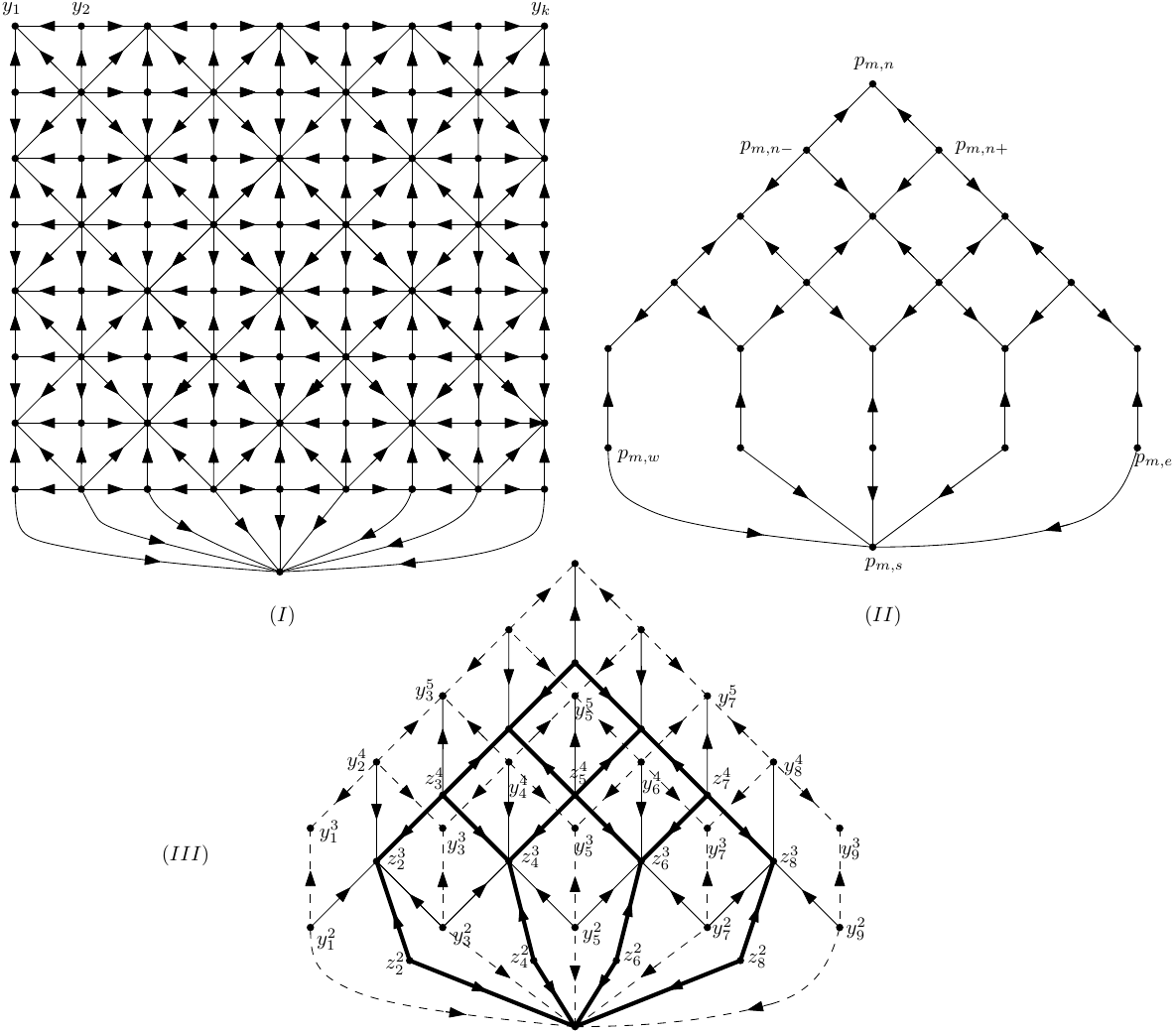}
		\caption{Some diagrams in the proof of Proposition~\ref{prop:intersection}.}
		\label{fig:round}
	\end{figure}
	%
	%Suppose two of $\{u_1,u_2,u_3\}$, say $u_1$ and $u_2$, have distance $\ge 3$ in $\bar\Delta_{\Lambda,\Lambda'}$. Then Lemma~\ref{lem:corner1} implies $d(u_1,u_3)\ge d(u_1,u_2)\ge 3$. Now applying Lemma~\ref{lem:corner1} with roles of $u_2$ and $u_3$ reversed implies that $d(u_1,u_2)\ge d(u_1,u_3)$. Hence $d(u_1,u_2)=d(u_1,u_3)$. Similarly, $d(u_2,u_3)=d(u_1,u_3)\ge 3$. Thus Lemma~\ref{lem:corner1} and Proposition~\ref{prop:bestvina} imply that each vertex in $\omega$ has the same distance to $u_1$. Similarly, each vertex in a normal form path from $u_i$ to $u_j$ (or from $u_j$ to $u_i$) has the same distance to $u_k$ for any pairwise distinct $\{i,j,k\}=\{1,2,3\}$. Lemma~\ref{lem:corner1} also implies $\tau(u_i)=n$ for $1\le i\le 3$.
	
	%Let $a=a_1=a_2=\cdots=a_k$. By Lemma~\ref{lem:strip}, $y^j_i\le y^j_{i+1}$ for $1\le j\le a$ and $i$ even, and $y^j_i\ge y^j_{i+1}$ for $1\le j\le a$ and $i$ odd. 
	
	Let $Z$ be the directed graph as in Figure~\ref{fig:round} (I) with a map $f:Z^0\to (\bar\Delta_{\Lambda,\Lambda'})^0$ such that $f$ maps the top horizontal path of $Z$ to $\omega$, and maps each vertical path to the normal form path from $y_i$ to $u_1$ for some $i$. The map $f$ is \emph{admissible}, i.e. if there is an oriented edge from $p$ to $p'$, then $f(p)\le f(p')$ (orientation of horizontal edges is derived from Lemma~\ref{lem:strip} and orientation of bevel edges follows from transitivity of the relation $<$). We slightly abuse notation and use the same name for a vertex in $Z$ and its $f$-image. 
	
	Let $m=(k-1)/2$. Let $P_m$ be the directed subgraph of $Z$ which is the full subgraph spanned by the following collection of vertices: $u_1$, $y^i_j$ with $i=2$ and $j$ odd, and $y^i_j$ for $3\le i\le (5+k)/2$ and $j=i-2,i,i+2,i+4,\ldots,k-(i-3)$. See Figure~\ref{fig:round} (II) for $P_4$. In general $P_m$ is made of a bottom layer made of $m$ hexagons, and layers of squares above them. We highlight a few special vertices $p_{m,s},p_{m,n},p_{m,n-},p_{m,n+},p_{m,w}$ and $p_{m,e}$ in $P_m$, see Figure~\ref{fig:round} (II).  Let $f_m:(P_m)^0\to \bar\Delta_{\Lambda,\Lambda'}$ be the restriction of $f$. We now construct an admissible map $f_{m-1}:(P_{m-1})^0\to (\bar\Delta_{\Lambda,\Lambda'})^0$ from $f_m$ such that (we slightly abuse notation and use the same symbol for a vertex and its image)
	\begin{enumerate}
		\item $p_{m,w}\le p_{m-1,w}$, $p_{m,e}\le p_{m-1,e}$	and $p_{m-1,s}=p_{m,s}$;
		\item $p_{m,n^-},p_{m,n},p_{m,n^+}$ are comparable to $p_{m-1,n^-},p_{m-1,n},p_{m-1,n^+}$ respectively.
	\end{enumerate}
	See Figure~\ref{fig:round} (III) for the following discussion.
	For positive interger $d$, let $Q_d$ be the set of vertices in $P_{m-1}$ that are distance $d$ from $p_{m-1,s}$. We construct the map $f_{m-1}$ inductively on $Q_d$.
	Note $\{y^2_i,y^4_{i+1},y^2_{i+2}\}$ are pairwise upper bounded for each odd $i$, and let $z^3_{i+1}$ be a common upper bound of them (Lemma~\ref{lem:big lattice}). Then vertices in $Q_2$ are mapped to $z^3_{i}$ with $i$ even. Let $z^2_i$ be the meet of $z^3_i$ and $u_1$. Then vertices in $Q_1$ are mapped to $z^2_i$ with $i$ even.
	By construction $\{z^3_i,y^5_{i+1},z^3_{i+2}\}$ are pairwise lower bounded for $i=2,4,6,\ldots,k-3$, then let $z^4_{i+1}$ be a common lower bound for them. Vertices in $Q_3$ are mapped to $z^4_i$ with $i=3,5,7,\ldots,k-2$. By repeating this procedure we obtain $f_{m-1}$ with the desired properties. 
	
	Now we show $f_{m-1}(p_{m-1,w})=z^2_2\in \bar X_2$ and $\tau(z^2_2)\ge 2$. Let $z$ be the join of $y^4_2$ and $y^2_1$, which exists by Lemma~\ref{lem:big lattice} and Lemma~\ref{lem:posets}. Then $z\le y^3_1$ and $z\le z^3_2$.  As $y^4_1y^3_1y^2_1$ is in normal form and $y^4_1\ge y^4_2$, we know $d(y^4_2,y^2_1)=2$. Hence $z\neq y^2_1$ and $\tau(z)>\tau(y^2_1)$. In particular $\tau(z)\ge 2$.
	If $z=y^3_1$, then we deduce from $x_2\sim z$, $z\le z^3_2$ and Lemma~\ref{lem:transitive} that $z^3_2\in \bar X_2$. As $z^2_2\ge y^2_1$, by Corollary~\ref{cor:convex}, $z^2_2\in \bar X_2$. If $z\neq y^3_1$, by considering the path $y^2_1zy^4_2y^4_1$ in $\lk(y^3_1,\bar\Delta_{\Lambda,\Lambda'})$ and Lemma~\ref{lem:convex}, $z\in \bar X_2$. As $\tau(z)>\tau(y^2_1)\ge 1$, we have $\tau(z)\ge 2$. As $x_2\sim z$ and $z\le z^3_2$, by Lemma~\ref{lem:transitive}, $x_2\sim z^3_2$ and we deduce $z^2_2\in \bar X_2$ as before. 
	If $\tau(z^2_2)=1$, as $z^2_2$ is the meet of $z^3_2$ and $u_1$, we must have $y^2_1=y^2_3$. As the distance between $y^2_3$ and $y_k$ in $Z$ is $k-2$, hence $d(y^2_1,u_3)<d(u_3,u_1)$, contradicting that vertices of a local normal form path between $u_1$ and $u_2$ have equal distance to $u_3$. Thus $\tau(z^2_2)\ge 2$. Similarly, we have $f_{m-1}(p_{m-1,e})=z^2_{k-1}\in \bar X_1$ and $\tau(z^2_{k-1})\ge 2$.

	By iterating the preceding procedure, we obtain a sequence of admissible maps  
	\[
	\{f_r : (P_r)^0 \to (\bar\Delta_{\Lambda,\Lambda'})^0\}_{1 \le r \le m}
	\]
	such that
	\[
	p_{r,w} \le p_{r-1,w}, \quad 
	p_{r,e} \le p_{r-1,e}, \quad 
	p_{r-1,s} = p_{r,s},
	\]
	and each of \(p_{r,n^-}, p_{r,n}, p_{r,n^+}\) is comparable with the \(p_{r-1,n^-}, p_{r-1,n}, p_{r-1,n^+}\), respectively. It follows that $ z^2_2 \le p_{1,w}$ and  
	$z^2_{k-1} \le p_{1,e}$.
	Note that
	\[
	d(p_{m,n}, y_{\frac{k+1}{2}}) \le \frac{k-5}{2}, \quad
	d(p_{m,n^-}, y_1) \le \frac{k-3}{2}, \quad
	d(p_{m,n^+}, y_k) \le \frac{k-3}{2}.
	\]
	Hence
	\[
	d(p_{1,n}, y_{\frac{k+1}{2}})
	\le d(p_{m,n}, y_{\frac{k+1}{2}}) + (m-1)
	\le \frac{k-5}{2} + \frac{k-1}{2} - 1
	= k-4.
	\]
	Similarly,
	\[
	d(p_{1,n^+}, y_k) \le k-3,
	\quad
	d(p_{1,n^-}, y_1) \le k-3.
	\]
	Since \(\tau(z^2_2) \ge 2\) and \(\tau(z^2_{k-1}) \ge 2\), Lemma~\ref{lem:transitive} implies $
	p_{1,w} \in \bar X_2$ and $p_{1,e} \in \bar X_1$.
	The vertices \(p_{1,w}, p_{1,e}, p_{1,n}\) admit a common upper bound; let \( * \) denote one such bound. We may assume \(\tau(*) = n\). Then $
	* \in \bar X_1 \cap \bar X_2,$
	and
	\[
	d(*, y_{(k+1)/2}) \le k-3, \quad
	d(*, y_k) \le k-1, \quad
	d(*, y_1) \le k-1.
	\]
	We now replace \(u_1\) by \( * \). By the minimality of \((u_1,u_2,u_3)\), we have
	\[
	d(*,u_2) = d(*,u_3) = d(u_2,u_3).
	\]
	As before, it follows that all vertices along a local normal form path from \(u_2\) to \(u_3\) are equidistant from \( * \), contradicting the bound $
	d(*, y_{(k+1)/2}) \le k-3$.
	
	%By repeating the previous process, we obtain a sequence of admissible maps $\{f_r:(P_r)^0\to (\bar \Delta_{\Lambda,\Lambda'})^0\}_{1\le r\le m}$ such that $$p_{r,w}\le p_{r-1,w}, p_{r,e}\le p_{r-1,e},p_{r-1,s}=p_{r,s}$$ and $p_{r,n^-},p_{r,n},p_{r,n^+}$ are comparable to $p_{r-1,n^-},p_{r-1,n},p_{r-1,n^+}$ respectively. Then $z^2_2\le p_{1,w}$ and $z^2_{k-1}\le p_{1,e}$. Note that $$d(p_{m,n},y_{\frac{k+1}{2}})\le \frac{k-5}{2},d(p_{m,n-},y_1)\le \frac{k-3}{2}, d(p_{m,n+},y_k)\le \frac{k-3}{2}.$$ Thus $$d(p_{1,n},y_{\frac{k+1}{2}})\le d(p_{m,n},y_{\frac{k+1}{2}})+m-1\le \frac{k-5}{2}+\frac{k-1}{2}-1=k-4.$$ Similarly, $$d(p_{1,n+},y_k)\le k-3, d(p_{1,n-},y_1)\le k-3.$$ As $\tau(z^2_2)\ge 2$ and $\tau(z^2_{k-1})\ge 2$, Lemma~\ref{lem:transitive} implies that $p_{1,w}\in \bar X_2$ and $p_{1,e}\in \bar X_1$. As $p_{1,w},p_{1,e},p_{1,n}$ are pairwise upper bounded, let $*$ be their common upper bound. We can assume $\tau(*)=n$. Then $*\in \bar X_2\cap \bar X_1$ and $d(*,y_{(k+1)/2})\le k-3$, $d(*,y_k)\le k-1$ and $d(*,y_1)\le k-1$. Now we replace $u_1$ by $*$. By minimality of $(u_1,u_2,u_3)$, we know $d(*,u_2)=d(*,u_3)=d(u_2,u_3)$ and we deduce as before that vertices in a normal form from $u_2$ to $u_3$ have the same distance to $*$, which contradicts that  $d(*,y_{(k+1)/2})\le k-3$.
	
	It remains to consider $\Lambda'$ is $\widetilde D_4$-like. Lemma~\ref{lem:corner1} and Proposition~\ref{prop:bestvina} imply that for $\{i,j,k\}$ pairwise distinct, if $d(u_i,u_j)\ge 2$ and $d(u_j,u_k)\ge 2$, then $d(u_i,u_j)=d(u_i,u_k)$; if in addition $d(u_j,u_k)\ge 3$, then all vertices in a local normal form path from $u_j$ to $u_k$ or from $u_k$ to $u_j$ have the same distance to $u_i$. Thus if among $\{d(u_1,u_2),d(u_2,u_3),d(u_1,u_3)\}$, one of them is $\ge 3$, and another one is $\ge 2$, then we repeat the previous argument to obtain a contradiction. By triangle inequality, it remains to consider $d(u_i,u_j)\le 2$ for any $1\le i\neq j\le 3$.

	Suppose two of $\{u_1,u_2,u_3\}$, say $u_2,u_3$, are adjacent in $\bar\Delta_{\Lambda,\Lambda'}$. If $d(u_1,u_2)=d(u_1,u_3)$, then $y^{a_1}_1<y^{a_1-1}_1$ if and only if $y^{a_2}_2<y^{a_2-1}_2$. Thus one of $u_2,u_3$ is not extremal, contradiction. Now we assume without loss of generality that $2=d(u_1,u_2)>d(u_1,u_3)=1$. Lemma~\ref{lem:corner1} implies $u_2$ is of type $\hat b_m$ and $u_3$ is of type $\hat b_{m-1}$. We assume $\tau(y^{a_1-1}_1)=4$ (as $\tau(y^{a_1-1}_1)\le 3$ and Lemma~\ref{lem:transitive} imply $u_3\sim x_2$, which gives $u_3\in X_1\cap X_2\cap X_3$). Thus $u_1$ is not of type $\hat b_m$. Now we apply Lemma~\ref{lem:corner1} with the role of $u_1,u_2,u_3$ in the lemma replaced by $u_2,u_1,u_3$ respectively, and deduce that $d(u_2,u_3)\ge d(u_2,u_1)$, contradiction.
	
	Suppose $d(u_i,u_j)=2$ for any $i\neq j$. By previous discussion, it suffices to consider the case that up to permutations of $\{u_1,u_2,u_3\}$, vertices of $\omega$ are not at constant distance from $u_1$. Assume $d(y_2,u_1)=1$. Lemma~\ref{lem:corner1} implies that $u_2$ and $u_3$ have type $\hat b_m$ and $y_2$ has type $\hat b_{m-1}$. We assume $y^{a_1-1}_1\neq y_2$, otherwise $x_2\sim y_2$ and $y_2\in X_2\cap X_3$, which reduces to the case in the previous paragraph. Similarly we assume $y^{a_3-1}_1\neq y_2$. In particular $\tau(y^{a_1-1}_1)\ge 2$ and $\tau(y^{a_3-1}_1)\ge 2$.
	Then $\{y^{a_1-1}_1,y^{a_3-1}_1,x_3\}$ is pairwise upper bounded in $((\bar\Delta^{y_2}_{\Gamma,\Gamma'})^0,<)$ with $\bar\Delta^{y_2}_{\Gamma,\Gamma'}$ defined in Assumption~\ref{assum:1}, and their common upper bound is contained in $X_1\cap X_2\cap X_3$. 
\end{proof}

In the following discussion, we will often use the following simple observation: for any simplex $\sigma\subset\Delta_\Lambda$ of type $\hat T$, and for induced subdiagrams $\Lambda_1\subset\Lambda_2$ with $\Lambda_1$ connected and $T\cap \Lambda_2=\emptyset$, we have $\Delta^\sigma_{\Lambda_2,\Lambda_1}=\Delta^\sigma_{\Lambda'_2,\Lambda_1}$ where $\Lambda'_2$ is the connected component of $\Lambda_2$ that contains $\Lambda_1$. 

\begin{cor}
	\label{cor:robust0}
	Under Assumption~\ref{assum:1}, suppose that:
	\begin{enumerate}
		\item for any vertex $s\in \Lambda'\setminus\Lambda_0$ and any vertex $v$ of $\hat s$, $(\bar\Delta^v_{\Lambda\setminus\{s\},\Lambda_0})^0$ is bowtie-free; moreover, for any terminal vertex $r\in \Lambda_0$ other than $s_1$ or $b_1$ and any vertex $v$ of type $\hat r$,  $\Delta_{\Lambda\setminus\{r\},\Lambda_0\setminus\{r\}}$ is bowtie free;
		%if $\Lambda_0$ is $D_m$-like, then we also require $\Delta_{\Lambda\setminus\{b_{m-1}\},\Lambda_0\setminus\{b_{m-1}\}}$ is bowtie free;
		\item for each leaf vertex $s\in \Lambda'$ and each vertex $v$ of type $\hat s$, $\Delta^v_{\Lambda\setminus \{s\},\Lambda'\setminus\{s\}}$ satisfies the labeled 4-cycle condition.
	\end{enumerate}
	Then	$(\bar\Delta_{\Lambda,\Lambda_0})^0$ is bowtie-free, and $\Delta_{\Lambda,\Lambda_0}$ satisfies the labeled 4-cycle condition. Assume in addition that
	\begin{enumerate}[resume]
		%\item when $\Lambda'$ is $\widetilde D_k$-like for some $k$, we assume in addition that $\Delta_{\Lambda\setminus\{t_1\},\Lambda'\setminus\{t_1\}}$ satisfies the labeled $4$-cycle condition;
		\item if $\Lambda_0$ is $D_m$-like, then for any connected proper subdiagram $\Omega'\subset \Lambda$ containing $Q=b_{m-2}b_{m-1}b_m$ and any nonempty simplex $\sigma$ of type $\hat T$ with $T=S\setminus \Omega'$, the $(b_{m-1},b_m)$-subdivision of $\Delta^\sigma_{\Omega',Q}$ with the induced poset structure from $(\bar\Delta_{\Lambda,\Lambda_0})^0$ is upward flag.
	\end{enumerate}
	Then	$(\bar\Delta_{\Lambda,\Lambda_0})^0$ is upward flag.
\end{cor}

\begin{proof}
	For a vertex $s\in\Lambda_0$ with $s\neq s_1,b_1$, let $\Gamma_s$ be the connected component of $\Lambda_0\setminus\{s\}$ containing $s_1$ or $b_1$.
	%let $\Theta_s$ (resp.\ $\Lambda_s$) denote the connected component of $\Theta\setminus\{s\}$ (resp.\ $\Lambda\setminus\{s\}$) that contains $s_1$ when $\Lambda_0$ is $B_m$-like, and that contains $b_1$ when $\Lambda_0$ is $D_m$-like. Lemma~\ref{lem:link} implies that $\Delta_{\Lambda\setminus\{s\},\Theta\setminus\{s\}}\cong \Delta_{\Lambda_s,\Theta_s}$ when $s$ is a leaf vertex of $\Theta$.
	Then Assumption 1 implies that $\Delta^v_{\Lambda\setminus\{s\},\Gamma_s}$ is bowtie free. This is clear if $s$ is terminal in $\Lambda_0$. Now assume $s$ is not terminal in $\Lambda_0$.
	Then there is a vertex $t$ such that either $t\in \Lambda'\setminus\Lambda_0$ or $t$ is terminal in $\Lambda_0$ such that $s$ separates $\Gamma_s$ from $t$. 
	Let $\Lambda_s$ be the connected component of $\Lambda\setminus\{s\}$ containing $\Gamma_s$. Then $\Delta^v_{\Lambda\setminus\{s\},\Gamma_s}= \Delta^v_{\Lambda_s,\Gamma_s}$. Note that $\Lambda_s\subset \Lambda\setminus\{t\}$ and $\Gamma_s\subset \Lambda_0\setminus\{t\}$. Let $w$ be a vertex of $\Delta_S$ of type $\hat t$ with $w\sim v$. 
	Then $\Delta_{\Lambda_s,\Gamma_s}$ can be realized as a join factor of $$
	\lk(\{v\}\ast\{w\},\Delta_S))=\lk(v,(\lk(w,\Delta_S)))=\lk(v,\Delta^w_{\Lambda\setminus\{t\},\Lambda_0\setminus\{t\}}).$$  Thus $\Delta^v_{\Lambda_s,\Gamma_s}$ is bowtie free by Assumption 1.
	
	Let $\mathcal P=((\bar\Delta_{\Lambda,\Lambda_0})^0,<)$, and let $\tau_0$ be the rank function on $\mathcal P$ such that type $\hat b_1$ or $\hat s_1$ elements have rank 1. First we show $\mathcal P$ is bowtie free.
	We only treat the case of Figure~\ref{fig:core} (3), as the other cases are similar and simpler. By Lemma~\ref{lem:bowtie free criterion} and Assumption 1,
	it suffices to prove for any $p_1,p_2,p_3,p_4$ be elements in $\mathcal P$ such that $p_1,p_3$ have the same type, $p_2$ and $p_4$ are maximal, and $p_2$ and $p_4$ are common upper bounds for $\{p_1,p_3\}$, there is $p\in \mathcal P$ such that $\{p_1,p_3\}\le p\le \{p_2,p_4\}$.  
	By Assumption 2, Proposition~\ref{prop:propagation} (3) and Proposition~\ref{prop:propagation2} (3), $\Delta_{\Lambda,\Lambda'}$ satisfies the labeled 4-cycle condition. By Remark~\ref{rmk:inherit}, $\Delta_{\Lambda,\Lambda'\cap\Lambda_0}$ satisfies the labeled 4-cycle condition. By Lemma~\ref{lem:subdivision bowtie free}, we only need to consider the case when $p_1,p_3$ have type $\hat t$ with $t\notin \Lambda'$.

	%By Lemma~\ref{lem:4wheel}, it suffices to show $\Delta_{\Lambda,\Theta'}$ is bowtie free for each maximal linear subgraph $\Theta'\subset \Theta$. 

	%By Remark~\ref{rmk:inherit} and Lemma~\ref{lem:4wheel}, $\Delta_{\Lambda,\Theta'}$ is bowtie free when $\Theta'=\Theta_{r_1,r_2}$, i.e. the linear subdiagram from $r_1$ to $r_2$, or $\Theta'=\Theta_{b_{m-1},r_i}$ for $i=1,2$.
	
	%Next we consider $\Theta'=\Theta_{b_1,r_1}$. Let $\mathcal Q$ be the poset on $(\Delta_{\Lambda,\Theta'})^0$ such that vertices of type $\hat b_1$ are minimal. For each maximal element $p\in \mathcal Q$, $\mathcal Q_{<p}$ is the vertex set of $\lk(p,\Delta_{\Lambda,\Theta'})$, which is isomorphic to $\Delta_{\Lambda_{r_1},\Theta_{r_1}}$ by Lemma~\ref{lem:link}. 

	%As $\Delta_{\Lambda,\Lambda'}$ is bowtie free, 

	Let $\Theta=\Lambda'\cup\Lambda_0$.  
	We claim if there is an edge path $\omega=y_1\cdots y_\ell$ from $p_2$ to $p_4$ in $\Delta_{\Lambda,\Theta}$ such that $y_i$ is adjacent to both $p_1$ and $p_3$ for each $i$, then the desired $p$ exists. In the following discussion, we assume $\omega$ is the shortest such path. Suppose $y_i$ has type $\hat u_i$. Then $u_i$ does not separate $u_{i-1}$ from $u_{i+1}$, otherwise $y_{i-1}$ and $y_{i+1}$ are adjacent by Lemma~\ref{lem:transitive} and we can shorten $\omega$. Similarly, $y_i\neq y_{i+2}$ and they are not adjacent for any $1\le i\le \ell-2$, otherwise we can shorten $\omega$. Note that $u_1=b_m$ or $b_{m-1}$. 
	
	We only treat the case $u_1=b_m$ as the other case is similar. We first show that up to replacing $y_2$, we can assume $d(u_2,b_1)<d(u_1,b_1)$. Indeed, if $d(u_2,b_1)\ge d(u_1,b_1)$, then $u_2=b_{m-1}$ or $u_2\in \Lambda'\setminus\Lambda_0$. Let $\Lambda_2$ be the line segment from $b_1$ to $b_m$.
	As $\Delta^{y_2}_{\Lambda\setminus\{u_2\},\Lambda_2}$ is bowtie free (Assumption 1) and $y_1$ is a common upper bound for $\{p_1,p_3\}$ in $(\Delta^{y_2}_{\Lambda\setminus\{u_2\},\Lambda_2})^0$ with the obvious partial order, we 
	%If $u_2=b_{m-1}$, then consider the 4-cycle $y_1p_1y_3p_3$ in $$\Delta_1=\Delta_{\Lambda\setminus\{b_{m-1}\},\Lambda_0\setminus\{b_{m-1}\}}.$$	As $\Delta_1$ is bowtie free and $y_1$ is a common upper bound for $\{p_1,p_3\}$ in  $(\Delta_1)^0$ with the obvious partial order, 
	let $y'_2\in (\Delta^{y_2}_{\Lambda\setminus\{u_2\},\Lambda_2})^0$ be the join of $p_1$ and $p_3$, which exists by Lemma~\ref{lem:posets}. Then $y'_2\sim y_1$. If $y'_2=y_2$, then $d(u_2,b_1)\le d(u_1,b_1)$ and the equality holds when $y_2=y_1$ (as $y_2\le y_1$ in $\Delta^0_{\Lambda\setminus\{u_2\},\Lambda_2}$), which allows us to shorten $\omega$. Thus $d(u_2,b_1)< d(u_1,b_1)$. It remains to consider $y'_2\neq y_2$.
	By condition C$_4$ applying to the 4-cycle $y'_2p_1y_3p_3\subset \Delta_{\Lambda\setminus\{u_2\}}$, there is a vertex $y''_2\in \Delta^{y_2}_{\Lambda\setminus\{u_2\}}$ of the same type as $y'_2$ such that $y''_2$ is adjacent to each of $\{p_1,y_3,p_3\}$. As $y''_2$ is a common upper bound for $\{p_1,p_3\}$ in $(\Delta^{y_2}_{\Lambda\setminus\{u_2\},\Lambda_0\setminus\{u_2\}})^0$, we must have $y''_2=y'_2$. So $y'_2$ is adjacent to $y_3$.
	Suppose $y'_2$ is of type $\hat u'_2$. By construction $d(b_1,u'_2)\le d(b_1,u_1)$. If $u'_2=u_1$, then $y'_2=y_1$ and $y_1$ and $y_3$ are adjacent, contradicting that $\omega$ is shortest. So by replacing $y_2$ by $y'_2$, we can assume $d(u_2,b_1)<d(u_1,b_1)$.

	Now we show $\ell=3$ by arguing that the existence of $y_4$ leads to a contradiction.	
	If $d(u_3,b_1)<d(u_2,b_1)$, then Lemma~\ref{lem:transitive} implies that $y_1$ and $y_3$ are adjacent, contradicting that $\omega$ is shortest. So $d(u_3,b_1)>d(u_2,b_1)$. 
	%We consider the 4-cycle $y_2p_1y_4p_3$ in $\Delta_{\Lambda\setminus\{u_3\}}$. 
	Let $\Lambda_3$ be the intersection of $\Lambda_0$ with the line segment from $b_1$ to the vertex adjacent to $u_3$ in $\Lambda\setminus\{u_3\}$. Then $\Delta^{y_3}_{\Lambda\setminus\{u_3\},\Lambda_3}$ is bowtie free by the first paragraph of the proof. As $y_2$ is a common upper bound for $\{p_1,p_3\}$ in $(\Delta^{y_3}_{\Lambda\setminus\{u_3\},\Lambda_3})^0$, we define $y'_3$ to be the join of $p_1$ and $p_3$ in $(\Delta^{y_3}_{\Lambda\setminus\{u_3\},\Lambda_3})^0$. Similarly to the previous paragraph, $y'_3$ is adjacent or equal to each of $\{y_2,y_4,p_1,p_3\}$. Note that $y'_3\neq y_2$ and $y'_3\neq y_4$, otherwise $\omega$ is not the shortest. We replace $y_3$ by $y'_3$. By construction, $d(u_3,b_1)<d(u_2,b_1)$, and we deduce a contradiction as before. Thus $\ell=3$ and we found the desired $p$.

	%We first arrange $t_i\in \Theta'$ for each $1<i<\ell$. Indeed, if $t_i=b_{m-1}$, then we consider induced 4-cycle $y_{i-1}p_1y_{i+1}p_3$ in $\lk(y_i,\Delta_{\Lambda,\Theta})\cong \Delta_{\Lambda_{b_{m-1}},\Theta_{b_{m-1}}}$. Lemma~\ref{lem:tripod} and assumption $(1)'$ imply that the 4-cycle $y_{i-1}p_1y_{i+1}p_3$ has a center $y'_i$ of type $t'_i$ with $t'_i\in \Theta'$, so we replace $y_i$ and $y'_i$. We can do a similar replacement if $t_i=r_2$. Thus after finitely replacement, $t_i\in \Theta'$ for $1<i<\ell$. For $1<i<\ell$, if $d(t_{i-1},b_1)< d(t_i,b_1)$ and $d(t_{i+1},b_1)> (t_i,b_1)$, then consider the induced 4-cycle $y_{i-1}p_1y_{i+1}p_3$ in $\Delta_{\Lambda_{t_i},\Theta_{t_i}}$. By assumption $(1)'$, $y_{i-1}p_1y_{i+1}p_3$ has a center $y'_i$ of type $\hat t'_i$ such that $t'_i\in \Theta'$, $d(t'_i,b_1)<d(t_{i-1},b_1)$ and $d(t'_i,b_1)<d(t_{i+1},b_1)$. Thus we can assume the function $i\to d(t_i,b_1)$ first strictly decreases, then strictly increases (the decreasing part, or the constant part, or the increasing part could be trivial). As $d(t_1,b_1)=d(t_{\ell},b_1)$, $\ell=3$ and $d(t_2,b_1)<d(t_1,b_1)$, which proves the claim.
	
	Now we produce the path $\omega$ in the claim. Note that $p_2,p_4\in \bar\Delta_{\Lambda,\Lambda'}$. For $i=1,3$, let $Y_i$ be the full subcomplex of $\Delta_{\Lambda,\Lambda'}$ spanned by vertices that are adjacent to $p_i$ in $\Delta_\Lambda$, and let $\bar Y_i$ be the subcomplex of $\bar\Delta_{\Lambda,\Lambda'}$ corresponding to $Y_i$. Then $\bar Y_i$ satisfies the conclusion of Corollary~\ref{cor:convex}. Moreover, $p_2,p_4\in\bar Y_1\cap \bar Y_3$. 
	%As $p_2$ and $p_4$ are extremal in $\bar\Delta_{\Lambda,\Lambda'}$, the existence of $\omega$ follows from Proposition~\ref{prop:convex subcomplex} and Proposition~\ref{prop:existence and uniqueness}. This finishes the case of $\Theta'=\Theta_{b_1,r_1}$. The case of $\Theta'=\Theta_{b_1,r_2}$ is identical.
	Note that there is an edge path $\omega=y_1y_2\ldots y_\ell\subset \bar Y_1\cap \bar Y_3$ joining $p_2$ and $p_4$. Indeed, if one of $p_2$ and $p_4$ is extremal in $\bar\Delta_{\Lambda,\Lambda'}$, then this follows from Corollary~\ref{cor:convex} and Proposition~\ref{prop:existence and uniqueness}. If none of $p_2$ and $p_4$ are extremal, then let $p'_2$ be a vertex with $\tau(p'_2)=n$ and $p'_2\ge p_2$. We deduce from Lemma~\ref{lem:transitive} that $p'_2\sim p_1,p_3$. Thus $p'_2\in \bar Y_1\cap \bar Y_3$. Corollary~\ref{cor:convex} and Proposition~\ref{prop:existence and uniqueness} imply $p'_2$ is connected to $p_4$ by an edge path in $\bar Y_1\cap \bar Y_3$, which gives the desired path $\omega$. This finishes the proof that $\mathcal P$ is bowtie free. As we also know $\Delta_{\Lambda,b_{m-1}b_{m-2}b_m}\subset \Delta_{\Lambda,\Lambda'}$ satisfies labeled 4-cycle condition by discussion before, it follows that $\Delta_{\Lambda,\Lambda_0}$ satisfies labeled 4-cycle condition.

	Now we verify $\mathcal P$ is upward flag, and let $\tau_0$ be the rank function on $\mathcal P$ as before.  By Corollary~\ref{cor:upward flag criterion}, it suffices to show that for pairwise upper bounded elements $p_1,p_2,p_3\in \mathcal P$ with $1\le \tau_0(p_1)=\tau_0(p_2)=\tau_0(p_3)\le m-1$, $\{p_1,p_2,p_3\}$ have a common upper bound in $\mathcal P$. If $\{p_1,p_2,p_3\}$ are $\Lambda'$-extremal (which means their type labels lie in $ (\Lambda_0\setminus\Lambda')\cup\{t_1\}$), then let $X_i$ be the full subcomplex of $\Delta_{\Lambda,\Lambda'}$ spanned by vertices that are adjacent or equal to $p_i$. As $\{p_1,p_2,p_3\}$ is pairwise upper bounded in $\mathcal P$, the assumption of Proposition~\ref{prop:intersection} is satisfied. So $X_1\cap X_2\cap X_3\neq\emptyset$. Take a vertex $z\in X_1\cap X_2\cap X_3$ and suppose $z$ has type $\hat t_z$. Then $z$ is adjacent to $\{p_1,p_2,p_3\}$ in $\Delta_\Lambda$. If $t_z\in \Lambda_0\cap\Lambda'$, then $z$ is a common upper bound of $\{p_1,p_2,p_3\}$ in $\mathcal P$. Suppose $t_z\notin \Lambda_0$. Let $\Omega$ be the connected component of $\Lambda\setminus\{t_z\}$ that contains $\Lambda_0$.
	Let $q_i$ be a common upper bound for $p_i$ and $p_{i+1}$ in $\mathcal P$ for $i\in \mathbb Z/3\mathbb Z$. 
	By applying C$_4$ to $p_iq_ip_{i+1}z$, we can assume $q_i\sim z$ up to replacing $q_i$ by another element of the same type. By Lemma~\ref{lem:link} applied to $\lk(z,\Delta_\Lambda)$, we can view $\{p_1,p_2,p_3\}$ as pairwise upper bounded elements in $(\bar\Delta^z_{\Omega,\Lambda_0})^0$, and use Assumption~\ref{assum:1} to produce a common upper bound for them.
	
	It remains to consider the case that $\{p_1,p_2,p_3\}$ are not $\Lambda'$-extremal. Then we are either in Figure~\ref{fig:core} (5) with $p_1,p_2,p_3$ being type $\hat s_{m-1}$, in which case the existence of their common upper bound follows from the flagness of $\bar\Delta_{\Lambda,\Lambda'}$ (Lemma~\ref{lem:big lattice}), or $\Lambda_0$ is $D_m$-like. If $\tau_0(p_i)=m-1$, then each $p_i$ is the middle point of an edge of $\Delta_\Lambda$ between a type $\hat b_m$ vertex and a type $\hat b_{m-1}$ vertex, and the existence of common upper bound is clear. The only case left is $\tau_0(p_i)=m-2$. Let $q_i$ be as before, and we can assume each $q_i$ is of type $\hat b_m$ or type $\hat b_{m-1}$. As $\Delta_{\Lambda,b_{m-1}b_{m-2}b_m}$ is bowtie free, by Lemma~\ref{lem:weakly flag equivalent}, it suffices to consider all the $q_i$'s are of type $\hat b_m$, or all of them have type $\hat b_{m-1}$. In the former case, as $(\bar\Delta_{\Lambda,\Lambda'})^0$ is flag (Lemma~\ref{lem:big lattice}), $\{p_1,p_2,p_3\}$ has a common upper bound $z$ in $((\bar\Delta_{\Lambda,\Lambda'})^0,<)$, and we conclude in the same way as the previous paragraph. If all $q_i$'s are of type $\hat b_{m-1}$, then Lemma~\ref{lem:big lattice} implies that $\{p_1,p_2,p_3\}$ has a common lower bound $z$ in $((\bar\Delta_{\Lambda,\Lambda'})^0,<)$. We can assume $\tau(z)=1$. If $z$ is of type $\hat b_{m-1}$, then the bowtie free property of $\Delta_{\Lambda,b_{m-1}b_{m-2}}$ implies that $z=q_1=q_2=q_3$, which gives a common upper bound for $\{p_1,p_2,p_3\}$ in $\mathcal P$. Now assume $z$ is of type $\hat t_1$. By using condition C$_4$ as in the previous paragraph, we assume $p_1q_1p_2q_2p_3q_3$ is contained in $\Delta^z_{\Omega',b_{m-2}b_{m-1}b_m}$, where $\Omega'$ is the connected component of $\Lambda\setminus\{t_1\}$ containing $b_{m-1}$. By Assumption 3, vertex set of the $(b_{m-1},b_m)$-subdivision of $\Delta^z_{\Omega',b_{m-2}b_{m-1}b_m}$ with induced partial order from $\mathcal P$ is upward flag, which gives the desired common upper bound for $\{p_1,p_2,p_3\}$. 
\end{proof}

%If exactly two of $\{q_1,q_2,q_3\}$, say $q_1$ and $q_3$, are of type $\hat b_{m-1}$, then $q_1<p_2<q_2$ and $q_3<p_3<q_2$ in $((\bar\Delta_{\Lambda,\Lambda'})^0,<)$. So $q_2$ and $p_1$ are two common upper bounds for $\{q_1,q_3\}$. Let $q$ be the join of $q_1$ and $q_3$ in $((\bar\Delta_{\Lambda,\Lambda'})^0,<)$, which exists by Lemma~\ref{lem:big lattice}.
%If $\tau(q)=2$, then $q_1=q_3$ and we are done. If $\tau(q)>2$, then $q=p_1$ and $p_1\le q_2$ in $((\bar\Delta_{\Lambda,\Lambda'})^0,<)$. Hence $q_2$ is a common upper bound for $\{p_1,p_2,p_3\}$ in $\mathcal P$.  If exactly two of $\{q_1,q_2,q_3\}$, say $q_1$ and $q_3$, are of type $\hat b_{m}$, then $q_1>p_2>q_2$ and $q_3>p_3>q_2$ in $((\bar\Delta_{\Lambda,\Lambda'})^0,<)$. Then $q_2$ and $p_1$ are two common lower bounds for $\{q_1,q_3\}$ in $((\bar\Delta_{\Lambda,\Lambda'})^0,<)$. As $\tau(p_1)=\tau(q_1)-1$, $p_1$ is the meet of $q_2$ and $p_1$. Thus $p_1\ge q_2$ in $((\bar\Delta_{\Lambda,\Lambda'})^0,<)$. It follows that $q_2$ is a common upper bound for $\{p_1,p_2,p_3\}$ in $\mathcal P$. This exhausts all the cases.

\subsection{Propagation of robustness}
\label{subsec:propagation}
Let $\Lambda$ be a doubly marked tree diagram. An edge of $\Lambda$ is \emph{marked} if it belongs to $E_m$. A $D_n$-like subdiagram $\Lambda_0$ of $\Lambda$ is \emph{weakly $\Lambda$-atomic}, if  all the requirements for being $\Lambda$-atomic hold for $\Lambda_0$, except that one of $b_1b_3$ and $b_2b_3$ is allowed to be marked (we label vertices of $\Lambda_0$ as in Figure~\ref{fig:BD} left). 

A \emph{link of $\Delta_\Lambda$ modeled on an induced subdiagram $\Lambda_1$} is defined to be the link of a simplex $\sigma$ of type $\hat T$ in $\Delta_\Lambda$ such that the induced subdiagram on $S\setminus T$ is $\Lambda_1$. Here we allow $\sigma=\emptyset$, in which case $\Lambda_1=\Lambda$.

A \emph{plain path} is a path made of plain edges.
We say $\Delta_\Lambda$ satisfies \emph{property (P)} if all of its links modeled on plain paths are bowtie free. If $\Delta_\Lambda$ has property (P), then the links of its simplices also have property (P).

\begin{lem}
	\label{lem:weak atomic}
	Let $\Lambda$ be a doubly marked tree diagram. Suppose that the link of any (possibly) simplex in $\Delta_\Lambda$ is atomic BD-robust. Suppose $\Delta_\Lambda$ has property (P).

	Let $\Lambda_0$ be a weakly $\Lambda$-atomic $D_n$-like subdiagram with vertices labeled as Figure~\ref{fig:BD} left.
	Then the $(b_1,b_2)$-subdivision of $\Delta_{\Lambda,\Lambda_0}$, with the poset on its vertex set as in Definition~\ref{def:subdivision}, is bowtie free and downward flag. Moreover, $\Delta_{\Lambda,\Lambda_0}$ satisfies labeled 4-cycle condition.
\end{lem}

Note that the assumption on the link implies that for any simplex $\sigma\subset \Delta_S$ of type $\hat T$, $\Delta^\sigma_\Gamma$ is atomic BD-robust for any connected component $\Gamma$ of $\Lambda\setminus T$.

In the following proof, we will use $\Delta^*_{\Lambda\setminus\{s\},\Theta}$ to denote the relative complex in the link of an arbitrary vertex $*\in \Delta_S$ of type $\hat s$.

\begin{proof}
	We induct on the number of vertices in $\Lambda$. The base case when $\Lambda$ has one vertex is trivial. We assume without loss of generality that $b_1b_3$ is marked and $b_2b_3$ is plain. If $n=3$, then $\Lambda_0$ is $\Lambda$-atomic $B_3$-like, hence $\Delta_{\Lambda_0}$ is bowtie free and upward flag, with the partial order on its vertex set induced by $b_2<b_3<b_1$. By Lemma~\ref{lem:d3 upward flag}, the $(b_1,b_2)$-subdivision of $\Delta_{\Lambda,\Lambda_0}$ with the poset on its vertex set in Definition~\ref{def:subdivision} is bowtie free and downward flag.
	We assume $n\ge 4$ for the rest of the discussion.
	
	Let $\Lambda'=b_1b_3\cup b_2b_3\cup b_4b_3$. Let $\bar\Delta_{\Lambda,\Lambda'}$ be the $(b_2,b_4)$-subdivision of $\Delta_{\Lambda,\Lambda'}$. 
	By Lemma~\ref{lem:link} and the assumption of Lemma~\ref{lem:weak atomic}, $\bar\Delta_{\Lambda,\Lambda'}$ is $\wtC$-like. Then the pair $(\Lambda',\Lambda_0)$ belongs to the configuration in Figure~\ref{fig:core} (4), and the lemma follows from  Corollary~\ref{cor:robust0}, modulo justifying all the assumptions of this corollary. Indeed, Assumption~\ref{assum:1} does not apply as $\Lambda'$ is not $\widetilde D_4$-like and $\Lambda'\setminus\Lambda_0=\emptyset$. Assumption 1 of Corollary~\ref{cor:robust0} corresponds to checking bowtie free condition for $\Delta^*_{\Lambda\setminus\{b_2\},\Lambda_0\setminus\{b_2\}}$ and $\Delta^*_{\Lambda\setminus\{b_1\},\Lambda_0\setminus\{b_1\}}$. The former follows from  that $\Lambda_0\setminus\{b_2\}$ is $\Lambda\setminus\{b_2\}$-atomic $B_{m-1}$-like. The latter follows from property (P) and Remark~\ref{rmk:bowtie free} if the connected component of $\Lambda\setminus\{b_1\}$ containing $\Lambda_0\setminus\{b_1\}$ is a path made of plain edges. Otherwise $\Lambda_0\setminus\{b_1\}$ is contained in a connected subdiagram of $\Lambda\setminus\{b_1\}$ which is (weakly) $\Lambda\setminus\{b_1\}$-atomic $D_k$-like or $B_k$-like, hence $\Delta^*_{\Lambda\setminus\{b_1\},\Lambda_0\setminus\{b_1\}}$ is bowtie free by induction assumption.
	In Assumption 2 of Corollary~\ref{cor:robust0}, $\Lambda'\setminus\{s\}$ is $\Lambda\setminus\{s\}$-atomic $D_3$-like or $B_3$-like, hence this assumption holds true.
	%$\Theta_s=\Lambda_0\setminus\{b_2\}$ or $\Theta_s=\Lambda_0\setminus\{b_1\}$. In the former case $\Theta_s$ is $\Lambda_s$-atomic $B_m$-like, hence satisfies the labeled 4-cycle condition by our assumption and Lemma~\ref{lem:4wheel}; In the later case, either $\Theta_s$ is contained a $\Lambda_s$-atomic or weakly $\Lambda_s$-atomic $D_m$-like subdiagram of $\Lambda_s$, or $\Theta_s$ is contained in a $\Lambda_s$-atomic $B_m$-like subdiagram of $\Lambda_s$, or $\Lambda_s$ is of type $A_m$ for some $m$. In either cases $\Delta_{\Lambda_s,\Theta_s}$ satisfies the labeled 4-cycle condition by induction hypothesis,  Theorem~\ref{thm:bowtie free}, and Remark~\ref{rmk:inherit}. 
	Assumption 3 of Corollary~\ref{cor:robust0} corresponds to considering the $(b_1,b_2)$-subdivision of $\Delta^\sigma_{\Omega',b_1b_3b_2}$ for proper connected subdiagram $\Omega'\subset \Lambda$ and simplex $\sigma$ of type $\hat T$ with $T=S\setminus\Omega'$. As $b_1b_3b_2$ is a $\Omega'$-atomic $B_3$-like, this assumption holds by Lemma~\ref{lem:d3 upward flag}.
\end{proof}

%\begin{lem}
%Let $\Lambda$ be a tree Coxeter diagram such that any proper induced subdiagram is atomic BD-robust and $\Lambda$ contains a $\widetilde C$-core. Then for each $\Lambda_0\subset \Lambda$ being atomic $B_m$-like or $D_m$-like, there is a robust $\widetilde C$-core $\Lambda'$ such that $\Lambda'\setminus \Lambda_0$ is connected and each vertex in $\Lambda_0\setminus\Lambda'$ is $\Lambda'$-extremal. Moreover, if $\Lambda_0$ is not contained in $\Lambda'$, then we can assume we are in one of Figure~\ref{fig:core} (1)-(10) such that 
%\begin{enumerate}
%	\item  the embedding $\Theta=\Lambda_0\cup\Lambda'\to \Lambda$ preserves the valence of all interior vertices of $\Theta$ except possibly at the square vertex in Figure~\ref{fig:core};
%	\item in Figure~\ref{fig:core} (1)-(2), (4)-(10), the edges of $\Theta$ without any labels are labeled by $3$, the edges with a $+$ sign can be labeled by any number $\ge 3$;
%	\item in Figure~\ref{fig:core} (3), when $\Lambda'$ is $\widetilde D_4$-like, the edge of $\Lambda'$ not contained in $\Lambda_0$ can have any label $\ge 3$, when $\Lambda'$ is $\widetilde D_k$-like for $k\ge 5$ all edges of $\Theta$ are labeled by $3$.
%\end{enumerate}
%\end{lem}

\begin{lem}
	\label{lem:robust C core}
Let $\Delta_\Lambda$ be modeled on a doubly marked tree diagram $\Lambda$ such that the link of each nonempty simplex is atomic BD-robust. Suppose $\Delta_\Lambda$ satisfies property (P). Suppose that $\Lambda$ contains a $\widetilde C$-core. Then for each subdiagram $\Lambda_0$ that is $\Lambda$-atomic and $B_m$-like or $D_m$-like, there exists a $\widetilde C$-core $\Lambda'$ robust with respect to $\Delta_\Lambda$ such that every vertex in $\Lambda_0\setminus \Lambda'$ is $\Lambda'$-extremal. 
	
	If $\Lambda_0\subset \Lambda'$, then we may assume that the pair $(\Lambda_0,\Lambda')$ appears in one of the configurations shown in Figure~\ref{fig:core1}\,(11)--(19).
	If $\Lambda_0\nsubseteq \Lambda'$, then we may assume that $(\Lambda_0,\Lambda')$ appears in one of the configurations shown in Figure~\ref{fig:core1}\,(1)--(10). Moreover, the following conditions hold:
	\begin{enumerate}
		\item In Figure~\ref{fig:core1}\,(1)--(10), let $\Theta$ be the smallest subtree of $\Lambda$ containing $\Lambda_0$ and $\Lambda'$. Then the embedding $\Theta\hookrightarrow \Lambda$ preserves the valence of all interior vertices of $\Theta$, except possibly at the square vertex appearing in Figure~\ref{fig:core1}.
		\item In Figure~\ref{fig:core1}\,(1)--(2) and (4)--(10), all unlabeled edges of $\Theta$ are plain, while edges marked with a ``$+$'' sign maybe plain or marked.
		\item In Figure~\ref{fig:core1}\,(3), if $\Lambda'$ is $\widetilde D_4$-like, then the edge of $\Lambda'$ not contained in $\Lambda_0$ maybe plain or marked; if $\Lambda'$ is $\widetilde D_k$-like with $k\ge 5$, then all edges of $\Theta$ are plain.
		\item In Figure~\ref{fig:core1}\,(11)--(16), the embedding $\Lambda'\hookrightarrow\Lambda$ preserves the valence of all interior vertices of $\Lambda'$, except possibly at the square vertex; and all unlabeled edges of $\Lambda'$ are plain. In Figure~\ref{fig:core1}\,(12) and (14), one of the edges marked with a ``+'' sign maybe plain or marked, the other one is plain.
	\end{enumerate}
\end{lem}

\begin{figure}
	\centering
	\includegraphics[scale=0.8]{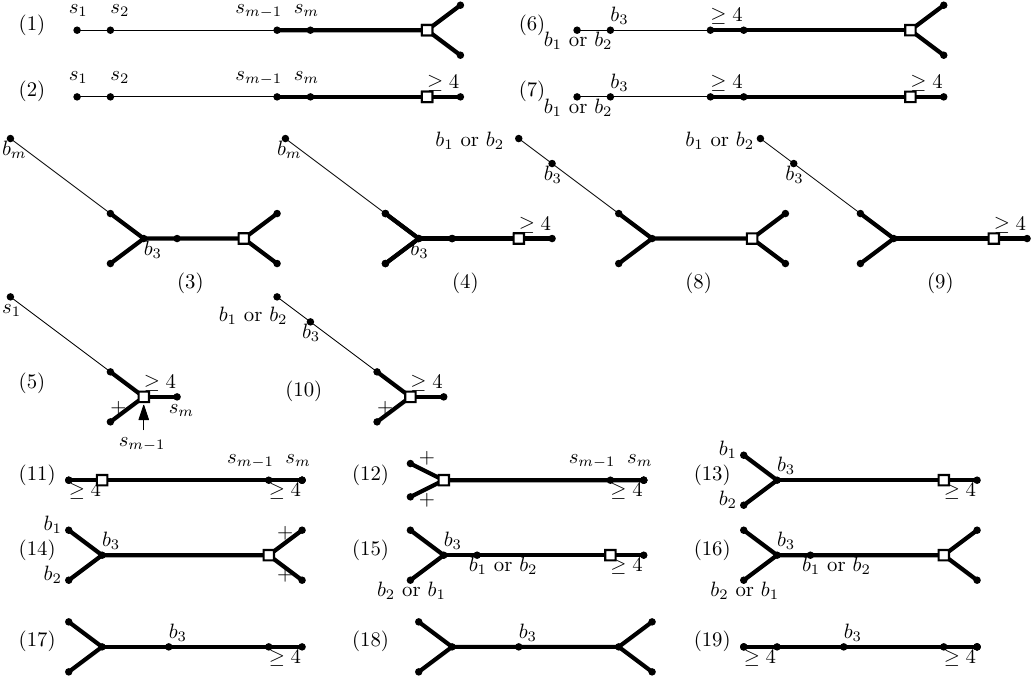}
	\caption{Vertices of $\Lambda_0$ are labeled as $s_1\ldots s_m$ with $m_{s_{m-1},s_m}\ge 4$ if $\Lambda_0$ is $B_m$-like, and vertices of $\Lambda_0$ are labeled as Figure~\ref{fig:BD} left if $\Lambda_0$ is $D_m$-like. The thickened diagram indicates the robust $\wtC$-core $\Lambda'$. In Figure~\ref{fig:core}\,(6)--(10) and (17)--(19), $\Lambda_0$ is $D_3$-like, and the center vertex $b_3$ of $\Lambda_0$ is labeled in each figure. Moreover, $b_3$ is allowed to be coincide with a leaf vertex of $\Lambda'$ in Figure~\ref{fig:core}\,(8)--(10).  In Figure~\ref{fig:core1}\,(11)--(12), $\Lambda_0$ is $B_m$-like and only vertices $s_{m-1},s_m$ of $\Lambda_0$ are indicated. In Figure~\ref{fig:core1}\,(13)--(16), $\Lambda_0$ is $D_m$-like for $m\ge 3$. Similar to Figure~\ref{fig:core}, we allow $\Lambda'$ to be $\widetilde D_4$-like in (3) and $\widetilde B_3$-like in (4).}
	\label{fig:core1}
\end{figure}
\begin{proof}
	First we assume $\Lambda_0$ is $\Lambda$-atomic $D_m$-like ($m\ge 4$). Label vertices of $\Lambda_0$ as in Figure~\ref{fig:BD} left. Let $\{e_i\}_{i=1}^3$ be the three edges of $\Lambda_0$ containing $b_3$. 
	If $b_3$ has valence $>3$ in $\Lambda$, then we choose $\Lambda'$ to be $e_1\cup e_2\cup e_3\cup e$ where $e$ is another edge of $\Lambda$ containing $b_3$. Then for each leaf vertex $t$ of $\Lambda'$, $\Lambda'\setminus\{t\}$ is $\Lambda\setminus\{t\}$-atomic or weakly $\Lambda\setminus\{t\}$-atomic $D_4$-like.
	As $\Lambda\setminus\{t\}$ is atomic BD-robust, Proposition~\ref{prop:ori link2} and Lemma~\ref{lem:weak atomic} imply that $\Delta_{\Lambda,\Lambda'}$ is $\wtC$-like. Thus $\Lambda'$ is a robust $\wtC$-core. This corresponds to Figure~\ref{fig:core1} (3).

	Suppose $b_3$ has valence $3$ in $\Lambda$ (so $b_3\neq b_{m-1}$). If $b_{m-1}$ has valence $\ge 3$ in $\Lambda$, let $e_0=b_{m-1}t_0$ be an edge with $t_0\notin \Lambda_0$, and let $\Lambda'=\Lambda_0\cup e_0$. Then $\{t_0,b_m,b_1,b_2\}$ are terminal vertices of $\Lambda'$. As $\Lambda_0$ is $\Lambda$-atomic, for $i=1,2$, $\Lambda'\setminus\{b_i\}\subset \Lambda\setminus\{b_i\}$ is weakly $\Lambda\setminus\{b_i\}$-atomic. Note that $\Lambda'\setminus \{t_0\}$ is $\Lambda\setminus\{t_0\}$-atomic $D_m$-like; and $\Lambda'\setminus\{b_m\}$ is either $\Lambda\setminus\{b_m\}$-atomic $D_m$-like (when $e_0$ is plain), or $\Lambda'\setminus\{b_m\}$ is a robust $\wtC$-core in $\Lambda\setminus\{b_m\}$ (when $e_0$ is marked). The latter follows from Proposition~\ref{prop:ori link0} and the assumption of Lemma~\ref{lem:robust C core}, and we still know that the $(b_1,b_2)$-subdivision of $\Delta^*_{\Lambda\setminus\{b_m\},\Lambda'\setminus\{b_m\}}$ is bowtie free and downward flag by Lemma~\ref{lem:big lattice}. Then $\Lambda'$ is a robust $\wtC$-core in $\Lambda$ by Lemma~\ref{lem:weak atomic} and Proposition~\ref{prop:ori link2}.  This corresponds to Figure~\ref{fig:core1} (14).
	
	Now we assume $b_{m-1}$ has valence $2$ in $\Lambda$. Let $\Phi$ be the set of all marked edges in $\Lambda$ together with all vertices of $\Lambda$ that are not $b_3$ and have valence $>2$. As $\Lambda$ has a $\wtC$-core, $\Phi$ is non-empty. We consider an element $\phi$ of $\Phi$ that has shortest distance to $b_3$. If $\phi$ an edge, then let $L$ be the smallest linear subgraph of $\Lambda$ containing $\phi$ and $b_3$. Then we define $\Lambda'=L\cup e_1\cup e_2\cup e_3$. Note that each interior vertex of $L$ has valence two in $\Lambda$ except the one $s_\phi$ that is contained in $\phi$. Thus for each leaf vertex $t$ of $\Lambda'$, $\Lambda'\setminus\{t\}$ is $\Lambda\setminus\{t\}$-atomic $D_k$-like or $B_k$-like (for some appropriate $k$). As $\Lambda\setminus\{t\}$ is atomic BD-robust, Proposition~\ref{prop:ori link0} implies that an appropriate subdivision of $\Delta_{\Lambda,\Lambda'}$ is $\wtC$-like. As $b_{m-1}$ has valence $2$ in $\Lambda$, each vertex in $\Lambda_0\setminus \Lambda'$ is $\Lambda'$-extremal. This corresponds to Figure~\ref{fig:core1} (4), (13) and (15).

	If $\phi$ is a vertex of $\Lambda$, let $L$ be the linear subgraph from $\phi$ to $b_3$, and let $e_4$ be the edge of $L$ containing $\phi$. We choose another two edges $e_5,e_6$ at $\phi$ such that $\{e_4,e_5,e_6\}$ contains all edges of $\Lambda_0$ with one vertex being $\phi$ (if there are any). We can assume $e_4,e_5,e_6$ are plain, otherwise we are reduced to the previous paragraph.
	Define $\Lambda'=L\cup(\cup_{i=1}^6e_i)$. Then for each leaf vertex $t'$ of $\Lambda'$, $\Lambda'\setminus\{t'\}$ is $\Lambda\setminus\{t'\}$-atomic $D_k$-like. Thus $\Lambda'$ is a robust $\wtC$-core in $\Lambda$ by Proposition~\ref{prop:ori link2}. By construction each vertex of $\Lambda_0\setminus\Lambda'$ is $\Lambda'$-extremal. This corresponds to Figure~\ref{fig:core1} (3), (14) and (16).

	Let $\Lambda_0=s_1s_2\ldots s_m$ be a $\Lambda$-atomic $B_m$-like subdiagram with $s_{m-1}s_m$ marked. 
	Suppose $s_2$ has valence $>2$ in $\Lambda$. Let $e$ be an edge at $s_2$ that is not in $\Lambda_0$. We define $\Lambda'=\Lambda_0\cup e$. Then $\Lambda'\setminus\{s_m\}$ is either $\Lambda\setminus\{s_m\}$-atomic or weakly $\Lambda\setminus\{s_m\}$-atomic $D_m$-like.
	Let $t$ be the vertex of $e$ that is not in $\Lambda_0$. Then $\Lambda'\setminus\{t\}$ is $\Lambda\setminus\{t\}$-atomic $B_m$-like. If $e$ is plain, then $\Lambda'\setminus\{s_1\}$ is $\Lambda\setminus\{s_1\}$-atomic $B_m$-like; if $e$ is marked, then $\Lambda'\setminus\{s_1\}$ is a robust $\wtC$-core in $\Lambda\setminus\{s_1\}$ (as $\Lambda'\setminus\{s_1,t\}$ is $\Lambda\setminus\{s_1,t\}$-atomic $B_{m-1}$-like for any leaf vertex $t$ of $\Lambda'\setminus\{s_1\}$), hence $\Delta^*_{\Lambda\setminus\{s_1\},\Lambda'\setminus\{s_1\}}$ is bowtie free and 2-sided flag by Lemma~\ref{lem:big lattice}. Thus Proposition~\ref{prop:ori link0} implies that $\Lambda'$ is a robust $\wtC$-core in $\Lambda$, as desired. This corresponds to Figure~\ref{fig:core1} (12).

	Suppose $s_{m-1}$ has valence $>2$ in $\Lambda$. Let $e$ be an edge at $s_{m-1}$ with $e\nsubseteq\Lambda_0$. Define $\Lambda'=s_{m-2}s_{m-1}\cup s_{m-1}s_m\cup e$. By a similar argument as before (and use Lemma~\ref{lem:d3 upward flag} if $e$ is marked), $\Lambda'$ is a robust $\wtC$-core in $\Lambda$. This corresponds to Figure~\ref{fig:core1} (5).

	Now we assume no interior vertex of $\Lambda_0$ has valence $>2$ in $\Lambda$. Let $\Phi$ be the set of all marked edges in $\Lambda$ that are not $s_{m-1}s_m$, together with all vertices of $\Lambda$ with valence $>2$. Then $\Phi\neq\emptyset$. Let $\phi$ be an element of $\Phi$ that is closest to the edge $s_{m-1}s_m$. If $\phi$ is a marked edge, then $\Lambda'$ is the smallest linear subdiagram containing $s_{m-1}s_m$ and $\phi$. Then $\Lambda'\setminus\{t'\}$ is $\Lambda\setminus\{t'\}$-atomic $B_k$-like  for each leaf vertex $t'$ of $\Lambda'$. Thus $\Lambda'$ is a robust $\wtC$-core in $\Lambda$. As $s_i$ has valence $=2$ in $\Lambda$ for $2\le i\le m-1$, each vertex in $\Lambda_0\setminus\Lambda'$ is $\Lambda'$-extremal. This corresponds to Figure~\ref{fig:core1} (2) and (11).
	If $\phi$ is a vertex with valence $>2$, then we define $\Lambda'$ to be the smallest linear subdiagram $L$ containing $\phi$ and $s_{m-1}s_m$, together with two edges $e_1,e_2$ at $\phi$ that are not  in $L$. Assume $e_1$ and $e_2$ are plain, otherwise we are reduced to the previous case. Then $\Lambda'$ is the desired robust $\wtC$-core in $\Lambda$ by Proposition~\ref{prop:ori link0} and the valence assumption on $\{s_i\}_{i=2}^{m-1}$. This corresponds to Figure~\ref{fig:core1} (1) and (12).

	It remains to consider $\Lambda_0$ is $\Lambda$-atomic $D_3$-like with vertices labeled as in Figure~\ref{fig:BD} left. As the argument is largely similar to before, here we only give a sketch.
	If $b_3$ has valence $>2$ in $\Lambda$, then we find a robust $\wtC$-core containing $\Lambda_0$, as in Figure~\ref{fig:core1} (13)-(16). If $b_3$ has valence $2$ in $\Lambda$, then let $\Phi$ be the collection of vertices of valence $\ge 3$ together with marked edges in $\Lambda$. For $i=1,2$, let $\Phi_i$ be the collection of elements in $\Phi$ that are contained in the same component of $\Lambda\setminus\{b_3\}$ as $b_i$. Then $\Phi=\Phi_1\cup\Phi_2$. If both $\Phi_1$ and $\Phi_2$ are nonempty, then $\Lambda_0$ is contained in a robust $\wtC$-core of $\Lambda$, as in Figure~\ref{fig:core1} (17)-(19). Now we assume $\Phi_1\neq\emptyset$ and $\Phi_2=\emptyset$. Let $\Phi'_1$ be the collection of elements of $\Phi_1$ that are closest to $b_3$. If $\Phi'_1$ contains both a vertex $v$ of valence $\ge 3$ and a marked edge, then $\Lambda'$ is taken to be a union of three edges containing $v$ as in Figure~\ref{fig:core1} (10) (note that even if $\Phi'_1$ have two marked edges, we still want to add one more edge to form $\Lambda'$ for the purpose of arranging that each vertex in $\Lambda_0\setminus\Lambda'$ is $\Lambda'$-extremal). If $\Phi'_1$ only has a marked edge, then $\Phi'_1\subsetneq \Phi_1$ and we take an element in $\Phi_1$ which is next closest to $b_3$, and these two elements of $\Phi_1$ give the desired $\Lambda'$, see Figure~\ref{fig:core1} (6) and (7). Now we assume $\Phi'_1$ is made of a single vertex. If this vertex has valence $\ge 4$, then $\Lambda'$ is taken to be a $\widetilde D_4$-like subdiagram containing this vertex. If this vertex has valence $3$, then we find the next closest element in $\Phi_1$ as before, see Figure~\ref{fig:core1} (8) and (9). 
\end{proof}

\begin{thm}
	\label{thm:BD robust prop}
Suppose	$\Delta_S$ is nondegenerate, regular, modeled on the a doubly marked tree diagram $\Lambda$, and satisfies C$_5$ and property (P). 
Suppose the link of each nonempty simplex in $\Delta_S$ is atomic BD-robust and $\Lambda$ contains a $\widetilde C$-core. 	Then $\Delta_\Lambda$ is atomic BD-robust. 
\end{thm}

\begin{proof}
	Let $\Lambda_0$ be a $\Lambda$-atomic $B_m$-like or $D_m$-like subdiagram. Let $\bar\Delta_{\Lambda,\Lambda_0}$ be as before. %We claim $\bar\Delta_{\Lambda,\Lambda_0}$ is bowtie free and satisfies the flagness requirements in Definition~\ref{def:good}, and $\Delta_{\Lambda,\Lambda_0}$ satisfies the labeled 4-cycle condition.
	Let $\Lambda'$ be the robust $\wtC$-core of $\Lambda$ given by Lemma~\ref{lem:robust C core}. 
	
	We first consider the case $\Lambda_0\subset \Lambda'$. If the pair $(\Lambda_0,\Lambda')$ is as in Figure~\ref{fig:core1} (11)--(14), then $\bar\Delta_{\Lambda,\Lambda_0}$ is bowtie free and satisfies the flagness requirements in Definition~\ref{def:good} by Lemma~\ref{lem:big lattice}, Proposition~\ref{prop:propagation} (1) and Proposition~\ref{prop:propagation2} (1) (2). By Lemma~\ref{lem:robust C core} (4), for each leaf vertex $s\in \Lambda'$, $\Lambda'\setminus\{s\}$ is either $\Lambda\setminus\{s\}$-atomic $B_k$-like (for some $k$), or (weakly) $\Lambda\setminus\{s\}$-atomic $D_k$-like, or a robust $\wtC$-core of $\Lambda\setminus\{s\}$. In the last case, $\Lambda'\setminus\{s,s'\}$ is $\Lambda\setminus\{s,s'\}$-atomic $B_{k'}$-like or $D_{k'}$-like for any leaf vertex $s'$ of $\Lambda'\setminus\{s\}$. By Lemma~\ref{lem:weak atomic}, Proposition~\ref{prop:propagation} (3) and Proposition~\ref{prop:propagation2} (3), $\Delta_{\Lambda,\Lambda'}$ satisfies the labeled 4-cycle condition. Hence $\Delta_{\Lambda,\Lambda_0}$ satisfies the labeled 4-cycle condition. 
	
	If we are in Figure~\ref{fig:core1} (15) and (16), then $\Lambda_0\subset \Lambda_1$ where $\Lambda_1$ is the union of all edges of $\Lambda'$ containing $b_3$. Let $\bar\Delta_{\Lambda,\Lambda_1}$ be the $(b_1,b_2)$-subdivision of $\Delta_{\Lambda,\Lambda_1}$, endowed with the partial order on its vertex set which is opposite to the one in Definition~\ref{def:subdivision}. By Corollary~\ref{cor:robust0}, $(\bar\Delta_{\Lambda,\Lambda_1})^0$ is bowtie free and upward flag, and $\Delta_{\Lambda,\Lambda_1}$ satisfies the labeled 4-cycle condition, modulo checking the assumptions of this corollary. Indeed,  Figure~\ref{fig:core1} (15) and (16) correspond to the degenerate situation of Figure~\ref{fig:core} (3) and (4) where $b_1$ is a leaf vertex of $\Lambda'$, so Assumption~\ref{assum:1} holds true automatically. Assumptions 1 and 3 of Corollary~\ref{cor:robust0} are clear from the assumption of Theorem~\ref{thm:BD robust prop}. By Lemma~\ref{lem:robust C core} (4), $\Lambda'\setminus\{s\}$ in Assumption 2 of Corollary~\ref{cor:robust0} is $\Lambda\setminus\{s\}$-atomic $B_k$-like or $D_k$-like (for some $k$), hence this assumption holds true.
	
	If we are in Figure~\ref{fig:core1} (17)--(19), then Proposition~\ref{prop:propagation} (1), Proposition~\ref{prop:propagation2} (1) (2),  Lemma~\ref{lem:4wheel} and Remark~\ref{rmk:inherit} imply that $\Delta_{\Lambda,\Lambda_0}$ is bowtie free, hence its $(b_1,b_2)$-subdivision is bowtie free by Lemma~\ref{lem:subdivision bowtie free}. By Lemma~\ref{lem:weakly flag equivalent}, it suffices to show $\Delta_{\Lambda,\Lambda_0}$ is weakly flag. Given an embedded 6-cycle $\omega=x_1y_1x_2y_2x_3y_3$ of type $(\hat b_1\hat b_3)^3$ in $\Delta_{\Lambda,\Lambda_0}$, we aim to find a vertex of type $\hat b_3$ or $\hat b_2$ that is adjacent to each of $\{x_1,x_2,x_3\}$ in $\Delta_{\Lambda,\Lambda_0}$.
	We view vertices of $\omega$ as elements in $(\bar\Delta_{\Lambda,\Lambda'})^0$. As $(\bar\Delta_{\Lambda,\Lambda'})^0$ is flag (Lemma~\ref{lem:big lattice}), there is a vertex $z$ of type $\hat t_z$ such that $z$ is adjacent to each of $\{x_1,x_2,x_3\}$ and $t_z$ is in the connected component of $\Lambda'\setminus\{b_1\}$ that contains $b_3$. As $(\bar\Delta_{\Lambda,\Lambda'})^0$ is bowtie free and $\{x_1,x_2,x_3\}$ are pairwise distinct, $z$ is adjacent or equal to each of $\{y_1,y_2,y_3\}$ in $\bar\Delta_{\Lambda,\Lambda'}$. If $z$ is of type $\hat b_3$ or $\hat b_2$, then we are done. Suppose this is not true. Then $\omega\subset \lk(z,\Delta_{\Lambda})$. Let $\Lambda_z$ be the connected component of $\Lambda\setminus\{t_z\}$ that contains $\Lambda_0$. By Lemma~\ref{lem:link}, $\omega\subset \Delta^z_{\Lambda_z,\Lambda_0}$. By our assumption of the theorem, $\Delta^z_{\Lambda_z,\Lambda_0}$ satisfies the labeled 4-cycle condition and its $(b_1,b_2)$-subdivision with the partial order in Definition~\ref{def:subdivision} is downward flag. Lemma~\ref{lem:weakly flag equivalent} and Lemma~\ref{lem:4wheel} imply $\Delta^z_{\Lambda_z,\Lambda_0}$ is weakly flag, which gives the desired vertex of type $\hat b_3$ or $\hat b_2$ adjacent to each of $\{x_1,x_2,x_3\}$. Other requirements of weakly flagness can be checked in a similar way.
	
	Now we consider the case $\Lambda_0\nsubseteq \Lambda'$. If we are in Figure~\ref{fig:core1} (1)--(5), then Corollary~\ref{cor:robust0} implies that $(\bar\Delta_{\Lambda,\Lambda_0})^0$ with the partial order in Assumption~\ref{assum:1} is bowtie free and upward flag, and $\Delta_{\Lambda,\Lambda_0}$ satisfies the labeled 4-cycle condition. Indeed, the requirement in the 5th paragraph of Assumption~\ref{assum:1} holds as each proper induced subdiagram of $\Lambda$ is atomic BD-robust. Let $\Gamma,\Gamma'$ be in the 6th paragraph of Assumption~\ref{assum:1}. Then Lemma~\ref{lem:robust C core} (1) and (3) imply that $\Gamma'$ is $\Gamma$-atomic or weakly $\Gamma$-atomic $D_k$-like. So the requirement in the 6th paragraph of Assumption~\ref{assum:1} follows from Lemma~\ref{lem:weak atomic}. Thus Assumption~\ref{assum:1} holds true. Corollary~\ref{cor:robust0} Assumption 1 is clear except the moreover part, and let $r\in \Lambda_0$ be as defined there. By Lemma~\ref{lem:robust C core} (1) (2) (3), either the connected component of $\Lambda\setminus\{r\}$ containing $\Lambda_0\setminus\{r\}$ is path made of plain edges, or $\Lambda_0\setminus\{r\}$ is contained in a (weakly) $\Lambda\setminus\{r\}$-atomic  $D_k$-like or $B_k$-like subdiagram (for some $k$), and the desired bowtie free condition in Corollary~\ref{cor:robust0} Assumption 1 holds true either by property (P) and Remark~\ref{rmk:bowtie free} or the assumption of Theorem~\ref{thm:BD robust prop} and Lemma~\ref{lem:weak atomic}. Corollary~\ref{cor:robust0} Assumption 2 follows from Lemma~\ref{lem:robust C core} (1) (2) (3) and our assumption.
	Corollary~\ref{cor:robust0} Assumption 3 is satisfied as $b_{m-2}b_{m-1}b_m$ is atomic $D_3$-like in $\Omega'$.

	It remains to consider Figure~\ref{fig:core1} (6)--(10). Let $\Theta$ be the smallest subtree containing $\Lambda_0$ and $\Lambda'$. Then we can find $\Lambda_1\subset\Theta$ containing $\Lambda_0$ such that $\Lambda_1$ is $B_k$-like or $D_k$-like (for some $k$) and the pair $(\Lambda_1,\Lambda')$ fits Figure~\ref{fig:core1} (1)--(5). By Corollary~\ref{cor:robust0} ($\Lambda_0$ replaced by $\Lambda_1$), $\bar\Delta_{\Lambda,\Lambda_1}$ is bowtie free and upward flag. Thus $\Delta_{\Lambda,\Lambda_0}$ is bowtie free, implying $\bar\Delta_{\Lambda,\Lambda_0}$ is bowtie free by Lemma~\ref{lem:subdivision bowtie free}. We can prove  $\Delta_{\Lambda,\Lambda_0}$ is weakly flag in a similar way as treating Figure~\ref{fig:core1} (17)-(19), using the one sided flagness of $\bar\Delta_{\Lambda,\Lambda_1}$. 
\end{proof}

A doubly marked tree diagram is \emph{$\wtC$-free} if it contains no $\wtC$-core. It is either a path with at most one marked edge, or tripod with plain edges. 

%Let $\Lambda$ be an arbitrary doubly marked tree diagram. A  simplex $\sigma$ of type $\hat T$ in $\Delta_\Lambda$ is \emph{free} if the induced subdiagram on $S\setminus T$ is $\wtC$-free. The simplex $\sigma$ is of \emph{type $A$}, if the induced subdiagram on $S\setminus T$ is a path with plain edges.

\begin{thm}
	\label{thm:atomic BD robust}
	Suppose	$\Delta_S$ is nondegenerate, regular, modeled on the a doubly marked tree diagram $\Lambda$, and satisfies C$_5$. 	Suppose that links of $\Delta_S$ modeled on $\wtC$-free subdiagrams are BD-robust, and links of $\Delta_S$ modeled on plain paths are bowtie free.
Then the following holds:
	\begin{enumerate}
		\item  the link of any simplex in $\Delta_S$ is atomic BD-robust;
		\item the link of any simplex of type $\hat T$ in $\Delta_S$ is contractible whenever the subdiagram induced on $S\setminus T$ contains a $\wtC$-core;
		\item if $\Lambda$ is not $\wtC$-free, then $\Lambda$ contains a robust $\wtC$-core $\Lambda'$ such that $\Delta_\Lambda$ deformation retracts onto  $\bar\Delta_{\Lambda,\Lambda'}$ which is a $\wtC$-like complex supporting an injective metric as in Remark~\ref{rmk:injective}.
	\end{enumerate}
In particular, $\Delta_S$ is atomic BD-robust. It is contractible if $\Lambda$ is not $\wtC$-free.
\end{thm}

\begin{proof}
	For Assertion 1, 
	let $P(n)$ be the statement that $\Delta^\sigma_{S\setminus T}$ is atomic BD-robust when the subdiagram induced on $S\setminus T$ is connected with $\le n$ vertices. $P(1)$ is clear. Suppose $P(n)$ is true. Let $\Lambda_1\subset \Lambda$ be a connected induced subdiagram with $n+1$ vertices, and $\sigma$ is chosen so its type $\hat T$ satisfies $T=S\setminus\Lambda_1$.	If $\Lambda_1$ does not contains a $\wtC$-core, then $\Lambda_1$ is $\wtC$-free, hence $\Delta^\sigma_{\Lambda_1}$ is BD-robust by our assumption. If $\Lambda_1$ contains a $\wtC$-core, then by Theorem~\ref{thm:BD robust prop} and $P(n)$, $\Delta^\sigma_{\Lambda_1}$ is atomic BD-robust. Thus $P(n+1)$ is true and Assertion 1 follows from induction.
	
	For Assertion 2, 	let $P(n)$ be the statement that $\Delta^\sigma_{S\setminus T}$ is contractible when the subdiagram induced on $S\setminus T$ contains a $\wtC$-core and $|S\setminus T|\le n$. $P(1)$ is vacuously true. Suppose $P(n)$ is true. Let $\Lambda_1\subset \Lambda$ be an induced subdiagram on $n+1$ vertices that contains a $\wtC$-core, and $\sigma$ is chosen so its type $\hat T$ satisfies $T=S\setminus\Lambda_1$.

	Suppose $\Lambda_1$ is not connected. Then $\Lambda_1$ has a connected component $\Lambda'_1$ that contains a $\wtC$-core. By \eqref{eq:join}, $\Delta^\sigma_{\Lambda'_1}$ is join factor of $\Delta^\sigma_{\Lambda_1}$. By inductive hypothesis, $\Delta^\sigma_{\Lambda'_1}$ is contractible, hence $\Delta^\sigma_{\Lambda_1}$ is contractible.

	Suppose $\Lambda_1$ is connected. By Assertion 1 and Lemma~\ref{lem:robust C core}, there is a $\wtC$-core $\Lambda'\subset \Lambda_1$ robust with respect to $\Delta^\sigma_{\Lambda_1}$. Thus $\bar\Delta^\sigma_{\Lambda_1,\Lambda'}$ is contractible and admits an injective metric by Theorem~\ref{thm:contractibleII} and Remark~\ref{rmk:injective}. Next we show $\Delta^\sigma_{\Lambda_1}$ deformation retracts onto $\Delta^\sigma_{\Lambda_1,\Lambda'}$. By Lemma~\ref{lem:dr}, it suffices to prove that for any simplex $\emptyset\neq \sigma'\subset\Delta^\sigma_{\Lambda_1}$ of type $\hat R$ with $R\subset \Lambda\setminus \Lambda_1$, $\lk(\sigma',\Delta^\sigma_{\Lambda_1})$ is contractible. However, $$\lk(\sigma',\Delta^\sigma_{\Lambda_1})=\lk(\sigma\ast\sigma',\Delta_\Lambda)=\Delta^{\sigma\ast\sigma'}_{\Lambda_1\setminus R}.$$
	As $\Lambda_1\setminus R\supset \Lambda_1$ contains a $\widetilde C$-core and $R\neq\emptyset$, by inductive hypothesis, $\lk(\sigma',\Delta^\sigma_{\Lambda_1})$ is contractible, as desired.
	
	Assertion 3 follows from the proof of Assertion 2.
\end{proof}

A \emph{special 4-cycle} is an embedded cycle in $\Delta_\Lambda$ of type $\hat s\hat t\hat s\hat t$; it is \emph{supported} on a subdiagram $\Theta$ if $s,t\in \Theta$. Its \emph{center} is a vertex adjacent to all vertices.

Let $\Lambda'$ be a $B_n$-like or $D_n$-like subdiagram of $\Lambda$.
A \emph{base} of $\Lambda'$ is a leaf of $\Lambda'$ farthest from the
marked edge in the $B_n$ case, or from
the branch vertex in the $D_n$ case with $n\ge4$; for $D_3$,
it is the middle vertex. A \emph{special 6-cycle} in $\Delta_\Lambda$ is an embedded cycle of type $\hat s\hat t_1\hat s\hat t_2\hat s\hat t_3$ supported on a $B_n$-like or $D_n$-like subdiagram $\Lambda'$, where $s$ is a base of $\Lambda'$ and $t_i$ are leaf vertices of $\Lambda'$. A
\emph{quasi-center} for this 6-cycle is a vertex adjacent to the three type $\hat s$ vertices.

\begin{thm}
	\label{thm:geometry}
	Let $\mathcal E$ be the collection of $\wtC$-free doubly marked diagrams with $E_h=\emptyset$. 
	Suppose $\Lambda$ is connected doubly marked tree, has at least three vertices and does not belong to $\mathcal E$.
	Suppose	the complex $\Delta_S$ satisfies
	\begin{enumerate}
		\item $\Delta_S$ is nondegenerate, regular, modeled on $\Lambda$, and satisfies C$_5$;
		\item in each link of $\Delta_S$ modeled on a diagram from $\mathcal E$, any special 4-cycles supported on a $B_n$-like, or $D_n$-like or a plain path subdiagram has a center;
		and any special 6-cycle  has a quasi-center. 
	\end{enumerate}
Then $\Delta_S$ is contractible, and it deformation retracts onto a subcomplex which admits either an injective metric or a CAT$(0)$ metric. More precisely:
\begin{enumerate}
	\item if $\Lambda$ is not $\wtC$-free, then there exists a robust $\wtC$-core $\Lambda'$ such that $\Delta_S$ deformation retracts onto  $\bar\Delta_{\Lambda,\Lambda'}\subset \Delta_S$ which can be metrized as an injective orthoscheme complex;
	\item if $\Lambda$ is $\wtC$-free, then there exists a path subdiagram $\Lambda'$ with three vertices such that $\Delta_S$ deformation retracts onto $\Delta_{\Lambda,\Lambda'}$ which can be metrized as a CAT$(0)$ piecewise Euclidean complex.
\end{enumerate}
\end{thm}

\begin{proof}

Let $\Lambda'$ be an induced $\wtC$-free subdiagram of $\Lambda$. We claim any links of $\Delta_\Lambda$ modeled on $\Lambda'$ is BD-robust. Take simplex $\sigma$ of type $\hat T$ with $S\setminus\Lambda'=T$. 
	
	First we assume $\Lambda'\in \mathcal E$. Let $\omega=x_1x_2x_3x_4$ be a special 4-cycle in $\Delta^\sigma_{\Lambda'}$ of type $\hat s\hat t\hat s\hat t$. By Assumption 2, there is a vertex $x\in \Delta^\sigma_{\Lambda'}$ of type $\hat r$ which is a center for $\omega$. We now show $x$ can be chosen so $t$ is contained in the linear subdiagram from $s$ to $t$. Indeed, if this is not the case, then $\omega\subset \lk(x,\Delta^\sigma_{\Lambda'})$. Let $\Lambda''$ be the connected component of $\Lambda'\setminus\{r\}$ that contains $s$ and $t$. Then $\Lambda''\in\mathcal E$, and by Lemma~\ref{lem:link}, we can view $\omega$ as a special 4-cycle in $\Delta^{\sigma'}_{\Lambda''}$ for a bigger simplex $\sigma'$. By repeating this argument finitely many times, we obtain a center with correct type. Likewise, for a special 6-cycle $x_1x_2x_3x_4x_5x_6$ of type $\hat s\hat t_1\hat s\hat t_2\hat s\hat t_3$ in $\Delta^\sigma_{\Lambda'}$ supported on a $B_n$-like or $D_n$-like subdiagram $\Lambda_0$, its quasi-center can be chosen to have type $\hat r$ with $r\in \Lambda_0$. Indeed, this follows from a similar ``passing to links'' argument as before (note that C$_4$ implies the quasi-center is a center up to replacing $x_2,x_4,x_6$ by vertices of the same type).
	By Corollary~\ref{cor:bowtie free criterion1}, $\Delta^\sigma_{\Lambda',\Lambda'_0}$ is bowtie free for any linear subdiagram $\Lambda'_0$ of $\Lambda'$. Hence $\Delta^\sigma_{\Lambda'}$ satisfies the labeled 4-cycle condition by Lemma~\ref{lem:4wheel}. Hence $\Delta^\sigma_{\Lambda'}$ is BD-robust by the assumption on special 6-cycles, Corollary~\ref{cor:upward flag criterion}, and Lemma~\ref{lem:subdivision bowtie free}.
	
	It follows in particular that $\Delta_S$ satisfies property (P).
	
	It remains to consider $\Lambda'$ is $\wtC$-free but not belong to any families in Definition~\ref{def:special}. Then $\Lambda'$ is a linear diagram with exactly one edge in $E_h$, and all other edges are plain. By \cite[Thm 10.12]{huangbestvina} and property (P), $\Delta^\sigma_{\Lambda'}$ is bowtie free.
	Given a $B_n$-like subdiagram $\Lambda''=s_1s_2\cdots s_n\subset \Lambda'$ such that $s_{n-1}s_n\in E_h$, we consider the complex $\Delta=\Delta^{\sigma}_{\Lambda',s_{n-2}s_{n-1}s_n}$, and metrize it as a piecewise Euclidean complex such that its fundamental domain (which is a single 2-simplex) is a flat triangle with angle $\pi/2$ at vertex of type $\hat s_{n-1}$, angle $\pi/6$ at vertex of type $\hat s_{n-2}$, and angle $\pi/3$ at vertex of type $\hat s_n$. It follows from Lemma~\ref{lem:scgeneral}, property (P) and the girth statement in \cite[Thm 10.12]{huangbestvina} that $\Delta$ is CAT$(0)$. Given an embedded 6-cycle $x_1x_2x_3x_4x_5x_6\subset \Delta^\sigma_{\Lambda',\Lambda''}$ of type $\hat s_1\hat s_n\hat s_1\hat s_n\hat s_1\hat s_n$, for $i=1,3,5$, let $X_i$ be the full subcomplex of $\Delta$ spanned by vertices that are adjacent to $x_i$ in $\Delta_\Lambda$. By \cite[Lem 9.4, Prop 9.11 (1)]{huang2023labeled} and property (P), $X_i$ is a convex subcomplex of $\Delta$ in the sense of CAT$(0)$ geometry (the results in \cite{huang2023labeled} were for Artin complexes, however, the proofs there generalize directly to simplicial complexes of type $S$). By construction, $\{X_1,X_3,X_5\}$ pairwise intersect. By an argument similar to \cite[pp. 47-48]{huangbestvina}, $X_1\cap X_2\cap X_3\neq\emptyset$, which gives a vertex in $\Delta$ that is adjacent to each of $\{x_1,x_3,x_5\}$ in $\Delta^\sigma_{\Lambda'}$. This gives the desired flagness property on $\Delta^\sigma_{\Lambda'}$ by Corollary~\ref{cor:upward flag criterion}. This proves the claim that any links of $\Delta_S$ modeled on a $\wtC$-free diagram is BD-robust.
	
	The first conclusion of this theorem follows from the claim and Theorem~\ref{thm:BD robust prop}. The CAT$(0)$-part of the second conclusion follows from the previous paragraph. The deformation retraction part can be proved inductively using Lemma~\ref{lem:dr} as in the proof of  Theorem~\ref{thm:BD robust prop}.
\end{proof}

\begin{remark}
\label{rmk:dimension}
If $\Lambda$ is not $\wtC$-free, then by our proof, the subdiagram $\Lambda'$ in the conclusion of case (1) of Theorem~\ref{thm:geometry} can be chosen to be any $\wtC$-like subdiagram that has the minimal number of vertices among all $\wtC$-like subdiagrams. In this case of Theorem~\ref{thm:geometry}, the dimension of the subcomplex is $|V\Gamma'|-1$. In case (2)  where $\Lambda$ is $\wtC$-free, the dimension of the subcomplex is 2.
\end{remark}

\subsection{Group theoretical applications}
Let $(G,\{H_s\}_{s\in S})$ be as in Definition~\ref{def:admissible parabolic system}. For any induced subdiagram $\Lambda'\subset \Lambda$, let $G_{\Lambda'}$ be the subgroup generated by $H_s$ with $s\in \Lambda'$ and let $\Delta_{\Lambda'}$ be the coset complex for $G_{\Lambda'}$.

The following is a consequence of Theorem~\ref{thm:geometry} and Lemma~\ref{lem:group}. 
\begin{thm}
	\label{thm:group}
Suppose $\Lambda$ is connected doubly marked tree with $|V\Lambda|\ge 3$ and $\Lambda\notin\mathcal E$.	
Let $(G,\{H_s\}_{s\in S})$ be a nondegenerate, strongly admissible parabolic system over $\Lambda$, with its coset complex $\Delta_S$. Suppose for any induced subdiagram $\Lambda'\subset\Lambda$ with $\Lambda'\in \mathcal E$, any special 4-cycles supported on a $B_n$-like, or $D_n$-like or a plain path subdiagram has a center;
and any special 6-cycle in $\Delta_{\Lambda'}$ has a quasi-center. 

Then all the conclusions of Theorem~\ref{thm:geometry} hold true. Moreover, the subcomplexes of $\Delta_S$ in the two conclusions of Theorem~\ref{thm:geometry} are $G$-invariant, and $G$ acts cocompactly by isometries on the relevant metric spaces.
\end{thm}

Using the map from extended Coxeter diagrams to doubly marked diagrams in Section~\ref{subsec:rel Artin complex}, we define the class $\mathcal E$ and $\wtC$-free among the extended Coxeter diagrams.
The following combines Theorem~\ref{thm:group}, Lemma~\ref{lem:shephard to general}, Theorem~\ref{lem:even shephard}, and Lemma~\ref{lem:nonadj}.

\begin{thm}
	\label{thm:s}
Suppose $\Lambda$ is connected doubly marked tree with $|V\Lambda|\ge 3$ and $\Lambda\notin\mathcal E$.	
Suppose $\Sh_\Lambda$ is an even Shephard group (or any strongly admissible Shephard group), with its coset complex $\Delta_\Lambda$. Suppose for any induced subdiagram $\Lambda'\subset\Lambda$ with $\Lambda'\in \mathcal E$, any special 4-cycles supported on a $B_n$-like, or $D_n$-like or a plain path subdiagram has a center;
and any special 6-cycle in $\Delta_{\Lambda'}$ has a quasi-center. 

Then all the conclusions of Theorem~\ref{thm:geometry} hold true. Moreover, the subcomplexes of $\Delta_S$ in the two conclusions of Theorem~\ref{thm:geometry} are $\Sh_\Lambda$-invariant, and $\Sh_\Lambda$ acts cocompactly by isometries on the relevant metric spaces. 
\end{thm}

\begin{thm}
	\label{thm:tripod reduction}
	Suppose $\Lambda$ is a forest Coxeter diagram such that $\Delta_{\Pi}$ is BD-robust for any induced   $\wtC$-free subdiagram $\Pi$ of $\Lambda$ and $A_\Pi$ satisfies the $K(\pi,1)$-conjecture. Then the Artin group $A_\Lambda$ satisfies the $K(\pi,1)$-conjecture. 
\end{thm}

\begin{proof}
	Assume $\Lambda$ is not spherical, other the $K(\pi,1)$-conjecture for $A_\Lambda$ already follows from \cite{deligne}.	
	We claim for any induced non-spherical subdiagram $\Lambda'\subset \Lambda$, $\Delta_{\Lambda'}$ is contractible. We induct on the number of vertices in $\Lambda'$ and suppose the claim holds whenever $\Lambda'$ has $\le n$ vertices. Take an induced non-spherical subdiagram $\Lambda'$ of $\Lambda$ with $n+1$ vertices. We assume $\Lambda'$ is connected, otherwise $\Delta_{\Lambda'}$ is a join $\Delta_{\Lambda'_1}\circ\Delta_{\Lambda'_2}$ with $\Lambda'_1\subsetneq \Lambda'$ non-spherical, and the induction hypothesis implies the contractibility of $\Delta_{\Lambda'_1}$, which gives the contractibility of $\Delta_{\Lambda'}$. If $\Lambda'$ is $\wtC$-free, then $A_{\Lambda'}$ satisfies the $K(\pi,1)$-conjecture. Consequently, $A_{\Lambda''}$ satisfies the $K(\pi,1)$-conjecture for any induced subdiagram $\Lambda''\subset \Lambda'$ by \cite[Cor 2.4]{paris2012k}. As $\Lambda'$ is non-spherical, \cite[Theorem~3.1]{paris2012k} implies that $\Delta_{\Lambda'}$ is contractible. Suppose $\Lambda'$ is not $\wtC$-free. Then by Lemma~\ref{lem:shephard to general}, Theorem~\ref{lem:even shephard}, Lemma~\ref{lem:nonadj}, Theorem~\ref{thm:atomic BD robust} and Theorem~\ref{thm:bowtie free}, $\Delta_{\Lambda'}$ is contractible. This proves the claim. The theorem follows from this claim and Theorem~\ref{thm:kpi1} by induction on the number of vertices in $\Lambda'$.
\end{proof}

Proof of the following is similar to Theorem~\ref{thm:tripod reduction}, using Theorems~\ref{thm:s} and \ref{thm:kpi1}.

\begin{thm}
	\label{thmb:reduction single}
	Let $\Lambda$ be a tree Coxeter diagrams such that the following holds for any induced subdiagram $\Lambda'\subset \Lambda$ in families $F_{r,s}$, $H_{r,s}$, and $E_{r,s,t}$ in Definition~\ref{def:special}:
	\begin{enumerate}
		\item the Artin group $A_{\Lambda'}$ satisfies the $K(\pi,1)$-conjecture;
		\item any special 4-cycles in the Artin complex $\Delta_{\Lambda'}$ supported in a $B_k$-like or $D_k$-like subdiagram of $\Lambda'$ has a center, and any special 6-cycle in $\Delta_{\Lambda'}$ has a quasi-center.
	\end{enumerate}
	The $A_\Lambda$ satisfies the $K(\pi,1)$-conjecture. Moreover, if $\Lambda$ does not belong to the families in Definition~\ref{def:special}, then $A_\Lambda$ acts cocompactly on an injective metric space or a CAT$(0)$ space such that all the point stabilizers are isomorphic to a proper standard parabolic subgroups of $A_\Lambda$. 
\end{thm}

\section{Artin groups of type $ABI$}
\label{sec:ABI}

A Coxeter diagram is of type $ABI$, if every induced irreducible spherical subdiagrams is of type $A_n,B_n$ or $I_2(n)$. In this section we prove:

\begin{thm}
	\label{thm:ABI K(pi,1)}
	Let $\Lambda$ be a Coxeter diagram of type $ABI$. Then the $K(\pi,1)$-conjecture holds true for the Artin group $A_\Lambda$.
	
	Moreover, these Artin groups acts cocompactly by isometries on CAT$(0)$ spaces or injective metric spaces or Bestvina complexes with stabilizers being Artin groups with strictly smaller number of generators. 
\end{thm}

This $K(\pi,1)$-part of the theorem can be deduced from \cite[Cor 9.10]{huangbestvina}, Proposition~\ref{prop:ABI BD-robust} and Proposition~\ref{prop:ABI bowtie free}. The moreover part follows from Theorem~\ref{thm:reduction single}, \cite[Thm 10.9]{huangbestvina} and \cite{huang2021helly}.

\begin{prop}
	\label{prop:ABI BD-robust}
	Let $\Lambda$ be a tree Coxeter diagram of type $ABI$ with edge labels $\le 5$. Then $\Lambda$ is atomic BD-robust and $A_\Lambda$ satisfies the $K(\pi,1)$-conjecture.
\end{prop}

\begin{proof}
	By Theorem~\ref{thm:atomic BD robust} and Theorem~\ref{thm:tripod reduction}, it suffices to show any $\wtC$-free induced subdiagram $\Lambda'$ of $\Lambda$ is BD-robust and satisfies the $K(\pi,1)$-conjecture. 
	As $\Lambda$ is of type $ABI$ with edge labels $\le 5$, $\Lambda'$ is of type $A_n$ or $B_n$ or $I_2(5)$, hence $A_{\Lambda'}$ satisfies the $K(\pi,1)$-conjecture. In the $A_n$-case, we deduce BD-robustness from \cite[Thm 5.6]{huang2024}, Lemma~\ref{lem:weakly flag equivalent} and Theorem~\ref{thm:bowtie free}. In the $B_n$-case, we deduce BD-robustness from \cite[Lem 12.2]{huang2024}, Lemma~\ref{lem:weakly flag equivalent}, Lemma~\ref{lem:d3 upward flag}, Theorem~\ref{thm:triple} and Theorem~\ref{thm:bowtie free}. In the $I_2(5)$ case, its Artin complex is a graph of girth $\ge 10$ \cite{AppelSchupp1983}, hence it is BD-robust.
\end{proof}

\begin{prop}
	\label{prop:ABI bowtie free}
	Let $\Lambda$ be a tree Coxeter diagram of type $ABI$ with edge labels $\le 5$. Then $\Delta_{\Lambda}$ satisfies the labeled 4-cycle condition and $\Lambda$ is $B$-robust.
\end{prop}

We will be working in the more general context of simplicial complex of type $S$. Due to the procedure from a Coxeter diagram to a doubly marked diagram in Section~\ref{subsec:rel Artin complex}, we will also write $\Delta_S$ is modeled on a Coxeter diagram, meaning $\Delta_S$ is modeled on the doubly marked diagram associated with this Coxeter diagram. Then Proposition~\ref{prop:ABI bowtie free} is a consequence Proposition~\ref{prop:ABI bowtie free general}, Lemma~\ref{lem:shephard to general}, Theorem~\ref{thm:even Artin}, Lemma~\ref{lem:nonadj}, \cite[Thm 5.6]{huang2024}, Theorem~\ref{thm:bowtie free} and Theorem~\ref{thm:triple}.

\begin{prop}
	\label{prop:ABI bowtie free general}
	Let $\Lambda$ be a tree Coxeter diagram of type $ABI$ with edge labels $\le 5$. Suppose $\Delta_\Lambda$ is a simplicial complex of type $S$ modeled on $\Lambda$ such that 
	\begin{enumerate}
		\item  $\Delta_\Lambda$ is nondegenerate, regular and satisfies C$_5$;
		\item links of $\Delta_\Lambda$ modeled on $A_n$, $B_n$ or $I_2(5)$ subdiagrams are BD-robust, and links of $\Delta_\Lambda$ modeled on $A_n$ subdiagrams are bowtie free.
	\end{enumerate}
	Then $\Delta_\Lambda$ is $B$-robust and satisfies the labeled 4-cycle condition.
\end{prop}

In the rest of this section, we prove Proposition~\ref{prop:ABI bowtie free general}. Also we assume throughout the whole section that $\Delta_S$ (or $\Delta_\Lambda$) is nondegenerate, regular, modeled on the a doubly marked diagram, and satisfies  C$_5$.

\subsection{Tree diagrams with $\widetilde C_2$-like subdiagram}

\begin{prop}
	\label{prop:labeled 4-cycle with C_2 subdiagram}
	Suppose $\Delta_\Lambda$ is modeled on a doubly marked tree diagram. Suppose the link of each nonempty simplex of $\Delta_\Lambda$ is $B$-robust and satisfies the labeled 4-cycle condition. We assume in addition that $\Lambda$ contains a $\wtC_2$-like subdiagram $\Lambda'=t_1t_2t_3$.
	Then $\Delta_{\Lambda}$ satisfies the labeled 4-cycle condition.
\end{prop}

\begin{proof}
	Let $\Omega=r_1\cdots r_m$ be a maximal linear subdiagram of $\Lambda$.  By Lemma~\ref{lem:4wheel}, it suffices to show $\Delta_{\Lambda,\Omega}$ is bowtie free. By Lemma~\ref{lem:locally determined}, Lemma~\ref{lem:scgeneral} and the $B$-robust assumption in the proposition, we know $\Lambda'$ is a robust $\wtC$-core of $\Delta_\Lambda$ and $\Delta_{\Lambda,\Lambda'}$ is locally determined. Let $\tau$ be the rank function on $(\Delta_{\Lambda,\Lambda'})^0$ such that vertices of type $\hat t_3$ are maximal.
	Consider a graph morphism $\pi:\Lambda\to\Lambda'$ sending each vertex of $\Lambda$ to the closet vertex in $\Lambda'$. Then $\pi(\Omega)$ is either $\Lambda'$ (Figure~\ref{fig:bowtie} (1)), or a leaf vertex of $\Lambda'$ (up to symmetry we assume it is $t_1$, see Figure~\ref{fig:bowtie} (2)), or $t_2$ (Figure~\ref{fig:bowtie} (3)), or an edge of $\Lambda'$ (up to symmetry we assume it is the edge $t_2t_3$, see Figure~\ref{fig:bowtie} (4)).
	
	\begin{figure}[h]
		\centering
		\includegraphics[scale=1]{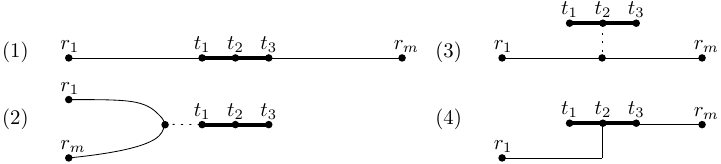}
		\caption{The dotted segments in (2) and (3) are allowed to be trivial. In (1), it is possible that $r_1=t_1$ or $t_3=r_m$. In (4), it is possible that $t_3=r_m$.}
		\label{fig:bowtie}
	\end{figure}
	
	For a vertex $z\in \Delta_\Lambda$ of type $\hat r$ with $r\notin \Lambda'$, we define $Z_z$ to be the full subcomplex of $\Delta_{\Lambda,\Lambda'}$ spanned by vertices that are adjacent to $z$. Note that $Z_z$ satisfies the conclusion of Corollary~\ref{cor:convex}. Indeed, this follows from Corollary~\ref{cor:convex} if $\pi(r)=t_1$ or $t_3$. The case of $\pi(r)=t_2$ can be proved in a similar way and we omit the details.
	
	Given vertices $s,t\in \Lambda$, let $\Lambda_{st}$ be the line segment from $s$ to $t$. The vertices $s,t\in\Lambda\setminus\Lambda'$ are \emph{weakly separated} by $\Lambda'$, if either a vertex of $\Lambda'$ separates $s$ from $t$, or $\Lambda'\cap\Lambda_{st}=\emptyset$ and the vertex of $\Lambda_{st}$ closest to $\Lambda'$ separates $s$ from $t$.
	\begin{claim}
		\label{claim:reduce}
		Suppose that for any embedded 4-cycle $z_1z_2z_3z_4$ in $\Delta_\Lambda$ of type $\hat s\hat t\hat s\hat t$ such that $s,t$ are weakly separated by $\Lambda'$, we have either $Z_{z_1}\cap Z_{z_3}\neq\emptyset$ or $Z_{z_2}\cap Z_{z_4}\neq\emptyset$. Then the proposition holds true.
	\end{claim}
	
	\begin{proof}
		If $\Lambda_0=s_1s_2\ldots s_k$ is a line segment from $s_1\in \Lambda\setminus\Lambda'$ to $s_k\in \Lambda'$, then $\Delta_{\Lambda,\Lambda_0}$ is bowtie free. Indeed, we enlarge $\Lambda_0$ if necessary, so that $s_{k-1}s_k$ is an edge of $\Lambda'$. Then this follows from Corollary~\ref{cor:robust0}, except the case $\Lambda'\subset \Lambda_0$, which can be treated using the same argument in the proof of Corollary~\ref{cor:robust0}.
		
		%By induction on $k$ and Lemma~\ref{lem:bowtie free criterion}, it suffices to show any embedded 4-cycle $y_1y_2y_3y_4$ of type $\hat s_1\hat s_k\hat s_1\hat s_k$ has a center $y$ of type $\hat t$ for some $t\in\Lambda_0$. Note that $Z_{y_1}$ and $Z_{y_3}$ are locally convex in $\Delta_{\Lambda,\Lambda'}$ and $\{y_2,y_4\}\subset Z_{y_1}\cap Z_{y_3}$. Similar to the proof of the bowtie free part of Corollary~\ref{cor:robust0}, we can find an edge path in $Z_{y_1}\cap Z_{y_3}$ from $y_2$ to $y_4$, which can be used to produce the vertex $y$. 
		
		Let $\Omega$ be as before and take $s,t\in \Omega$. By Corollary~\ref{cor:bowtie free criterion1}, we need to show any embedded 4-cycle $z_1z_2z_3z_4\subset \Delta_\Lambda$ of type $\hat s\hat t\hat s\hat t$ has a center $z$ of type $\hat r$ with $r\in \Lambda_{st}$. If $s,t$ are not weakly separated by $\Lambda'$, then this follows from the previous paragraph. Suppose $s,t$ are weakly separated by $\Lambda'$. If $Z_{z_1}\cap Z_{z_3}\neq\emptyset$, then there is a vertex $z'\in \Delta_\Lambda$ of type $\hat r'$ such that $r'\in \Lambda'$ and $z'$ is adjacent $\{z_1,z_3\}$. Let $\Lambda_0$ be the line segment from $s$ to $r'$. We put the partial order on $(\Delta_{\Lambda,\Lambda_0})^0$ such that vertices of type $\hat s$ are minimal. Then $z'$ is a common upper bound for $\{z_1,z_3\}$ in $(\Delta_{\Lambda,\Lambda_0})^0$. By previous paragraph, $\{z_1,z_3\}$ have the join, denoted by $z$, of type $\hat r$. Now we consider the 4-cycle $z_1zz_3z_2$. By condition C$_4$, there is a vertex $\bar z$ of the same type as $z$ such that $\bar z$ is adjacent to $\{z_1,z_3,z_2\}$. Then $\bar z$ is a common upper bound for $\{z_1,z_3\}$ in $(\Delta_{\Lambda,\Lambda_0})^0$, which implies that $z=\bar z$. Hence $z$ is adjacent to $z_2$. Similarly, $z$ is adjacent to $z_4$. If $r\in \Lambda_{st}$, then $z$ is the desired center. If $r\notin \Lambda_{st}$, then Lemma~\ref{lem:link} implies that the 4-cycle $z_1z_2z_3z_4$ is contained in $\Delta^z_{\Lambda\setminus\{r\},\Lambda_{st}}$, and we can conclude using the bowtie free property of  $\Delta^z_{\Lambda\setminus\{r\},\Lambda_{st}}$. The case of  $Z_{z_2}\cap Z_{z_4}\neq\emptyset$ can be handled in a similar way.
	\end{proof}
	
	In the rest of the proof, we show the assumption of Claim~\ref{claim:reduce} holds true. We can assume without loss of generality that $s=r_m$ and $t=r_1$, and $r_1,r_m\notin \Lambda'$. Let $Z_i=Z_{z_i}$. Let $\Xi$ be the collection of all quadruples $(u_1,u_2,u_3,u_4)$ such that each $u_i$ is an extremal vertex in $Z_i\cap Z_{i+1}$ for $i\in \mathbb Z/4\mathbb Z$. For each element in $\Xi$, we consider the quantity
	\begin{equation}
		\label{eq:ds2}
		d(u_1,u_2)+d(u_2,u_3)+d(u_3,u_4)+d(u_4,u_1).
	\end{equation}
	Let $\Xi_{\min} \subset \Xi$ be the collection of elements in $\Xi$ such that \eqref{eq:ds2} is minimized.

	\begin{claim}
		\label{claim:identical}
		Suppose we are in Figure~\ref{fig:bowtie} (2)-(4). Then there exists $(u_1,u_2,u_3,u_4)\in \Xi_{\min}$ such that $u_i=u_{j}$ for some $1\le i\neq j\le 4$. In particular, $Z_1\cap Z_3\neq\emptyset$ or $Z_2\cap Z_4\neq\emptyset$.
	\end{claim}

	We need a preparatory claim before proving Claim~\ref{claim:identical}.
	Given $(u_1,u_2,u_3,u_4)\in \Xi$.
	Let $\omega=y_1y_2\cdots y_k$ be a local normal form path from $u_2$ to $u_3$. Then $\omega\subset Z_3$ by Corollary~\ref{cor:convex}. For each $i$, let $y^{a_i}_{i}y^{a_i-1}_{i}\cdots y^1_{i}$ be the local normal form path from $y_i=y^{a_i}_{i}$ to $y^1_{i}=u_1$, which exists by Proposition~\ref{prop:existence and uniqueness}. These local normal form paths are also normal form paths by Proposition~\ref{prop:local vs global}.
	
	\begin{claim}
		\label{claim:decrease}
		Suppose we are in Figure~\ref{fig:bowtie} (2)-(4).	Suppose $\{u_1,u_2,u_3\}$ are pairwise distinct.
		If $d(y_2,u_1)\le d(u_2,u_1)$, then we can replace $u_2$ by another extremal vertex $u'_2\in Z_2\cap Z_3$ such that $$d(u'_2,u_3)+d(u'_2,u_1)<d(u_2,u_3)+d(u_2,u_1).$$ 
	\end{claim}
	
	\begin{proof}
		We only treat the case $d(y_2,u_1)=d(u_2,u_1)$ as the case $d(y_2,u_1)<d(u_2,u_1)$ is similar and simpler. If $\tau(u_2)=3$, then $y_2<y_1=u_2$ and Lemma~\ref{lem:strip} implies that $y^i_2\le y^i_1$ for $1\le i\le a_1$, see Figure~\ref{fig:coner3} (2). Similarly, if $\tau(u_2)=1$, then $y^i_2\ge y^i_1$ for $1\le i\le a_1$. See Figure~\ref{fig:coner3} (1). In the latter case by the second paragraph of the proof of Lemma~\ref{lem:corner1}, $y^{a_2}_2<y^{a_2-1}_2$ (the argument relies on the existence of $y^{a_1-2}_1$, however, if $y^{a_1-2}_1$ does not exist, then $y^{a_1-1}_1$ is maximal and $y^{a_1-1}_1=y^{a_2-1}_2$, hence  $y^{a_2}_2<y^{a_2-1}_2$ is clear).

		\begin{figure}[h]
			\centering
			\includegraphics[scale=1]{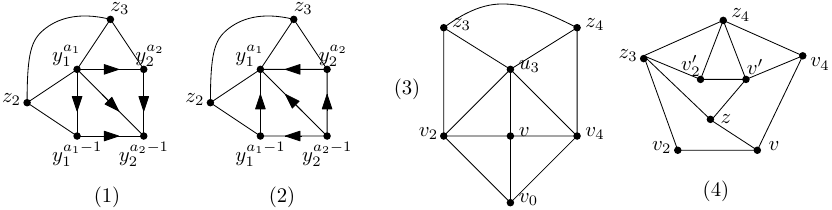}
			\caption{The arrows indicate poset relations.}
			\label{fig:coner3}
		\end{figure}

		If we are in Figure~\ref{fig:bowtie} (2) and $\tau(u_2)=1$, then by Lemma~\ref{lem:transitive} $z_2\sim y^{a_2-1}_2$ and $z_3\sim y^{a_2-1}_2$, as $y^{a_2-1}_2$ is of type $\hat t_3$, and $t_1$ separates $t_3$ from $\{r_1,r_m\}$. So $y^{a_2-1}_2$ is an extremal vertex in $Z_2\cap Z_3$. The claim follows by taking $u'_2=y^{a_2-1}_2$. If $\tau(u_2)=3$, then we deduce from $z_3\sim y^{a_2}_2$ and $y^{a_2}_2\sim y^{a_2-1}_2$ that $z_3\sim y^{a_2-1}_2$. Similarly, $z_2\sim y^{a_2-1}_2$. Thus we can take $u'_2=y^{a_2-1}_2$ as before. The cases of Figure~\ref{fig:bowtie} (3) can be treated again by taking $u'_2=y^{a_2-1}_2$.
		
		If we are in Figure~\ref{fig:bowtie} (4) and $\tau(u_2)=3$, then $z_3\sim y^{a_1}_1$ and $y^{a_1}_1\sim y^{a_2-1}_2$ imply $z_3\sim y^{a_2-1}_2$, and $z_2 \sim y^{a_1-1}_1$ and $y^{a_1-1}_1\sim y^{a_2-1}_2$ imply $z_2\sim y^{a_2-1}_2$. Thus we take $u'_2=y^{a_2-1}_2$. Suppose $\tau(u_2)=1$. Then $z_2\sim y^{a_2-1}_2$. Consider the 4-cycle $y^{a_2-1}_2y^{a_2}_2z_3z_2$ in $\lk(y^{a_1}_1,\Delta_\Lambda)$. By Lemma~\ref{lem:link}, we view $y^{a_2-1}_2y^{a_2}_2z_3z_2$ as a 4-cycle in $\Delta^{u_2}_{\Lambda\setminus\{t_1\},\Omega}$, which is bowtie free. By consider the join of $y^{a_2}_2$ (which is of type $\hat t_2$) and $z_2$ (which is of type $\hat r_1$) in $(\Delta^{u_2}_{\Lambda\setminus\{t_1\},\Omega})^0$ (assume vertices of $\hat r_m$ are maximal), we deduce that either $z_3\sim y^{a_2-1}_2$, in which case take $u'_2=y^{a_2-1}_2$, or $z_2\sim y^{a_2}_2$. As $y^{a_2}_2$ is not extremal, $y^{a_3}_3$ exists and it is of type $\hat t_1$. As $t_2$ separates $r_1$ from $t_1$, Lemma~\ref{lem:transitive} implies that $z_2\sim y^{a_3}_3$. Thus $y^{a_3}_3$ is an extremal vertex in $Z_2\cap Z_3$ and we take $u'_2=y^{a_3}_3$.
	\end{proof}
	
	\begin{proof}[Proof of Claim~\ref{claim:identical}]
		Take $(u_1,u_2,u_3,u_4)\in \Xi_{\min}$. Assume $\{u_1,u_2,u_3,u_4\}$ are pairwise distinct. We will either deduce a contradiction, which implies two of them must be the same, or we will replace one of $u_i$ to obtain another element in $\Xi_{\min}$ satisfying Claim~\ref{claim:identical}.
		For $i=2,4$, let $v_i$ be the vertex in the normal form path from $u_i$ to $u_3$ that is adjacent to $u_3$ ($v_2=y_{k-1}$). 
		Claim~\ref{claim:decrease} and Proposition~\ref{prop:bestvina} imply that $d(u_3,u_1)>d(v_2,u_1)$. Similarly, $d(u_3,u_1)>d(v_4,u_1)$. By Corollary~\ref{cor:convex}, $v_4\in Z_4$. Let $u_3vv_0$ be the first three vertices in the normal form path from $u_3$ to $u_1$ (which exists by Proposition~\ref{prop:existence and uniqueness}), see Figure~\ref{fig:coner3} (3).
		
		If we are in Figure~\ref{fig:bowtie} (2) with $\tau(u_3)=3$, then $v_2<u_3$ and $v_4<u_3$. As both $v_2$ and $v$ are in geodesics from $u_1$ to $u_3$, $v<u_3$. So $v<v_0$. As $d(u_1,v_2)<d(u_1,u_3)$, Lemma~\ref{lem:strip} implies that $v_2\le v$. Similarly, $v_4\le v$. Thus $v_i<v_0$ for $i=2,4$. As $z_3\sim v_2$ and $v_2\sim v_0$, Lemma~\ref{lem:transitive} implies that $z_3\sim v_0$. Similarly, $z_4\sim v_0$. Let $v'_0$ be an extremal vertex in $\Delta_{\Lambda,\Lambda'}$ with $v'_0\ge v_0$. Then $v'_0\sim z_i$ for $i=3,4$, and $d(v'_0,u_1)=d(u_3,u_1)-2$. Moreover, $d(u_2,v'_0)\le d(u_2,u_3)$ and $d(u_4,v'_0)\le d(u_4,u_3)$. So if we replace $u_3$ by $v'_0$, the quantity \eqref{eq:ds2} will not increase. If $v'_0=u_1$, then we are done. Now we assume $v'_0\neq u_1$. By the choice of $\{u_i\}_{i=1}^4$, $d(u_2,v'_0)=d(u_2,u_3)$ and $d(u_4,v'_0)= d(u_4,u_3)$. 
		Let $y'_1y'_2$ be the first two vertices in the local normal form path from $u_2$ to $v'_0$. As $d(v'_0,u_1)=d(u_3,u_1)-2$, by Proposition~\ref{prop:bestvina}, $d(u_1,y'_1)<d(u_1,y'_2)$ is impossible. Claim~\ref{claim:decrease} implies that we can replace $u_2$ to decrease \eqref{eq:ds2}, which is a contradiction. 
		Suppose $\tau(u_3)=1$. Then $u_3<v_4$. As $z_3\sim u_3$ and $u_3\sim v_4$, we have $z_3\sim v_4$. Take $v'_4$ to be an extremal vertex in $\Delta_{\Lambda,\Lambda'}$ with $v'_4\ge v_4$. Then $v'_4\sim z_3$, $v'_4\sim z_4$, $d(v'_4,u_4)=d(u_3,u_4)-1$ and $d(v'_4,u_2)\le d(u_3,u_2)+1$.  By the choice of $\{u_i\}_{i=1}^4$, $d(v'_4,u_2)=d(u_3,u_2)+1$. We replace $u_3$ by $v'_4$. If $v'_4=u_4$ or $u_1$, then we are done. Suppose $v'_4\neq u_4$ and $v'_4\neq u_1$. As $v_4$ is in a geodesic (which is an up-down path) from $u_1$ to $u_3$, so is $v'_4$. Hence $d(u_1,v'_4)=d(u_1,u_3)-1$. Let $y'_1y'_2$ be the first two vertices in the local normal form path from $u_2$ to $v'_4$. Again by Proposition~\ref{prop:bestvina} $d(u_1,y'_1)<d(u_1,y'_2)$ is impossible, and we reach a contradiction.
		
		Suppose we are in Figure~\ref{fig:bowtie} (4) and $\tau(u_3)=3$. Then $z_3\sim u_3$, $u_3\sim v_4$, $u_3$ being of type $\hat t_3$ and $u_3>v_4$ imply that $z_3\sim v_4$. Then we finish in the same way as the previous paragraph. Suppose $\tau(u_3)=1$. Then $v\le v_2$. We deduce from $z_3\sim v_2$ and $v_2\sim v$ that $z_3\sim v$. If $v=v_4$ then $z_3\sim v_4$ and we finish as before. Suppose $v\neq v_4$. Then $\tau(v)=2$ and $\tau(v_4)=3$.
		Consider the 4-cycle $vv_4z_4z_3$ in $\lk(u_3,\Delta_\Lambda)$. We can view cycle as in $\Delta^{u_3}_{\Lambda\setminus\{t_1\},\Omega}$ by Lemma~\ref{lem:link}. As $\Delta^{u_3}_{\Lambda\setminus\{t_1\},\Omega}$ is bowtie free, by considering the meet of $z_3$ and $v_4$ in $(\Delta^{u_3}_{\Lambda\setminus\{t_1\},\Omega})^0$ (suppose vertices of type $\hat r_1$ are minimal), either $z_3\sim v_4$ or $z_4\sim v$. In the former case, we conclude as before. In the latter case, $v\in Z_3\cap Z_4$. As $\tau(v)=2$ and $u_3<v$, we have $v_0<v$ and $\tau(v_0)=1$. Lemma~\ref{lem:transitive} implies $v_0\in Z_3\cap Z_4$ and we conclude as in the previous paragraph.
		
		Suppose we are in Figure~\ref{fig:bowtie} (3). Then cases of $\tau(u_3)=3$ and $\tau(u_3)=1$ are symmetric, so we only treat $\tau(u_3)=3$. Consider $z_3z_4v_4vv_2$ in $\lk(u_3,\Delta_\Lambda)$. See Figure~\ref{fig:coner3} (4) for the following discussion. By condition C$_5$ and Lemma~\ref{lem:C_k link inherit}, there exist $v'_2,v'$ of the same type as $v_2,v$ respectively such that $z_3\sim v'_2$, $v'_2\sim v'$, $v'\sim v_4$ and $z_4\sim\{v'_2,v'\}$. We consider $z_3v'_2v'v_4vv_2$, viewed as in $\Delta^{u_3}_{\Lambda\setminus\{t_3\},\Omega'}$, where $\Omega'$ is the line segment from $r_m$ to $t_1$. We assume vertices of type $\hat t_1$ are maximal in $(\Delta^{u_3}_{\Lambda\setminus\{t_3\},\Omega'})^0$. Then $(\Delta^{u_3}_{\Lambda\setminus\{t_3\},\Omega'})^0$ is bowtie free and upward flag. Note that $\{v',v,z_3\}$ are pairwise upper bounded in $(\Delta^{u_3}_{\Lambda\setminus\{t_3\},\Omega'})^0$, hence they have a common upper bound, denoted by $z$. We can assume $z$ is of type $\hat t_1$. Now consider $zv'v_4v$ in $(\Delta^{u_3}_{\Lambda\setminus\{t_3\},\Omega'})^0$. Note that $z$ and $v_4$ are common upper bounds for $\{v',v\}$. It follows from bowtie free property of $(\Delta^{u_3}_{\Lambda\setminus\{t_3\},\Omega'})^0$ that either $z=v_4$ or $v'=v$. In the former case, $z_3\sim v_4$ and we finish as before. In the latter case, $z_4\sim v$. Note that $z_4$ has type $\hat r_1$, $v$ has type $\hat t_1$ or $\hat t_2$, and $v\ge v_2$, we deduce that $z_4\sim v_2$, which is symmetric to the case of $z_3\sim v_4$. 
	\end{proof}
	
	In the remainder the proof, we are in the case of Figure~\ref{fig:bowtie} (1). If $d(u_2,u_3)=1$, then $Z_2\cap Z_4\neq\emptyset$ (indeed, if $u_2<u_3$, then $z_2\sim u_3$; if $u_2>u_3$, then $u_2\sim z_4$). If $d(u_2,u_3)=2$ and $\tau(u_2)=1$, then $y_1<y_2>y_3$ and consequently $y_2\in Z_2\cap Z_4$. By Claim~\ref{claim:reduce}, we will assume $d(u_2,u_3)\ge 2$ if $\tau(u_2)=3$ and $d(u_2,u_3)\ge 3$ if $\tau(u_2)=1$. By symmetry, we also assume $d(u_4,u_3)\ge 2$ if $\tau(u_4)=1$ and $d(u_4,u_3)\ge 3$ if $\tau(u_4)=3$. Take $(u_1,u_2,u_3,u_4)\in \Xi_{\min}$.
	
	\begin{claim}
		\label{claim:decrease2}
		Under such assumptions, if $\tau(u_2)=3$, then there is a geodesic path $\omega'=w_1w_2\cdots w_k\subset Z_3$ from $u_2$ to $u_3$ such that the function $f(i)=d(w_i,u_1)$ is strictly increasing in $[2,k]$. If $\tau(u_2)=1$, then $d(y_i,u_1)$ is strictly increasing in $[3,k]$.
	\end{claim}
	
	\begin{proof}
		Suppose $\tau(u_2)=3$. Then
		we deduce from $z_2\sim y^{a_2}_2$, $y^{a_2}_2\sim y^{a_2-1}_2$ and Lemma~\ref{lem:transitive} that $z_3\sim y^{a_2-1}_2$. Similarly $z_3\sim y^{a_1-1}_1$. Thus $y^{a_1-1}_1\in Z_2\cap Z_3$.
		If $y^{a_1-1}_1$ is extremal, then replacing $u_2$ by $y^{a_1-1}_1$ decreases \eqref{eq:ds2}. So $y^{a_1-1}_1$ is not extremal. Then $\tau(y^{a_1-1}_1)=2$. Moreover, $y^{a_1-2}_1$ exists and $\tau(y^{a_1-2}_1)=3$.

		First we consider the case $d(y_2,u_1)<d(u_2,u_1)$. We aim to show $d(y_3,u_1)>d(y_2,u_1)$, then the claim follows from Proposition~\ref{prop:bestvina}. We first treat the subcase $d(y_3,u_1)=d(y_2,u_1)$. Then $y^{a_i-1}_i>y^{a_i}_i$  for $i=2,3$. Note that $y^{a_2}_2<y^{a_3}_3$ (as $\omega$ is an up-down path), thus $\tau(y^{a_3}_3)=\tau(y^{a_1-1}_1)=2$, and $\tau(y^{a_2}_2)=1$.
		If $y^{a_3-1}_3=y^{a_1-2}_1$, then we consider the 4-cycle $z_3y^{a_3}_3y^{a_3-1}_3y^{a_1-1}_1$, viewed as in $\Delta^{y_2}_{\Lambda\setminus\{t_1\},\Omega}$, where $\Omega$ is the line segment from $t_2$ to $r_m$. Since $d(y^{a_1}_1,y^{a_3}_3)=2$, $y^{a_3}_3\neq y^{a_1-1}_1$. Then we deduce from the bowtie free property of $\Delta^{y_2}_{\Lambda\setminus\{t_1\},\Omega}$ that $z_3\sim y^{a_1-2}_1$. Thus $y^{a_1-2}_1$ is an extremal vertex in $Z_2\cap Z_3$. Then replacing $u_2$ by $y^{a_1-2}_1$ decreases \eqref{eq:ds2}. Now we assume  $y^{a_3-1}_3\neq y^{a_1-2}_1$. By considering the 6-cycle $z_3y^{a_3}_3y^{a_3-1}_3y^{a_2-1}_2y^{a_1-2}_1y^{a_1-1}_1$ in $\lk(y^{a_2}_2,\Delta_\Lambda)$, we know $\{z_3,y^{a_3-1}_3,y^{a_1-2}_1\}$ is pairwise upper bounded in $(\Delta^{y_2}_{\Lambda\setminus\{t_1\},\Omega})^0$ (assume vertices of type $\hat t_2$ are maximal). Thus there is a vertex $z$ of type $\hat t_2$ such that $z\sim \{z_3,y^{a_3-1}_3,y^{a_1-2}_1,y^{a_2}_2\}$. If $z\neq y^{a_1-1}_1$, then we consider the 4-cycle $y^{a_1-1}_1z_3zy^{a_1-2}_1$ in $\Delta^{y_2}_{\Lambda\setminus\{t_1\},\Omega}$ and deduce that $z_3\sim y^{a_1-2}_1$, and conclude as before. If  $z= y^{a_1-1}_1$, then $z\neq y^{a_3}_3$. Consider the 4-cycle $zz_3y^{a_3}_3y^{a_3-1}_3$ in $\Delta^{y_2}_{\Lambda\setminus\{t_1\},\Omega}$, we deduce $z_3\sim y^{a_3-1}_3$. As $z_2\sim z$ and $z\sim y^{a_3-1}_3$, we deduce from Lemma~\ref{lem:transitive} that $z_2\sim y^{a_3-1}_3$. Thus $y^{a_3-1}_3$ is an extremal vertex in $Z_2\cap Z_3$. Replacing $u_2$ by $y^{a_3-1}_3$ decreases \eqref{eq:ds2}. This finishes the subcase $d(y_3,u_1)=d(y_2,u_1)$. The subcase $d(y_3,u_1)<d(y_2,u_1)$ is similar and simpler, so we omit the details.
		
		Next we consider the case $d(y_2,u_1)=d(u_2,u_1)$. As $y^{a_1}_1>y^{a_2}_2>y^{a_2-1}_2$, $\tau(y^{a_2-1}_2)=1$. Note that $z_3\sim y^{a_2-1}_2$. We set $w_1=z_3$ and $w_2=y^{a_2-1}_2$. Let $w_2w_3\cdots w_k$ be the local normal form path from $w_2$ to $u_3$. As $w_2\in Z_3$, we know $w_2\cdots w_k\subset Z_3$ by Corollary~\ref{cor:convex}. By the same argument in the previous paragraph (with the roles of $y_2,y_3$ replaced by $w_2,w_3$, and $y^{a_3-1}_3$ replaced by the second vertex in the local normal form path from $w_3$ to $u_1$), we know $d(w_3,u_1)>d(w_2,u_1)$. Hence the claim follows from Proposition~\ref{prop:bestvina}.
		
		Now assume $\tau(u_2)=1$. If $d(y_2,u_1)>d(y_1,u_1)$, then the claim follows from Proposition~\ref{prop:bestvina}.
		If $d(y_2,u_1)<d(y_1,u_1)$, then $z_2\sim y_2$ and replacing $u_2$ by $y_2$ decreases \eqref{eq:ds2}. It remains to consider $d(y_2,u_1)=d(y_1,u_1)$. As $y^{a_2-1}_2>y_2>y_3$, $d(y_3,u_1)\le d(y_2,u_1)$. However, $d(y_3,u_1)<d(y_2,u_1)$ is ruled out by Proposition~\ref{prop:local vs global} and Definition~\ref{def:normal}. So $d(y_2,u_1)=d(y_3,u_1)$.
		Next we prove $d(y_4,u_1)>d(y_3,u_1)$. We first rule out $d(y_4,u_1)=d(y_3,u_1)$. Lemma~\ref{lem:transitive} implies that $z_2\sim y^{i}_2$ for $a_2-2\le i\le a_2$. As $y_3<y_4<y^{a_4-1}_4$, $\tau(y_4)=2$. Thus $y_5$ exists and $y_5<y_4$. Moreover, we argue as before to deduce $d(y_5,u_1)=d(y_4,u_1)$. As $y_1<y_2>y_3<y_4$, Lemma~\ref{lem:strip} implies that $y^{a_1-1}_1\le y^{a_2-1}_2\ge y^{a_3-1}_3\le y^{a_4-1}_4$. 
		Note that $y^{a_2-1}_2\neq y^{a_3-1}_3$, indeed, if this is not the case, as $y_1<y_2<y^{a_2-1}_2$, we have $\tau(y^{a_2-1}_2)=3$; hence from $y^{a_2-1}_2= y^{a_3-1}_3$ and $y^{a_3-1}_3\le y^{a_4-1}_4$ we deduce $y^{a_2-1}_2=y^{a_4-1}_4$. 
		This implies $d(y_1,y_5)\le 2<4$ by considering $y_1y^{a_2-1}_2y_5$, contradicting that $y_1\cdots y_k$ is a normal form path (Proposition~\ref{prop:local vs global}). Now we consider $z_2y_2y_3y^{a_3-1}_3y^{a_3-2}_3y^{a_2-2}_2$ in $\lk(y^{a_2-1}_2,\Delta_\Lambda)$. Let $\Omega'$ be the line segment from $r_1$ to $t_2$, and let $*=y^{a_2-1}_2$. Then Lemma~\ref{lem:link} implies that we can view $$z_2y_2y_3y^{a_3-1}_3y^{a_3-2}_3y^{a_2-2}_2$$ as in $\Delta^*_{\Lambda\setminus\{t_3\},\Omega'}$. Thus there is a vertex $z$ of type $\hat t_2$ such that $z\sim \{y^{a_2-1}_2,z_2,y_3,y^{a_3-2}_3\}$. By considering $zy_3y^{a_3-1}_3y^{a_3-2}_3$, we deduce from the bowtie free property of  $\Delta^*_{\Lambda\setminus\{t_3\},\Omega'}$ that $z=y^{a_3-1}_3$. This gives a 4-cycle $z_2y_2y_3z$ in $\Delta_{\Lambda\setminus\{t_3\},\Omega'}$. By bowtie free again, $z_2\sim y_3$. Thus replacing $u_2$ by $y_3$ decreases \eqref{eq:ds2}. This rules out $d(y_4,u_1)=d(y_3,u_1)$. We can rule out $d(y_4,u_1)<d(y_3,u_1)$ by a similar argument. Thus $d(y_4,u_1)>d(y_3,u_1)$ and the claim follows from Proposition~\ref{prop:bestvina}.
	\end{proof}
	
	To conclude the proof in the case of Figure~\ref{fig:bowtie} (1), we argue in the same way as the proof of Claim~\ref{claim:identical}, using Claim~\ref{claim:decrease2}, to replace $u_3$ by another extremal vertex in $Z_3\cap Z_4$ such that the new quadruple is still in $\Xi_{\min}$, but $d(u_1,u_3)$ decreases. Repeating this procedure for finitely many times we end up in one of the situations discussed right before Claim~\ref{claim:decrease2}. 
\end{proof}

\subsection{Labeled 4-cycle property for type $ABI$ diagrams}

\begin{lem}
	\label{lem:widetilde b3 core}
	Suppose $\Delta_\Lambda$ is modeled on a doubly marked tree diagram $\Lambda$. Suppose the link of each nonempty simplex in $\Delta_\Lambda$  is atomic BD-robust. 
	Let $\Lambda'$ be a $\widetilde B_3$-like subdiagram with its vertex set labeled as Figure~\ref{fig:BD}. Then the $(b_1,b_2)$-subdivision of $\Delta_{\Lambda,\Lambda'}$ is $\wtC$-like. %In particular, this holds when $\Lambda$ is a tree Coxeter diagram of type $ABI$ with edge labels $\le 5$.
\end{lem}

\begin{proof}
	We first consider the case that both $b_1b_3$ and $b_2b_3$ are plain. 	
	Then $b_1b_3b_2$ is $\Lambda\setminus\{b_4\}$-atomic $D_3$-like, hence $\Delta^*_{\Lambda\setminus\{b_4\},b_1b_3b_2}$ satisfies Assumption 1 of Proposition~\ref{prop:propagation}. Note that $b_1b_3b_4$ is $\Lambda\setminus\{b_2\}$-atomic $B_3$-like, and $b_2b_3b_4$ is $\Lambda\setminus\{b_1\}$-atomic $B_3$-like. Thus Assumption 2 of Proposition~\ref{prop:propagation} holds true. Hence the lemma follows.
	
	If $b_1b_3$ is plain and $b_2b_3$ is marked, then $b_1b_3b_2$ is $\Lambda\setminus\{b_4\}$-atomic $B_3$-like. Hence Lemma~\ref{lem:d3 upward flag} implies that $\Delta^*_{\Lambda\setminus\{b_4\},b_1b_3b_2}$ satisfies Assumption 1 of Proposition~\ref{prop:propagation}. Note that $b_2b_3b_4$ is robust $\wtC$-core with respect to $\Delta^*_{\Lambda\setminus\{b_1\}}$. Hence $\Delta^*_{\Lambda\setminus\{b_1\},b_2b_3b_4}$ is upward flag and bowtie free by Lemma~\ref{lem:big lattice}. Now the lemma follows from Proposition~\ref{prop:propagation}.
	
	The remaining cases can be handled in a similar way. 
\end{proof}

\begin{prop}
	\label{prop:labeled 4-cycle}
	Suppose $\Delta_\Lambda$ is modeled on  a tree Coxeter diagram of type $ABI$ with edge labels $\le 5$ which is not $\wtC$-free. Suppose the link of each nonempty simplex in $\Delta_\Lambda$ is atomic BD-robust, is B-robust and satisfies the labeled 4-cycle condition. Then $\Delta_{\Lambda}$ satisfies the labeled 4-cycle condition.
\end{prop}

\begin{proof}
	By Proposition~\ref{prop:labeled 4-cycle with C_2 subdiagram}, it suffices to consider the case when $\Lambda$ does not contain any $\wtC_2$-like subdiagram. As before, let $\Omega=r_1\cdots r_m$ be a maximal linear subdiagram of $\Lambda$ and it suffices to show $\Delta_{\Lambda,\Omega}$ is bowtie free.
	As $\Lambda$ is of type $ABI$, $\Omega$ has at most two edges with label $\ge 4$, and such edges, if exist, appear at the two ends of $\Omega$. If $\Omega$ has two edges with label $\ge 4$, as $\Delta^*_{\Lambda\setminus\{t\}}$ is $B$-robust for each terminal vertex $t$ of $\Omega$, we deduce from Theorem~\ref{thm:contractibleII} that $\Omega$ is a robust $\wtC$-core of $\Lambda$. Lemma~\ref{lem:big lattice} implies that $\Delta_{\Lambda,\Omega}$ is bowtie free. Suppose $\Omega$ has at most one edges with label $\ge 4$. We assume such edge, if exists, is $r_{m-1}r_m$. At least one interior vertex of $\Omega$ has valence $\ge 3$ in $\Lambda$, otherwise $\Lambda=\Omega$ is $\wtC$-free. Suppose $r_{m-1}$ has valence $\ge 3$. Let $e'$ be an edge based at $r_{m-1}$ with $e'\nsubseteq \Omega$. Let $\Lambda'=(r_{m-2}r_{m-1}r_m)\cup e'$. If $r_{m-1}r_m$ has label $\ge 4$, then the $(r_{m-2},r')$-subdivision of $\Delta_{\Lambda,\Lambda'}$ is a $\wtC$-like complex by Lemma~\ref{lem:widetilde b3 core}, where $r'$ is the vertex of $e'$ outside $\Omega$. Then $(\Omega,\Lambda')$ is the configuration in Figure~\ref{fig:core1} (5) (with $\Omega$ playing the role of $\Lambda_0$). By Corollary~\ref{cor:robust0}, $\Delta_{\Lambda,\Omega}$ is bowtie free.
	
	Now we assume either $r_{m-1}$ has valence $2$ in $\Lambda$, or $r_{m-1}$ has valence $3$ and $r_{m-1}r_m$ has label $3$. Let $r_k$ be an interior vertex of $\Omega$ with valence $\ge 3$ in $\Lambda$. As $\Lambda$ is of type $ABI$ without $\wtC_2$-like subdiagrams, $r_k$ has valence $3$ in $\Lambda$ and the edge $e'$ at $r_k$ outside $\Omega$ has label $\ge 4$. Let $\Lambda'=(r_{k-1}r_kr_{k+1})\cup e'$. Then Lemma~\ref{lem:widetilde b3 core} implies that the $(r_{k-1},r_{k+1})$-subdivision $\bar\Delta$ of $\Delta_{\Lambda,\Lambda'}$ is $\wtC$-like. Let $\tau$ be the rank function on the vertex set of $\bar \Delta$ as in Definition~\ref{def:subdivision}, taking value between $1$ and $4$. We define the vertex $r'\in e'$ as before.

	Let $z_1z_2z_3z_4$ be an embedded 4-cycle in $\Delta_\Lambda$ of type $\hat r_m\hat r_1\hat r_m\hat r_1$. Let $Z_i$ be the full subcomplex of $\Delta_{\Lambda,\Lambda'}$ spanned by vertices that are adjacent to $z_i$, and let $\bar Z_i$ be the subcomplex of $\bar\Delta$ corresponding to $Z_i$. Then $\bar Z_i\subset \bar\Delta$ satisfies the conclusion of Corollary~\ref{cor:convex}. By considering a maximal simplex of $\Delta_\Lambda$ containing the edge $z_iz_{i+1}$, we know $\bar Z_i\cap\bar Z_{i+1}$ contains an extremal vertex of $\bar\Delta$ for each $i$.
	\begin{claim*}
		Either $\bar Z_1\cap\bar Z_3\neq\emptyset$ or $\bar Z_2\cap\bar Z_4\neq\emptyset$. 
	\end{claim*}
	
	\begin{proof}
		Let $\Xi$ be the collection of all quadruples $(u_1,u_2,u_3,u_4)$ such that each $u_i$ is an extremal vertex of $\bar\Delta$ in $\bar Z_i\cap \bar Z_{i+1}$ for $i\in \mathbb Z/4\mathbb Z$. For each element in $\Xi$, we consider the quantity in \eqref{eq:ds2}, where $d$ is the distance function in the 1-skeleton of $\bar\Delta$ with unit edge length.
		Let $\Xi_{\min} \subset \Xi$ be the collection of elements in $\Xi$ such that \eqref{eq:ds2} is minimized. We define  $\omega=y_1\cdots y_k$ and $y^{a_i}_iy^{a_i-1}_i\cdots y^1_i$ in the same way as in the proof of Proposition~\ref{prop:labeled 4-cycle with C_2 subdiagram}. Take $(u_1,u_2,u_3,u_4)\in \Xi_{\min}$. We assume these $u_i$'s are pairwise distinct, otherwise the claim already follows.
		
		Now we prove $d(y_1,u_1)<d(y_2,u_1)$.
		
		Suppose $\tau(u_2)=1$. We first rule out $d(y_1,u_1)=d(y_2,u_1)$. Indeed, if this is the case, then we argue in the same way as in Proposition~\ref{prop:labeled 4-cycle with C_2 subdiagram} (or Lemma~\ref{lem:corner1}) to deduce that $y^i_1\le y^i_2$ for any $i$ and $y_2<y^{a_2-1}_2$. As $\tau(y_2)\ge 2$ and $\tau(y^{a_1-1}_1)\ge 2$, we deduce from $z_3\sim y_2$ and $y_2<y^{a_2-1}_2$ that $z_3\sim y^{a_2-1}_2$, and we deduce from $z_2\sim y^{a_1-1}_1$ and $y^{a_1-1}_1\le y^{a_2-1}_2$ that $z_2\sim y^{a_2-1}_2$. Thus $y^{a_2-1}_2\in Z_2\cap Z_3$. Let $u'_2$ be an extremal vertex of $\bar \Delta$ that is $\ge y^{a_2-1}_2$. Then $u'_2\in Z_2\cap Z_3$ and replacing $u_2$ by $u'_2$ decreases \eqref{eq:ds2}. This rules out $d(y_1,u_1)=d(y_2,u_1)$. Similarly we can rule out $d(y_1,u_1)>d(y_2,u_1)$. 
		
		Suppose $\tau(u_2)=4$. If $d(y_1,u_1)=d(y_2,u_1)$, then $u_2>y_2>y^{a_2-1}_2$ and $u_2>y^{a_1-1}_1\ge y^{a_2-1}_2$. Thus $\tau(y^{a_2-1}_2)=1$ or $2$. If $\tau(y^{a_2-1}_2)=2$, then $y^{a_2-1}_2$ is the midpoint of an edge $vv'$ of $\Delta_{\Lambda,\Lambda'}$ such that $v'$ is of type $\hat r_{k-1}$ and $v$ is of type $\hat r_{k+1}$. As $r_k$ separates $r_{k-1}$ from $r_m$, we deduce $z_3\sim v'$. Similarly, $z_2\sim v$. Consider the 4-cycle $z_3y_2vz_2$ in $\lk(y_1,\Delta_\Lambda)$, which can also be viewed as inside $\Delta^{u_2}_{\Lambda\setminus\{r'\},\Omega}$. Using bowtie freeness of $\Delta^{u_2}_{\Lambda\setminus\{r'\},\Omega}$ and Lemma~\ref{lem:posets}, we know either $z_2\sim y_2$ or $z_3\sim v$. In the former case, as $y_2<y_3$ and $\tau(y_2)\ge 2$, we know $z_2\sim y_3$ and replacing $u_2$ by any extremal vertex of $\bar\Delta$ that is $\ge y_3$ decreases \eqref{eq:ds2}. In the latter case, $z_3\sim y^{a_2-1}_2$. If $z_2\sim y^{a_2-1}_2$, then $\{z_2,z_3,v,v'\}$ span a simplex in $\Delta_{\Lambda}$. Then replacing $u_2$ by $v$ or $v'$ decreases \eqref{eq:ds2}. If $z_2\nsim y^{a_2-1}_2$, then by considering $z_2z_3v'y^{a_1-1}_1$ in $\Delta^{u_2}_{\Lambda\setminus\{r'\},\Omega}$, we obtain that $z_3\sim y^{a_1-1}_1$. As $\tau(y^{a_1-1}_1)\ge\tau(y^{a_2-1}_2)=2$, $z_3\sim y^{a_1-2}_1$, if $y^{a_1-2}_1$ exists. Then replacing $u_2$ by any extremal vertex of $\bar\Delta$ that is $\ge y^{a_1-2}_1$ (or replacing $u_2$ by $y^{a_1-1}_1$ if $y^{a_1-2}_1$ does not exist) decreases \eqref{eq:ds2}.

		Now suppose $\tau(y^{a_2-1}_2)=1$. Then $y^{a_2-1}_2$ is of type $\hat r_{k-1}$ or $\hat r_{k+1}$. We will only treat the case that $y^{a_2-1}_2$ is of type $\hat r_{k-1}$, as the other case is similar. Then $z_3\sim y^{a_2-1}_2$. If $z_2\sim y^{a_2-1}_2$, then we replace $u_2$ to decrease \eqref{eq:ds2} as before. Suppose $z_2\nsim y^{a_2-1}_2$. Then $\tau(y^{a_1-1}_1)=3$. Consider $z_2z_3y^{a_2-1}_2y^{a_1-1}_1$ in $\lk(u_2,\Delta_\Lambda)$, also viewed as in $\Delta^{u_2}_{\Lambda\setminus\{r'\},\Omega}$, we deduce from the bowtie free property that $z_3\sim y^{a_1-1}_1$ and replace $u_2$ to decrease \eqref{eq:ds2} in the same way as before. This rules out $d(y_1,u_1)=d(y_2,u_1)$. The case of $d(y_1,u_1)>d(y_2,u_1)$ can be ruled out in a similar way.
		
		Next we show that as long as $\{u_i\}_{i=1}^4$ are pairwise distinct, then it is possible to replace $u_3$ by another extremal vertex in $\bar Z_3\cap \bar Z_4$ such that the resulting quadruple is still contained in $\Xi_{\min}$, but $d(u_1,u_3)$ decreases.
		Let $v_2\in \bar Z_3$, $v_4\in \bar Z_4$ and $v,v_0$ be defined identical to the proof of Claim~\ref{claim:identical} in Proposition~\ref{prop:labeled 4-cycle with C_2 subdiagram}, see Figure~\ref{fig:coner3} (3). As $d(y_1,u_1)<d(y_2,u_1)$, by Proposition~\ref{prop:bestvina}, $d(u_3,u_1)>d(v_2,u_1)$. Similarly, $d(u_3,u_1)>d(v_4,u_1)$. 
		
		Suppose $\tau(u_3)=1$. Then $u_3$ is of type $\hat r_{k-1}$ or $\hat r_{k+1}$. If $u_3$ is of type $\hat r_{k-1}$, as $r_{k-1}$ separates $r_1$ from any vertices of $\Lambda'\setminus\{r_{k-1}\}$, Lemma~\ref{lem:transitive} implies that $z_4\sim v_2$. Hence $v_2\in Z_3\cap Z_4$. Let $\sigma$ be a maximal simplex of $\Delta_\Lambda$ containing $\{z_3,z_4,v_2\}$. Then we can replace $u_3$ by the maximal vertex $u'_3$ of $\bar\Delta$ in $\sigma$. Note that $d(u'_3,u_i)\le d(u_3,u_i)-1$ for $i=1,2$ and $d(u'_3,u_4)\le d(u_3,u_4)+1$. The case of $u_3$ being type $\hat r_{k+1}$ is symmetric.
		
		Suppose $\tau(u_3)=4$. First we treat the case $\tau(v_2)\ge 2$. As $v_2\le v$, $z_3\sim v$. If $\tau(v_4)\ge 2$, then $z_4\sim v$. As $v\le v_0$ and $\tau(v)\ge \tau(v_2)=2$, $v_0\in \bar Z_3\cap \bar Z_4$. Then we replace $u_3$ by any extremal vertex in $\bar\Delta$ that is $\ge v_0$. Now assume $\tau(v_4)=1$. If $v_4$ is of type $\hat r_{k-1}$, then $z_4\sim v$ and we conclude as before. Suppose $v_4$ is of type $\hat r_{k+1}$. If $\tau(v)=2$, then $v$ is the midpoint of an edge in $\Delta_{\Lambda,\Lambda'}$ with one endpoint being $v_4$. Thus $z_3\sim v$ implies $z_3\sim v_4$ and we replace $u_3$ by $v_4$. Suppose $\tau(v)=3$. Then $v$ is of type $\hat r_k$. Consider the 4-cycle $z_3z_4v_4v$ in $\lk(u_3,\Delta_\Lambda)$, which can also be viewed as a 4-cycle in $\Delta^{u_3}_{\Lambda\setminus\{r'\},\Omega}$. As $\Delta^{u_3}_{\Lambda\setminus\{r'\},\Omega}$ is bowtie free, either $z_3\sim v_4$ or $v\sim z_4$. So we replace $u_3$ by $v_4$ in the former case, and $u_3$ by an extremal vertex of $\bar\Delta$ that is $\ge v_0$ in the latter case. By symmetry, the case of $\tau(v_4)\ge 2$ can be treated similarly. From now on we assume $\tau(v_2)=\tau(v_4)=1$. If $v_2$ is of type $\hat r_{k+1}$, then $z_3\sim v_2$ and $v_2\sim v$ imply $z_3\sim v$. Hence we can repeat the argument as before. By symmetry, the same argument applies if $v_4$ is of type $\hat r_{k-1}$.

		It remains to consider $v_2$ is of type $\hat r_{k-1}$ and $v_4$ is of type $\hat r_{k+1}$. If $v_2=u_2$, then $z_2\sim v_2$. As $z_2$ is of type $\hat r_1$ and $v_2<u_3$, we obtain $z_2\sim u_3$, hence $u_3\in Z_2\cap Z_4$. Thus we will assume $v_2\neq u_2$ in the following discussion. In particular, $y_{\ell-2}$ exists.
		As $\tau(v)\le 3$, $v_2\sim v_4$. Let $v_2v'_2v''_2$ be the first three vertices in the normal form path from $v_2$ to $u_1$ (then $v'_2=y^{a_{k-1}-1}_{k-1}$). Consider $z_3u_3v_4v_0v'_2y_{k-2}$ in $\lk(v_2,\Delta_\Lambda)$ (if $\tau(v'_2)=2$, then we replace $v'_2$ by $v''_2$). 
		This 6-cycle can also be viewed as in $\Delta^{v_2}_{\Lambda\setminus\{r'\},\Omega'}$ where $\Omega'$ is the line segment from $r_m$ to $r'$. As $y^{k-2}>v_2<u_3$, by Lemma~\ref{lem:strip}, $y^{k-2}\ge v'_2\le v_0$. Thus $\{v_4,v'_2,z_3\}$ is pairwise upper bounded in $(\Delta^{v_2}_{\Lambda\setminus\{r'\},\Omega'})^0$. Hence there is a vertex $z$ of type $\hat r'$ such that $z\sim \{v_4,v'_2,z_3,v_2\}$. Then $d(z,u_1)\le d(v'_2,u_1)+1=d(u_3,u_1)-1$.
		Consider $z_3u_3v_4z$ in $\Delta^{v_2}_{\Lambda\setminus\{r'\},\Omega'}$. If $z_3\sim v_4$, then we replace $u_3$ by $v_4$ as before. So we assume $z_3\nsim v_4$. Note that $u_3\neq z$ as they are at different distances from $u_1$ in $\bar\Delta^1$. Thus $z_3u_3v_4z$ is an induced 4-cycle in  $\Delta^{v_2}_{\Lambda\setminus\{r'\},\Omega'}$. As $\Delta^{v_2}_{\Lambda\setminus\{r'\},\Omega'}$ is bowtie free, there is a vertex $z'$ of type $\hat r_k$ such that $z'$ is a center for $z_3u_3v_4z$. Consider $z_3z'v_4z_4$ in $\lk(u_3,\Delta_{\Lambda})$, also viewed as in $\Delta^{u_3}_{\Lambda\setminus\{r'\},\Omega}$. As $z_3\nsim v_4$, we deduce from bowtie free property of $\Delta^{u_3}_{\Lambda\setminus\{r'\},\Omega}$ that $z'\sim z_4$. As $z_4,z',z$ have types $\hat r_1,\hat r_k,\hat r'$ respectively and $r_k$ separates $r_1$ from $r'$, we obtain $z_4\sim z$. Hence $z\in Z_3\cap Z_4$. Note that $$d(z,u_4)\le d(z,v_4)+d(v_4,u_4)=1+d(u_3,u_4)-1=d(u_3,u_4)$$ and
		$$d(z,u_2)\le d(z,v'_2)+d(v'_2,y_{k-2})+d(y_{k-2},u_2)=d(u_2,u_3).$$ So replacing $u_3$ by $z$ satisfies all the requirements.
		
		As such replacement can be done when $\{u_i\}_{i=1}^4$ are pairwise distinct, after finitely many replacements, two of $\{u_1,u_2,u_3,u_4\}$ are identical, and the claim follows.
	\end{proof}
	
	By Corollary~\ref{cor:robust0}, $\Delta_{\Lambda,\Lambda(r_1,r',r_{k+1})}$ satisfies the labeled 4-cycle condition, where $\Lambda(r_1,r',r_{k+1})$ is the smallest subtree of $\Lambda$ containing $\{r_1,r',r_{k+1}\}$. In particular, $\Delta_{\Lambda,\Lambda(r_1,r_{k+1})}$ is bowtie free. Similarly, $\Delta_{\Lambda,\Lambda(r_{k-1},r_m)}$ is bowtie free. Thus we can reduce the proposition to the above claim in the same way as Claim~\ref{claim:reduce} of Proposition~\ref{prop:labeled 4-cycle with C_2 subdiagram}. This finishes the proof.
\end{proof}

\subsection{A variation of intersection lemma}
We prove a small variation of Proposition~\ref{prop:intersection}.

\begin{assumption}
	\label{assum:2}
	Let $\Delta_S$ be a nondegenerate regular simplicial complex of type $S$, modeled on a doubly marked tree $\Lambda$, and satisfying C$_5$. Let $\Lambda'$ be a  $\widetilde C$-core robust with respect to $\Delta_S$. Let $\Lambda_0$ be a $B_m$-like $(m\ge 2)$ or $D_m$-like $(m\ge 4)$ subdiagram of $\Lambda$.
	
	If $\Lambda_0$ is $B_m$-like, we label $\Lambda_0=s_1\ldots s_m$ with $s_{m-1}s_m\in E_m$ and assume that $\Lambda_0\cap \Lambda'$ is the edge $s_{m-1}s_m$, as in Figure~\ref{fig:core} (1), (2), or $\Lambda_0\cap \Lambda'$ is $s_{m-1}s_m$ together with another edge as in Figure~\ref{fig:core}  (5).
	If $\Lambda_0$ is $D_m$-like, we label the vertices of $\Lambda_0$ as in Figure~\ref{fig:core} (0) and assume that $\Lambda_0\cap \Lambda'$ is $D_4$-like, as in Figure~\ref{fig:core} (3), (4). We define $\bar\Delta_{\Lambda,\Lambda_0}$ and the poset structure on its vertex set in the same way as Assumption~\ref{assum:1}. Let $t_1,t'_1$, $\{\theta_i\}_{i=1}^3,\{X_i\}_{i=1}^3$, $\{\bar X_i\}_{i=1}^3$ and $\tau$ be as in Assumption~\ref{assum:1}.
	
	In the case when $t'_1$ exists, let $\Gamma'=(\Lambda'\cup\Lambda_0)\setminus\{t'_1\}$. Let $\Gamma$ be the connected component of $\Lambda\setminus\{t'_1\}$ that contains $\Gamma'$. For each vertex $*$ of type $\hat t'_1$, let $\bar\Delta^*_{\Gamma,\Gamma'}$ be either $\Delta^*_{\Gamma,\Gamma'}$ or its appropriate subdivision (when $\Gamma'$ is $D_m$-like). We order vertices in $\bar\Delta^*_{\Gamma,\Gamma'}$ such that vertices of type $\hat b_1$ or $\hat s_1$ are minimal.
\end{assumption}

\begin{lem}
	\label{lem:corner2}
	Under Assumption~\ref{assum:2}, suppose $(\bar\Delta^*_{\Gamma,\Gamma'})^0$ is upward flag for any vertex $*$ of type $\hat t'_1$. 
	Let $(u_1,u_2,u_3)$ be three extremal vertices in $\bar\Delta_{\Lambda,\Lambda'}$ such that $u_1,u_2\in \bar X_2$ and $u_2,u_3\in \bar X_3$. Suppose $u_2$ is chosen such that
	\begin{equation}
		\label{eq:ds1}
		d(u_1,u_2)+d(u_2,u_3)
	\end{equation}
	is minimized among all extremal vertices $u_2
	\in \bar X_2\cap \bar X_3$.
	Let $\omega=y_1y_2\cdots y_k$ be a path from $u_2$ to $u_3$ which is either in local normal form from $u_2$ to $u_3$ or in local normal form from $u_3$ to $u_2$. Suppose $d(u_1,u_2)\ge 2$.
	\begin{enumerate}
		\item If $\tau(u_2)=1$, then $d(u_1,y_2)>d(u_1,u_2)$;
		\item If $\tau(u_2)=n$, then there is at most one $i$ with $d(u_1,y_i)<d(u_1,u_2)$, and if such $i$ exists then $i=2$.
	\end{enumerate}
\end{lem}

\begin{proof}
	The proof is identical to Lemma~\ref{lem:corner1} except the case when $\theta(u_2)=n$, $d(y_2,u_1)<d(y_1,u_1)$ and $y_2$ is of type $\hat t'_1$. The proof of Lemma~\ref{lem:corner1} in this case relies on the existence of $z$ in Figure~\ref{fig:core}, which is not guaranteed by the assumption of this lemma, so we need to find an alternative argument (in the same line as the $\widetilde D_4$-like subcase of Lemma~\ref{lem:corner1}). Our goal is to rule out $d(y_3,u_1)\le d(y_2,u_1)$. See Figure~\ref{fig:corner2} (I) for the following discussion.

	If $d(y_3,u_1)=d(y_2,u_1)$, as $y^{a_1-1}_1<y^{a_1-2}_1$, we know $y^{a_2}_2<y^{a_2-1}_2$ and $y^{a_3}_3<y^{a_3-1}_3$.
	As $y_2<y_3$, by Lemma~\ref{lem:strip}, $y^{a_2-1}_2<y^{a_3-1}_3$. As $\tau(y^{a_3}_3)\ge 2$, $x_3\sim y^{a_3}_3$ and $y^{a_3}_3\sim y^{a_3-1}_3$, we know $x_3\sim y^{a_3-1}_3$. 
	Note that $y^{a_2}_2\sim\{y^{a_1-1}_1,y^{a_2-1}_2,x_3\}$, and $y^{a_2}_2\neq \{x_3,y^{a_2-1}_2\}$. We can assume $y^{a_2}_2\neq y^{a_1-1}_1$, otherwise $y^{a_2}_2\in X_2\cap X_3$ and replacing $u_2$ by $y^{a_2}_2$ decreases \eqref{eq:ds2} (note that $y^{a_2}_2$ is an extremal vertex in $X_2\cap X_3$). In particular $\tau(y^{a_1-1}_1)\ge 2$. Let $\Gamma,\Gamma'$ be as in Assumption~\ref{assum:2}. By Lemma~\ref{lem:link}, we view $\{y^{a_1-1}_1,y^{a_2-1}_2,x_3\}$ as in $\bar\Delta^{y_2}_{\Gamma,\Gamma'}$ that are pairwise upper bounded modulo the following consideration:
	\begin{itemize}
		\item we can assume $\tau(y^{a_1-1}_1)\neq 2$, otherwise $y^{a_1-1}_1$ is the middle point of an edge in $\Delta_\Lambda$, hence $x_2\sim y^{a_2}_2$ and replacing $u_2$ by $y^{a_2}_2$ decreases \eqref{eq:ds2};
		\item if $\tau(y^{a_2-1}_2)=2$, then $y^{a_2}_2\sim y^{a_2-2}_2$ and we replace $y^{a_2-1}_2$ by $y^{a_2-2}_2$.
	\end{itemize}
	Thus $\{y^{a_1-1}_1,y^{a_2-1}_2,x_3\}$ have a common upper bound $z$ in $((\bar\Delta^{y_2}_{\Gamma,\Gamma'})^0,<)$. We can assume $z$ is maximal in $((\bar\Delta^{y_2}_{\Gamma,\Gamma'})^0,<)$, hence $z\in \bar\Delta_{\Lambda,\Lambda'}$ and $\tau(z)=n$. Then $$d(z,u_1)\le d(z,y^{a_2-1}_2)+d(y^{a_2-1}_2,u_1)\le 1+a_2-1<a_1.$$
	%As $y^{a_2-2}_2<y^{a_2-1}_2$ and $y^{a_2-1}_2\le z$, we know $y^{a_2-2}_2<z$ hence $d(z,u_1)\le a_2-1\le a_1-2$. 
	As $x_2\sim y^{a_1-1}_1$ and $y^{a_1-1}_1\sim z$, we deduce from $\tau(y^{a_1-1}_1)>1$ and Lemma~\ref{lem:transitive} that $x_2\sim z$. Hence $z\in X_2\cap X_3$. 
	As $y^{a_3}_3$ is not extremal, $y^{a_4}_4$ must exist. Then $y^{a_4}_4<y^{a_3}_3$, hence $y^{a_4}_4<y^{a_3-1}_3$ and $d(z,y^{a_4}_4)\le 3$. It follows that $d(z,u_3)\le d(u_2,u_3)$. Thus replacing $u_2$ by $z$ decreases \eqref{eq:ds0}.
	The case of $d(y_3,u_1)<d(y_2,u_1)$ can be ruled out similarly. Thus $d(y_3,u_1)>d(y_2,u_1)$ and Proposition~\ref{prop:bestvina} implies $d(y_j,u_1)\ge d(u_1,u_2)$ for any $j\ge 3$.
\end{proof}

\begin{figure}[h]
	\centering
	\includegraphics[scale=1]{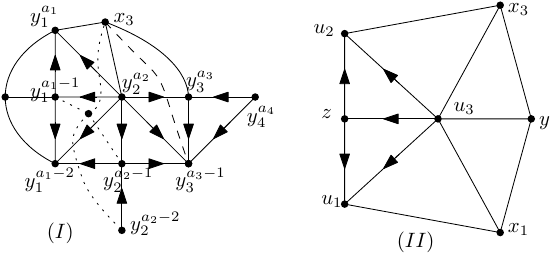}
	\caption{Diagrams in the proof of Lemma~\ref{lem:corner2} and Proposition~\ref{prop:intersection2}.}
	\label{fig:corner2}
\end{figure}

\begin{prop}
	\label{prop:intersection2}
	Under Assumption~\ref{assum:2},	suppose $(\bar\Delta^*_{\Gamma,\Gamma'})^0$ is upward flag for any vertex $*$ of type $\hat t'_1$. 
	If any pairwise intersection of $\{X_1,X_2,X_3\}$ contains a vertex whose type is not $\hat t'_1$, then $X_1\cap X_2\cap X_3\neq\emptyset$.
\end{prop}

\begin{proof}
	By Lemma~\ref{lem:transitive}, if $\bar X_i\cap \bar X_j$ contains a vertex which is not of type $\hat t'_1$, then $\bar X_i\cap \bar X_j$ contains a vertex with $\theta$-value $n$. Thus $\bar X_i\cap \bar X_j$ always contain an extremal vertex. Let $(u_1,u_2,u_3)$ be as in the proof of Proposition~\ref{prop:intersection} which minimizes \eqref{eq:ds1}. We will prove two of $\{u_1,u_2,u_3\}$ are the same, which implies the proposition. For now we assume $\{u_1,u_2,u_3\}$ are pairwise distinct, and aim at deducing contradictions. By the same argument in the proof of Proposition~\ref{prop:intersection}, we are reduced to consider $d(u_i,u_j)\le 2$ for any $1\le i\neq j\le 2$. We will use the same notation as the proof of Proposition~\ref{prop:intersection} in the following discussion.
	
	%If among $\{d(u_1,u_2),d(u_2,u_3),d(u_1,u_3)\}$, one of them (say $d(u_1,u_2)$) is $\ge 3$, and another one (say $d(u_2,u_3)$) is $\ge 2$, then Lemma~\ref{lem:corner2} implies that $d(u_1,u_3)\ge d(u_1,u_2)\ge 3$. If we reserve the role of $u_2$ and $u_3$, then $d(u_1,u_2)\ge d(u_1,u_3)\ge 3$. So $d(u_1,u_3)=d(u_1,u_2)\ge 3$. Similarly, $d(u_2,u_1)=d(u_2,u_3)\ge 3$. Then Lemma~\ref{lem:corner2} and Proposition~\ref{prop:bestvina} imply that vertices of the normal form path from $u_i$ to $u_j$ (or $u_j$ to $u_i$) have constant distance to $u_k$ with $\{i,j,k\}$ pairwise distinct. Thus we can repeat the argument in the proof of Proposition~\ref{prop:intersection} to deduce a contradiction.  

	%By triangle inequality, it remains to consider $d(u_i,u_j)\le 2$ for any $1\le i\neq j\le 2$. 
	
	Suppose two of $\{u_1,u_2,u_3\}$, say $u_2,u_3$, are adjacent in $\bar\Delta_{\Lambda,\Lambda'}$. If $d(u_1,u_2)=d(u_1,u_3)$, then $y^{a_1}_1<y^{a_1-1}_1$ if and only if $y^{a_2}_2<y^{a_2-1}_1$. Thus one of $u_2,u_3$ is not extremal, contradiction. Assume without loss of generality that $d(u_1,u_2)>d(u_1,u_3)$. Then Lemma~\ref{lem:corner2} (1) implies $
	\tau(u_2)=n$, hence $u_3<u_2$. Then $u_3\le y^{a_1-1}_1$ by Definition~\ref{def:normal}. As $d(u_1,u_2)\le 2$, $u_1$ and $u_3$ are adjacent. Let $z=y^{a_1-1}_1$, see Figure~\ref{fig:corner2} (II). If $u_3$ is not of type $\hat t'_1$, then by $u_3<u_1$ and Lemma~\ref{lem:transitive}, $u_1\in \bar X_3$. Hence $u_1\in \bar X_1\cap \bar X_2\cap \bar X_3$. Suppose $u_3$ is of type $\hat t'_1$. By our assumption, there is a vertex $y\in \bar X_2\cap \bar X_3$ such that $y$ is not of type $\hat t'_1$. By C$_4$ (applying to $u_3x_3yx_1$), we can assume $y$ is adjacent to $u_3$ in $\bar\Delta_{\Lambda,\Lambda'}$. By considering $x_3yx_1u_1zu_2$ around $u_3$ and using Lemma~\ref{lem:link}, we deduce that $\{x_3,x_1,z\}$ are pairwise upper bounded in $((\bar\Delta^{u_3}_{\Gamma,\Gamma'})^0,<)$, hence they have a common upper bound $w$ which can taken to be extremal in  $((\bar\Delta^{u_3}_{\Gamma,\Gamma'})^0,<)$. Then $w\in \bar\Delta_{\Lambda,\Lambda'}$ and $\tau(w)=n$. As $z>u_3$, $\tau(z)>1$. Then $x_2\sim z$, $z\le w$ and Lemma~\ref{lem:transitive} imply $x_2\sim w$. Hence $w\in \bar X_1\cap \bar X_2\cap \bar X_3$.
	
	It remains to consider $d(u_i,u_j)=2$ for $1\le i\neq j\le 3$. This is similar to the last paragraph of the proof of Proposition~\ref{prop:intersection} and we omit the details.
\end{proof}
%	We assume $\{u_1,u_2,u_3\}$ are pairwise non-adjacent in $\bar\Delta_{\Lambda,\Lambda'}$. Suppose $d(u_i,u_j)=2$ for any $i\neq j$. We assume $\theta(u_2)=n$ (otherwise we can decrease \eqref{eq:ds} as in the proof of Lemma~\ref{lem:corner1}). Then $\theta(u_1)=\theta(u_3)=n$. 
%	Hence $\{y_2,y^2_1,y^2_3\}$ are pairwise upper bounded in $((\bar\Delta_{\Lambda,\Lambda'})^0,<)$. Lemma~\ref{lem:big lattice} implies they have a common upper bound $z$. We can assume $\theta(z)=n$. Proposition~\ref{prop:convex subcomplex} implies $x_2\sim y^2_1,x_1\sim y^2_3$ and $x_3\sim y_2$. If none of $\{y_2,y^2_1,y^2_3\}$ has type $\hat t'_1$, then $z\in \bar X_1\cap\bar X_2\cap\bar X_3$ by Lemma~\ref{lem:transitive}. Suppose one of $\{y_2,y^2_1,y^2_3\}$ has type $\hat t'_1$. By Lemma~\ref{lem:strip}, $y^2_2\le y^2_1$ and $y^2_2\le y^2_3$. Moreover, $y^2_2\le y_2$. Thus $y^2_2$ has type $\hat t'_1$. Let $x'_3=y_2$ if $y^2_2\neq y_2$ and $x'_3=x_3$ if $y^2_2=y_2$. Similarly, let $x'_2$ be the element in $\{x_2,y^2_1\}$ that is adjacent to $y^2_2$, and $x'_1$ be the element in $\{x_1,x^2_3\}$ that is adjacent to $y^2_2$. Then $\{x'_1,x'_2,x'_3\}$ are pairwise upper bounded in $((\bar\Delta_{\Gamma,\Gamma'})^0,<)$, hence has a common upper bound $w$ that can be taken to be extremal. Then $\theta(w)=n$ and Lemma~\ref{lem:transitive} implies $w\in \bar X_1\cap\bar X_2\cap\bar X_3$. 
%\end{proof}

\subsection{Propagation of $B$-robustness}

%\begin{lem}
%Suppose $\Lambda$ is a tree Coxeter diagram of type $ABI$ without any $\wtC_2$-like subdiagram. Then $\Lambda$ contain a $A_n$-type subdiagram $\Lambda_1$ such that each edge of $\Lambda$ not contained in $\Lambda_1$ has label $>3$ and intersects $\Lambda_1$ in a vertex. Moreover each vertex of $\Lambda$ has valence $3$.
%\end{lem}
%
%\begin{proof}
%Let $\Lambda$ be a tree Coxeter diagram of type $ABI$ without $\wtC_2$-like subdiagram.	
%Let $\Lambda_1$ be the induced subdiagram of $\Lambda$ spanned by edges with label $3$. Note that each connected component of $\Lambda_1$ is of type $A_n$ as $\Lambda$ does not contain $D_4$-subdiagram. Note that $\Lambda_1$ only has one connected component. Indeed, if $\Lambda_1$ has more than one connected components, let $\Lambda_2$ be a shortest possible edge path connecting two different connected components of $\Lambda_1$ . By definition, each edge of $\Lambda_2$ has label $\ge 4$. As $\Lambda$ does not contain $\wtC_2$-subdiagram, $\Lambda_2$ is a single edge. Then $\Lambda_2$ together with edges from the two components of $\Lambda_1$ where $\Lambda_2$ is connecting gives a $F_4$-subdiagram of $\Lambda$, contradiction. Similarly, by using $\Lambda$ does not have $\wtC_2$-diagram, we know each vertex in $\Lambda\setminus\Lambda_1$ has distance 1 to $\Lambda_1$, and each vertex in $\Lambda$ has valence $3$. Now the lemma follows.
%\end{proof}

\begin{prop}
	\label{prop:prop B-robustness}
	Let $\Delta_\Lambda$ be modeled on a tree Coxeter diagram  $\Lambda$ of type ABI with edge labels $\le 5$. Suppose $\Delta_\Lambda$ satisfies property (P) in Section~\ref{subsec:propagation}. 
Suppose the link of each nonempty simplex in $\Delta_\Lambda$ is atomic BD-robust, is $B$-robust and satisfies the labeled 4-cycle condition. Suppose $\Lambda$ is not $\wtC$-free. Then $\Delta_{\Lambda}$ satisfies the labeled 4-cycle condition and is $B$-robust.
\end{prop}

\begin{proof}
	By Proposition~\ref{prop:labeled 4-cycle}, it suffices to show $\Delta_\Lambda$ is $B$-robust. Let $\Omega=r_1\cdots r_m$ be a $B_m$-like subdiagram of $\Lambda$ with $r_{m-1}r_m\in E_m$. If $\Omega$ has another edge $e'$ in $E_m$, then let $\Lambda'$ be the subdiagram spanned by $r_{m-1}r_m$ and $e'$. As any proper subdiagram of $\Lambda$ is $B$-robust, $\Delta_{\Lambda,\Lambda'}$ is $\wtC$-like. 
	As $\Delta_\Lambda$ satisfies the labeled 4-cycle condition (Proposition~\ref{prop:labeled 4-cycle}), $\Delta_{\Lambda,\Omega}$ is bowtie free.
	To show $\Delta_{\Lambda,\Omega}$ is upward flag, by Corollary~\ref{cor:upward flag criterion}, it suffices to show if $\{p_1,p_2,p_3\}$ are pairwise upper bounded in $(\Delta_{\Lambda,\Omega})^0$ such that $p_1,p_2,p_3$ have the same type $\hat t$, then there is a common upper bound $p$ for them in $(\Delta_{\Lambda,\Omega})^0$. This is true if $t\in \Lambda'$, as $(\Delta_{\Lambda,\Lambda'})^0$ is flag by Lemma~\ref{lem:big lattice}. If $t\in \Omega\setminus \Lambda'$, then we let $X_i$ be the full subcomplex of $\Delta_{\Lambda,\Lambda'}$ spanned by vertices that are adjacent to $p_i$ in $\Delta_\Lambda$. Note that $X_i\cap X_j$ contains a vertex of type $\hat r_m$ for any $1\le i\neq j\le 3$. Thus $X_1\cap X_2\cap X_3\neq\emptyset$ by Proposition~\ref{prop:intersection2}. Any vertex in such intersection is a common upper bound for $\{p_1,p_2,p_3\}$.
	
	Now assume $\Omega$ only has one edge $e=r_{m-1}r_m$ with label $\ge 4$. As $\Lambda$ is of type $ABI$, but not $\wtC$-free, at least one of the following possibilities holds true:
(1) there are at least two edges labeled by $3$ at the vertex $r_{m-1}$ or at $r_m$;
(2) there is another edge $e'$ of $\Lambda$ with label $\ge 4$ such that the subdiagram $\Lambda'$ of $\Lambda$ spanned by $e$ and $e'$ satisfies that either $\Omega\subset \Lambda'$, or $\Lambda'\setminus e'\subset \Omega$. or $\Omega\cap \Lambda'=r_{m-1}r_m$. 

	We first consider case (1). If the two edges in case (1), denoted by $e_1,e_2$, are based at $r_m$, then $m=2$ (otherwise we obtain a $F_4$-subdiagram of $\Lambda$). Let $\Lambda'=r_{m-1}r_m\cup e_1\cup e_2$. 
	%Then $e_1\cup e_2$ is $\Lambda\setminus\{r_1\}$-atomic $D_3$-like, hence Theorem~\ref{thm:atomic BD robust} implies that $\Delta_{\Lambda\setminus\{r_{m-1}\},e_1\cup e_2}$ satisfies assumption (2) of Proposition~\ref{prop:propagation}. 
	By Lemma~\ref{lem:widetilde b3 core}, $\Lambda'$ is a robust $\wtC$-core with respect to $\Delta_\Lambda$. Hence $\Delta_{\Lambda,\Omega}$ is a graph with girth $\ge 8$ by Lemma~\ref{lem:big lattice}. Hence $\Delta_{\Lambda,\Omega}$ is 2-sided flag. It remains to consider $e_1$ and $e_2$ are based at $r_{m-1}$. We can assume $e_1=r_{m-2}r_{m-1}$ if $m\ge 3$. Let $\Lambda'$ be as before, which is a robust $\wtC$-core with respect to $\Delta_\Lambda$ by Lemma~\ref{lem:widetilde b3 core}. Then the pair $(\Omega,\Lambda')$ corresponds to Figure~\ref{fig:core} (5). Note that the requirement in paragraph 5 of Assumption~\ref{assum:1} holds true. Let $\{p_1,p_2,p_3\}$ be pairwise upper bounded elements in $(\Delta_{\Lambda,\Omega})^0$ of the same type, and we define $\{X_i\}_{i=1}^3$ as before. Then Proposition~\ref{prop:intersection} implies that $X_1\cap X_2\cap X_3\neq\emptyset$. Let $p$ be a vertex in the common intersection and suppose $p$ has type $\hat u$. If $u\in \Omega$, then $p$ is a common upper bound for $\{p_1,p_2,p_3\}$. If $u\notin \Omega$, then we can assume $\{p_1,p_2,p_3\}$ is pairwise upper bounded in $\Delta^p_{\Lambda\setminus\{u\},\Omega}$ (indeed, if the upper bound $q_i$ of $p_i$ and $p_{i+1}$ is not adjacent to $p$ in $\Delta_\Lambda$, then we use C$_4$ to produce $q'_i$ of the same type as $q_i$ such that $q'_i$ is adjacent to $\{p_i,p_{i+1},p\}$ in $\Delta_\Lambda$). Thus $\{p_1,p_2,p_3\}$ have a common upper bound in  $\Delta^p_{\Lambda\setminus\{u\},\Omega}$, hence in $\Delta_{\Lambda,\Omega}$.
	
	It remains to consider case (2). As before, $\Delta_{\Lambda,\Lambda'}$ is $\wtC$-like. We assume $\Omega$ is not contained in $\Lambda'$, otherwise $\Delta_{\Lambda,\Omega}$ is upward flag by Lemma~\ref{lem:big lattice}. If $\Omega\cap \Lambda'=r_{m-1}r_m$, then $(\Omega,\Lambda')$ belongs to the configuration in Figure~\ref{fig:core} (2). Consider $\{p_i\}_{i=1}^3$ and $\{X_i\}_{i=1}^3$ as before. By Proposition~\ref{prop:intersection}, there is a vertex $p\in X_1\cap X_2\cap X_3$, and we conclude in the same way as the previous paragraph. It remains to consider $\Lambda'\setminus e'\subset \Omega$. 
	
	Suppose $r_k=e'\cap \Omega$. Let $r'$ be the other vertex of $e'$, and let $\Lambda'=(r_{k-1}r_k\cdots r_m)\cup e'$, which is $\widetilde B_n$-like for some $n$. We claim $\Lambda'$ is a robust $\wtC$-core with respect to $\Delta_\Lambda$. It suffices to show the two assumptions of Proposition~\ref{prop:ori link0} hold true. Note that $\Lambda'\setminus \{r_{k-1}\}$ is a robust $\wtC$-core with respect to each $\Delta^*_{\Lambda\setminus\{r_{k-1}\}}$. Thus Assumption 2 of Proposition~\ref{prop:ori link0} follows from Lemma~\ref{lem:big lattice}. Now take a vertex $x$ of type $\hat r_m$ and let $\Gamma$ be the connected component of $\Lambda\setminus\{r_m\}$ containing $\Lambda'\setminus \{r_m\}$. If $k=m-1$, then $\Lambda'\setminus\{r_{m}\}=r_{m-2}r_{m-1}r'$, and $\Delta^x_{\Gamma,r_{m-2}r_{m-1}r'}$ is bowtie free and upward flag (with vertices of type $\hat r'$ being maximal). By Lemma~\ref{lem:d3 upward flag}, Assumption 1 of Proposition~\ref{prop:ori link0} holds true. Suppose $k<m-1$. Let $\Lambda''=(r_{k-1}r_kr_{k+1})\cup e'$, which is $\widetilde B_3$-like. By Lemma~\ref{lem:widetilde b3 core}, $\Lambda''$ is a robust $\wtC$-core with respect to $\Delta^x_{\Gamma}$. Then the pair $(\Lambda'\setminus\{r_m\},\Lambda'')$ corresponds to Figure~\ref{fig:core} (4), where the thickened part corresponds to $\Lambda''$ and $b_{m-2}b_m=e'$. %Note that $\Delta_{\Lambda\setminus\{r_m\}}$ is atomic BD-robust by Proposition~\ref{prop:ABI BD-robust}.
	Now we argue in the same way as the second paragraph of the proof of Lemma~\ref{lem:weak atomic}, using Corollary~\ref{cor:robust0}, to deduce that Assumption 1 of Proposition~\ref{prop:ori link0} holds true for $\Delta^x_{\Gamma,\Lambda'\setminus\{r_m\}}$. Thus the claim is proved. 
	
	The subdiagram $\Theta=\Omega\cup \Lambda'$ corresponds to Figure~\ref{fig:core} (4), with the thickened subdiagram being $\Lambda'$. Let $\{p_1,p_2,p_3\}\subset(\Delta_{\Lambda,\Omega})^0$ be as before such that they are of type $\hat t$. If $t\in \Lambda'$, then the existence of a common upper bound for $\{p_i\}_{i=1}^3$ follows from the previous claim and Lemma~\ref{lem:big lattice}. If $t\notin \Lambda'$, then we define $\{X_i\}_{i=1}^3$ as before. Taking $\Lambda_0=(r_1\cdots r_{k+1})\cup e'$ in Proposition~\ref{prop:intersection2}.
	As $\Delta^*_{\Lambda\setminus\{r'\},\Omega}$ is upward flag, Proposition~\ref{prop:intersection2} implies that $\cap_{i=1}^3 X_i\neq\emptyset$, and we conclude as before.
\end{proof}

\begin{proof}[Proof of Proposition~\ref{prop:ABI bowtie free general}]
	By Theorem~\ref{thm:atomic BD robust} and Assumption 2, the link of every
	(possibly empty) simplex in $\Delta_\Lambda$ is atomic
	BD-robust. We prove by induction on $|V\Gamma|$ that
	$\Delta^\sigma_{\Lambda\setminus T}$ is $B$-robust and
	satisfies the labeled 4-cycle condition, where $\sigma$
	is any simplex of type $\widehat T$ and
	$\Gamma=\Lambda\setminus T$.
	If $\Gamma$ is $\wtC$-free, the claim follows from
	Assumption 2, since $\Gamma$ is of type $A_n$, $B_n$,
	or $I_2(5)$. Otherwise, we apply Proposition~\ref{prop:prop B-robustness} to $\Delta^\sigma_{\Gamma}$ and induct on number of vertices in $\Gamma$ (the property (P) assumption of Proposition~\ref{prop:prop B-robustness} follows from the second part of Assumption 2).
\end{proof}

\begin{thm}
	\label{thm:ABI general}
	Let $\Lambda$ be a connected Coxeter diagram of type $ABI$ that is not of type $A_n,B_n$ or $I_2(n)$. Suppose $\Delta_\Lambda$ is a simplicial complex modeled on $\Lambda$ such that 
	\begin{enumerate}
		\item  $\Delta_\Lambda$ is nondegenerate, regular and satisfies C$_5$;
		\item any special 4-cycles has a center
		and any special 6-cycle  has a quasi-center in each link of $\Delta_\Lambda$ modeled on $A_n$ or $B_n$. 
	\end{enumerate}
	Then $\Delta_\Lambda$ is contractible. 
\end{thm}

\begin{proof}
(2) implies that links modeled on $A_n$ or $B_n$ are
BD-robust and bowtie free; the same holds for links modeled on
$I_2(5)$ by Definition~\ref{def:model}.
Proposition~\ref{prop:ABI bowtie free general} gives bowtie
freeness for links modeled on trees without marked edges,
and \cite[Thm 10.12]{huangbestvina} extends this to all tree
subdiagrams.
If $\Lambda$ contains a cycle, contractibility follows from
\cite[Thm 10.10]{huangbestvina}. If $\Lambda$ is a tree with
$E_h\neq\emptyset$, it follows from
\cite[Thm 10.12]{huangbestvina}. Otherwise, $\Lambda$ is a
tree with $E_h=\emptyset$ and is not $\wtC$-free.
Every $\wtC$-free subdiagram is then of type $A_n$, $B_n$,
or $I_2(5)$, so Theorem~\ref{thm:geometry} applies.
\end{proof}

%The following can be proved in a similar as Theorem~\ref{thm:ABI general}, using Lemma~\ref{lem:group}. 
%
%\begin{thm}
%	Let $\Lambda$ be a connected Coxeter diagram of type $ABI$ such that is not type $A_n,B_n$ or $I_2(n)$. Let $(G,\{H_s\}_{s\in S})$ be a nondegenerate, strongly admissible parabolic system over $\Lambda$, with its coset complex $\Delta_S$. Suppose that for any induced subdiagram $\Lambda'$ of $\Lambda$ of type $A_n$ or $B_n$, any special 4-cycle in $\Delta_{\Lambda'}$ has a center
%	and any special 6-cycle in $\Delta_{\Lambda'}$ has a quasi-center.
%	Then $\Delta_S$ is contractible. 
%	
%	Moreover, $G$ acts cocompactly by isometries on either a Bestvina complex of a categorical Garside structure, or a piecewise Euclidean CAT$(0)$ complex, or an injective orthoscheme complex with stabilizers of faces being isomorphic to proper standard parabolic subgroups of $G$.
%\end{thm}
%Here a (proper) standard parabolic subgroup of $G$ is a subgroup generated by a (strict) sub-collection of elements from $\{H_s\}_{s\in S}$.
\bibliographystyle{alpha}
\bibliography{mybib}

\newcommand{\etalchar}[1]{$^{#1}$}
\begin{thebibliography}{CCG{\etalchar{+}}25}

\bibitem[AS83]{AppelSchupp1983}
Kenneth~I. Appel and Paul~E. Schupp.
\newblock Artin groups and infinite {C}oxeter groups.
\newblock {\em Invent. Math.}, 72(2):201--220, 1983.

\bibitem[BDS87]{Bandelt_Dahlmann_Schutte:Absolute_Retract_Bipartite:1987}
H.-J. Bandelt, A.~D\"ahlmann, and H.~Sch\"utte.
\newblock Absolute retracts of bipartite graphs.
\newblock {\em Discrete Appl. Math.}, 16(3):191--215, 1987.

\bibitem[Bes99]{Bestvina1999}
Mladen Bestvina.
\newblock Non-positively curved aspects of {A}rtin groups of finite type.
\newblock {\em Geom. Topol.}, 3:269--302 (electronic), 1999.

\bibitem[Bes06]{bessis2006garside}
David Bessis.
\newblock Garside categories, periodic loops and cyclic sets.
\newblock {\em arXiv preprint math/0610778}, 2006.

\bibitem[BFH93]{Bandelt_Farber_Hell:AbsoluteReflexiveBipartiteRetract:1993}
Hans-J\"urgen Bandelt, Martin Farber, and Pavol Hell.
\newblock Absolute reflexive retracts and absolute bipartite retracts.
\newblock {\em Discrete Appl. Math.}, 44(1-3):9--20, 1993.

\bibitem[BM10]{brady2010braids}
Thomas Brady and Jonathan~P. McCammond.
\newblock Braids, posets and orthoschemes.
\newblock {\em Algebr. Geom. Topol.}, 10(4):2277--2314, 2010.

\bibitem[CCG{\etalchar{+}}25]{CCGHO:Helly:2025}
J\'er\'emie Chalopin, Victor Chepoi, Anthony Genevois, Hiroshi Hirai, and
  Damian Osajda.
\newblock Helly groups.
\newblock {\em Geom. Topol.}, 29(1):1--70, 2025.

\bibitem[CCHO14]{weaklymodular}
J{\'e}r{\'e}mie Chalopin, Victor Chepoi, Hiroshi Hirai, and Damian Osajda.
\newblock Weakly modular graphs and nonpositive curvature.
\newblock {\em arXiv preprint arXiv:1409.3892. To appear in Mem. Amer. Math.
  Soc.}, 2014.

\bibitem[CD95]{CharneyDavis}
Ruth Charney and Michael~W. Davis.
\newblock The ${K}(\pi,1)$-problem for hyperplane complements associated to
  infinite reflection groups.
\newblock {\em J. Amer. Math. Soc.}, 8(3):597--627, 1995.

\bibitem[Cha04]{charney2004deligne}
Ruth Charney.
\newblock The {D}eligne complex for the four-strand braid group.
\newblock {\em Transactions of the American Mathematical Society},
  356(10):3881--3897, 2004.

\bibitem[CMS10]{callegaro2010k}
Filippo Callegaro, Davide Moroni, and Mario Salvetti.
\newblock The ${K} (\pi, 1)$ problem for the affine {A}rtin group of type
  {$\widetilde B_n$} and its cohomology.
\newblock {\em Journal of the European Mathematical Society (EMS Publishing)},
  12(1), 2010.

\bibitem[CMV23]{cumplido2020parabolic}
Mar{\'\i}a Cumplido, Alexandre Martin, and Nicolas Vaskou.
\newblock Parabolic subgroups of large-type {A}rtin groups.
\newblock {\em Mathematical Proceedings of the Cambridge Philosophical
  Society}, 174(2):393--414, 2023.

\bibitem[Del72]{deligne}
Pierre Deligne.
\newblock Les immeubles des groupes de tresses g\'en\'eralis\'es.
\newblock {\em Invent. Math.}, 17:273--302, 1972.

\bibitem[EM02]{elder2002curvature}
Murray Elder and Jon McCammond.
\newblock Curvature testing in 3-dimensional metric polyhedral complexes.
\newblock {\em Experimental Mathematics}, 11(1):143--158, 2002.

\bibitem[GH25]{goldman2025deligne}
Katherine Goldman and Amy Herron.
\newblock The {D}eligne complex for the ${B}_3$ {A}rtin group.
\newblock {\em arXiv:2503.15820}, 2025.

\bibitem[Gol24]{goldman2022k}
Katherine Goldman.
\newblock The ${K} (\pi, 1)$ conjecture and acylindrical hyperbolicity for
  relatively extra-large {A}rtin groups.
\newblock {\em Algebraic \& Geometric Topology}, 24(3):1487--1504, 2024.

\bibitem[Gol26]{goldman2026}
Katherine Goldman.
\newblock Parabolic subgroups of shephard groups.
\newblock {\em In preparation}, 2026.

\bibitem[GP12]{godelle2012k}
Eddy Godelle and Luis Paris.
\newblock ${K}(\pi,1)$ and word problems for infinite type {A}rtin--{T}its
  groups, and applications to virtual braid groups.
\newblock {\em Mathematische Zeitschrift}, 272(3):1339--1364, 2012.

\bibitem[Hae24]{haettel2021lattices}
Thomas Haettel.
\newblock Lattices, injective metrics and the ${K}(\pi,1)$ conjecture.
\newblock {\em Algebraic \& Geometric Topology}, 24(7):4007--4060, 2024.

\bibitem[Hae25]{haettel2022link}
Thomas Haettel.
\newblock A link condition for simplicial complexes, and {CUB} spaces.
\newblock {\em Mathematische Annalen}, 393(2):1939--1987, 2025.

\bibitem[HH25]{haettel2023new}
Thomas Haettel and Jingyin Huang.
\newblock New {G}arside structures and applications to {A}rtin groups.
\newblock {\em Duke Mathematical Journal}, 174(9):1665--1722, 2025.

\bibitem[Hir20]{hirai2020uniform}
Hiroshi Hirai.
\newblock Uniform modular lattices and affine buildings.
\newblock {\em Advances in Geometry}, 20(3):375--390, 2020.

\bibitem[HKS16]{haettel20166}
Thomas Haettel, Dawid Kielak, and Petra Schwer.
\newblock The 6-strand braid group is {${\rm CAT}(0)$}.
\newblock {\em Geom. Dedicata}, 182:263--286, 2016.

\bibitem[HM]{HodaMunro:BiHelly_Directed_Geodesics:2026}
Nima Hoda and Zachary Munro.
\newblock Combing bi-{H}elly graphs.
\newblock In preparation.

\bibitem[HO21]{huang2021helly}
Jingyin Huang and Damian Osajda.
\newblock Helly meets {G}arside and {A}rtin.
\newblock {\em Inventiones mathematicae}, 225(2):395--426, 2021.

\bibitem[HP25]{huang2025353}
Jingyin Huang and Piotr Przytycki.
\newblock 353-combinatorial curvature and the $3 $-dimensional ${K}(\pi,1)$
  conjecture.
\newblock {\em arXiv preprint arXiv:2509.06914}, 2025.

\bibitem[Hua24a]{huang2024}
Jingyin Huang.
\newblock Cycles in spherical {D}eligne complexes and application to
  ${K}(\pi,1)$-conjecture for {A}rtin groups.
\newblock {\em arXiv preprint arXiv:2405.12068}, 2024.

\bibitem[Hua24b]{huang2023labeled}
Jingyin Huang.
\newblock Labeled four cycles and the ${K}(\pi,1)$-conjecture for {A}rtin
  groups.
\newblock {\em Inventiones mathematicae}, 238(3):905--994, 2024.

\bibitem[Hua26]{huangbestvina}
Jingyin Huang.
\newblock Bestvina metric and tree reduction for ${K}(\pi,1)$ conjecture.
\newblock {\em arXiv preprint arXiv:2602.17982}, 2026.

\bibitem[J{\'S}06]{JanuszkiewiczSwiatkowski2006}
Tadeusz Januszkiewicz and Jacek {\'S}wi{\c{a}}tkowski.
\newblock Simplicial nonpositive curvature.
\newblock {\em Publ. Math. Inst. Hautes \'Etudes Sci.}, 104:1--85, 2006.

\bibitem[Juh23]{juhasz2023class}
Arye Juh{\'a}sz.
\newblock A class of {A}rtin groups in which the ${K}(\pi, 1)$ conjecture
  holds.
\newblock {\em Journal of Pure and Applied Algebra}, page 107401, 2023.

\bibitem[{Ope}26]{OAI26}
{OpenAI}.
\newblock Harmonic heights and the {Artin} {$K(\pi,1)$} conjecture.
\newblock OpenAI Math Release preprint, 2026.
\newblock Manuscript dated September 23, 2026; released October 6, 2026.

\bibitem[Par14]{paris2012k}
Luis Paris.
\newblock ${K}(\pi,1)$ conjecture for {A}rtin groups.
\newblock {\em Ann. Fac. Sci. Toulouse Math. (6)}, 23(2):361--415, 2014.

\bibitem[PS21]{paolini2021proof}
Giovanni Paolini and Mario Salvetti.
\newblock Proof of the ${K}(\pi,1)$ conjecture for affine {A}rtin groups.
\newblock {\em Invent. Math.}, 224(2):487--572, 2021.

\bibitem[vdL83]{lek}
Harm van~der Lek.
\newblock {\em The homotopy type of complex hyperplane complements}.
\newblock 1983.

\end{thebibliography}
\end{document}